\documentclass[11pt]{article}

\usepackage[T1]{fontenc}
\usepackage[utf8]{inputenc}
\usepackage{mathpazo}
\usepackage[a4paper,margin=3cm]{geometry}
\usepackage{amsmath}
\usepackage{amsfonts}
\usepackage{amssymb}
\usepackage{amsthm}
\usepackage{mathtools}
\usepackage{soul}
\usepackage[colorlinks=true,linkcolor=black,citecolor=black,hypertexnames=false]{hyperref}
\PassOptionsToPackage{unicode}{hyperref}
\usepackage{tikz}
\usepackage{framed,enumitem}
\usepackage{quiver}
\usepackage{graphicx}
\usepackage{pgfplots}
\usepackage{subcaption}
\usepackage{array}
\usepackage{authblk}
\usepackage[capitalise]{cleveref}
\crefname{subsection}{Subsection}{subsections}
\crefname{figure}{Figure}{Figure}
\usepackage{makecell}
\usepackage{tikz, tikz-cd}
\usetikzlibrary{patterns}
\usepackage{scrextend}

\makeatletter
\newcommand{\customlabel}[2]{%
   \protected@write \@auxout {}{\string \newlabel {#1}{{#2}{\thepage}{#2}{#1}{}} }%
   \hypertarget{#1}{#2}%
}
\makeatother

\hypersetup{
    colorlinks=true,
    linkcolor=.,  
    citecolor=.,
    urlcolor=.
}

\usepackage{sagetex}
\usepackage[most]{tcolorbox}
\usepackage{FiraMono}

\definecolor{background}{HTML}{F5F6FF}
\definecolor{border}    {HTML}{B8C0FF}
\definecolor{comment}   {HTML}{35C04A}

\newtcolorbox{sagecolorbox}{
  width=\textwidth,
  colback = background,
  colframe = border,
  arc = 3px,
  boxrule = 1pt,
  fontupper = \scriptsize\ttfamily,
  left  = 0.15cm,
  right = 0.15cm,
  top   = -0.2cm,
  bottom= -0.2cm,
  before skip = 0.4cm,
  after skip  = 0.4cm,
  width = \linewidth,
  center
}

\newtheorem{theorem}{Theorem}[section]
\newtheorem*{theorem*}{Theorem}
\newtheorem{proposition}[theorem]{Proposition}
\newtheorem*{proposition*}{Proposition}

\newtheorem*{corollary*}{Corollary}

\newtheorem*{lemma*}{Lemma}

\newtheorem*{conj*}{Conjecture}

\theoremstyle{definition}
\newtheorem{definition}[theorem]{Definition}
\newtheorem{remark}[theorem]{Remark}
\newtheorem{example}[theorem]{Example}

\newcommand{\Function}[5]{
  \[
    \begin{array}{rrcl}
    #1: & #2 & \to     & #3 \\
        & #4 & \mapsto & #5
    \end{array}
  \]
}

\newcommand{\FunctionNoname}[4]{
  \[
    \begin{array}{rcl}
      #1 & \to     & #2 \\
      #3 & \mapsto & #4
    \end{array}
  \]
}

\newcommand{\EqFunction}[6]{
  \begin{equation}
    \label{#6}
    \begin{array}{rrcl}
    #1: & #2 & \to     & #3 \\
        & #4 & \mapsto & #5
    \end{array}
  \end{equation}
}

\renewcommand{\tilde}{\widetilde}
\renewcommand\epsilon{\varepsilon}
\renewcommand\geq{\geqslant}
\renewcommand\leq{\leqslant}

\newcommand\N{\mathbb{N}}
\newcommand\C{\mathbb{C}}
\newcommand\Q{\mathbb{Q}}
\newcommand\F{\mathbb{F}}
\newcommand\Fq{\mathbb{F}_q}

\newcommand\Fqm{\F_{q^m}}

\newcommand\R{\mathbb{R}}
\newcommand\Z{\mathbb{Z}}
\newcommand{\Hom}{\text{\rm Hom}}
\newcommand\End{\text{\rm End}}

\newcommand\Acar{\text{\rm char}_A}

\newcommand\rk{\text{\rm rk}}

\newcommand{\p}{\mathfrak{p}}
\newcommand{\pp}{\mathfrak{p}}

\newcommand{\Fpp}{\F_{\pp}}

\renewcommand{\H}{\mathcal H}
\DeclareMathOperator{\Pic}{Pic}
\DeclareMathOperator{\Cl}{Cl}

\newcommand\calX{\mathcal{X}}

\newcommand\bk{\boldsymbol{k}}
\newcommand\bs{\boldsymbol{s}}

\newcommand{\Kinf}{K_{\infty}}
\newcommand{\Cinf}{C_{\infty}}

\newcommand{\MM}{\mathbf{M}}
\newcommand{\Tl}{\mathbf{T}_{\mathfrak l}}
\newcommand{\TT}{\mathbf{T}}

\newcommand{\ee}{\mathbf{e}}
\renewcommand{\a}{\mathfrak{a}}
\newcommand{\fll}{\mathfrak{l}}
\newcommand{\Fs}{F^{\text{\rm s}}}
\newcommand{\Ks}{K^{\text{\rm s}}}

\newcommand{\Mat}{\text{\rm Mat}}

\newcommand{\Gal}{\text{\rm Gal}}

\newcommand{\BC}{\text{\rm BC}}
\DeclareMathOperator{\GL}{\mathrm{GL}}
\DeclareMathOperator{\enc}{\mathrm{enc}}

\newcommand{\phiact}[1]{{}^\phi\! #1}
\newcommand{\psiact}[1]{{}^\psi\! #1}

\newcommand{\Ftau}{F\{\tau\}}

\newcommand{\rgcd}{\mathrm{rgcd}}
\newcommand{\llcm}{\mathrm{llcm}}

\newcommand{\ie}{\textit{i.e.,~}}
\newcommand{\eg}{\textit{e.g.,~}}

\DeclareMathOperator{\h}{h}

\newcounter{sageexplaincounter}[section]
\renewcommand{\thesageexplaincounter}{\thesection.\arabic{sageexplaincounter}}

\newenvironment{sageexplain}
{
  \vskip 1em
  \begin{addmargin}[0.05\linewidth]{0.05\linewidth}
  \footnotesize
  \color{black!70}
  \parindent=0pt
  \refstepcounter{sageexplaincounter}%
  \textbf{SageMath example~\thesageexplaincounter.}
}
{
  \end{addmargin}
  \vskip 1em
}
\newcolumntype{C}[1]{>{\centering\hspace{0pt}\arraybackslash}m{#1}} 
\newcolumntype{L}[1]{>{\raggedright\hspace{0pt}\arraybackslash}m{#1}} 

\usepackage{float} 

\usepackage[normalem]{ulem}

\title{A computational approach to Drinfeld modules}

\author[1]{Cécile Armana}
\author[2]{Elena Berardini}
\author[2]{Xavier Caruso}
\author[3]{Antoine Leudière}
\author[4]{Jade Nardi}
\author[5]{Fabien Pazuki}
\affil[1]{Univ. Lille, CNRS, UMR 8524 - Laboratoire Paul Painlevé, F-59000 Lille,
France}
\affil[2]{CNRS; IMB, Université de Bordeaux, France}  
\affil[3]{University of Calgary, Canada}
\affil[4]{CNRS; Univ Rennes, IRMAR - UMR 6625, F-35000 Rennes, France} 
\affil[5]{Department of Mathematical Sciences, University of Copenhagen, Universitetsparken 5,
    2100 Copenhagen, Denmark}

\date{}
\makeatletter
\newcommand{\subjclass}[1]{
  \gdef\@subjclass{#1}
}
\newcommand{\keywords}[1]{
  \gdef\@keywords{#1}
}
\newcommand{\printclassifications}{
  \ifx\@subjclass\@empty\else
    \begingroup
      \renewcommand\thefootnote{}
      \footnotetext{\textit{2020 Mathematics Subject Classification.} \@subjclass}
    \endgroup
  \fi
  \ifx\@keywords\@empty\else
    \begingroup
      \renewcommand\thefootnote{}
      \footnotetext{\textit{Keywords:} \@keywords}
    \endgroup
  \fi
}
\makeatother

\begin{document}
\sloppy
	\maketitle

	\begin{abstract} 
This survey provides a practical and algorithmic perspective on Drinfeld modules over $\Fq[T]$. Starting with the construction of the Carlitz module, we present Drinfeld modules in any rank and some of their arithmetic properties. We emphasise the analogies with elliptic curves, and in the meantime, we also highlight key differences such as their rank structure and their associated Anderson motives.

This document is designed for researchers in number theory, arithmetic geometry, algorithmic number theory, cryptography, or computer algebra, offering tools and insights to navigate the computational aspects of Drinfeld modules effectively. We include detailed SageMath implementations to illustrate explicit computations and facilitate experimentation. Applications to polynomial factorisation, isogeny computations, cryptographic constructions, and coding theory are also presented. 
\end{abstract}


\printclassifications

\setcounter{tocdepth}{2}
\tableofcontents

\section{Introduction} \label{sec:motivations}

Number theory is sometimes depicted as the study of \emph{number fields}, which are finite extensions of $\mathbb{Q}$, together with their Galois properties. 
The simplest class of number fields are the so-called \emph{cyclotomic fields} obtained from $\mathbb{Q}$ by adjoining roots of unity. 
Although quite elementary in appearance, they revealed a remarkable structure and numerous applications, including Kummer’s proof of many cases of Fermat’s Last Theorem.
Ultimately, their investigation leads to class field theory.
Beyond cyclotomic extensions, one finds nonabelian extensions of $\Q$.
Here, the situation is far less understood; nevertheless, algebraic geometry provides powerful tools for building such extensions.
As a basic example, the field generated by the coordinates of the $n$-torsion points of an elliptic curve defined over $\Q$ is a number field whose Galois group naturally sits in $\GL_2(\Z/n\Z)$.
Studying these extensions proved to be fascinating and again led to outstanding applications, including the complete proof of Fermat's Last Theorem due to Wiles and Taylor--Wiles.
Nowadays, these developments are encompassed in the far-reaching Langlands programme, which, roughly speaking, aims at understanding all number fields by group-theoretical means.

In parallel with number fields, one often considers function fields, which are finite extensions of $\Fq(T)$ where $\Fq$ is a finite field.
However, extending the definitions of cyclotomic fields and elliptic curves to function fields in a straightforward way does not lead to notions exhibiting sufficiently rich arithmetic content. Regarding cyclotomy, for example, we note that the extensions of $\Fq(T)$ generated by roots of unity are simply the fields $\Fqm(T)$---which leaves out many interesting abelian function fields.

In the 1930s, Carlitz~\cite{C35} introduced an analogue of the exponential function in the framework of function fields, resulting in the construction of a new large family of cyclotomic extensions. Roughly speaking, these extensions are defined by adjoining ``$f(T)$-th roots of unity'' for any \emph{polynomial} $f(T) \in \Fq[T]$.
About forty years later, Drinfeld went further and proposed new objects, first called \emph{elliptic modules}, as a meaningful arithmetic replacement of elliptic curves over function fields.
Elliptic modules, which are nowadays called \emph{Drinfeld modules}, allow for building extensions with Galois groups sitting in $\GL_r(\Fq[T]/f(T)\Fq[T])$ for any rank~$r$ and any polynomial $f(T) \in \Fq[T]$.
In rank~$1$, Drinfeld's theory meets Carlitz's constructions, while in rank~$2$, it mimics the classical theory of elliptic curves.
However, Drinfeld's framework allows one to explore higher ranks similarly---this feature turned out to be of remarkable importance towards the Langlands program for function fields~\cite{drinfeld1980langlands}, as it allowed Lafforgue~\cite{lafforgue2002chtoucas} to completely establish it in the case of $\GL_r$ in 2002.

We also rapidly mention that, the same way that elliptic curves are related to modular forms, Drinfeld modules are also related to the so-called Drinfeld modular forms \emph{via} modularity theorems~\cite{drinfel1974elliptic,gekeler-reversat-96}.
They are also the source of a rich theory of transcendence over function fields, including algebraic independence \cite[Chapter~10]{Tha04}.

Algorithms for elliptic curves have been studied for a long time.
On the contrary, despite the remarkable impacts of Drinfeld modules, their algorithmic aspects are less known and only a subject of very recent developments~\cite{caranay_computing_2018,caranay2020computing,W22,ayotte2023arithmetic,CL23, leudiere_morphisms_2024, caruso_computation_2025, gloch_algorithms_2025}.
This research was initially driven by the profusion of shared properties with elliptic curves, and more generally, abelian varieties, over number fields, but it has now started to gain its independence and even to go further in certain directions.
As a striking example, the problems of computing isogenies or $L$-functions are basically solved for Drinfeld modules, whereas they are still challenging questions over number fields.
Moreover, algorithms for Drinfeld modules also start to find applications in connected domains, including computer algebra~\cite{doliskani_drinfeld_2021}, cryptography~\cite{scanlon_public_2001,joux_drinfeld_2019,leudiere_hard_2022} and coding theory~\cite{bassa2015good,BDM24preprint}.

\subsection{Purpose and organisation of the survey}

This document is designed as a general introduction to the theory of Drinfeld modules, with a particular look towards computational aspects and applications.
Our presentation will be illustrated by many examples handled with the help of the software SageMath, which includes an implementation of Drinfeld modules: the first features were integrated in SageMath 10.0~\cite{ayotte_drinfeld_2023, leudiere_morphisms_2024}, while the more recent ones were added at the same time as writing the present article.

In practice, our text will be interspersed with SageMath snippets as follows.

\begin{sageexplain}
\begin{sagecolorbox}
\begin{sagecommandline}

  sage: # This is a comment in the SageMath interpretor
  sage: # The following is a command and its output
  sage: 57.is_prime()
  sage: # Perfection.

\end{sagecommandline}
\end{sagecolorbox}
\end{sageexplain}

We start in~\cref{sec:carlitz} with Carlitz's \emph{analytic construction} of the so-called \emph{Carlitz module}, and present its applications to cyclotomy.
In~\cref{sec:drinfeld}, we climb the ladder of ranks and define general Drinfeld modules, moving moreover from an analytic treatment to an algebraic one.
There, we also introduce classical algebraic invariants attached to Drinfeld
modules, namely their \emph{Tate modules} and their ``motives'', called
\emph{Anderson motives}, the latter being important algorithmic assets.

\cref{ssec:drinfeld:morphisms} is dedicated to morphisms between Drinfeld modules, a central topic of the theory; we will particularly study the action they induce at the level of Anderson motives and, building on this, will associate meaningful invariants to them.
Our goal is to equip the reader with the necessary tools and insights to navigate these structures.
In \cref{sec:arithmetics}, we continue drawing parallels between Drinfeld modules and elliptic curves, regarding their arithmetic aspects.
We draw in particular a picture of Drinfeld modules over finite fields, underlying the importance of the Frobenius isomorphism, and over function fields, giving an overview of the theory of $L$-series.
We also explore some recent developments related to height theory for Drinfeld modules, a topic which is somewhat less treated in standard references.
Finally, \cref{sec:applications} gives an overview of the applications of Drinfeld modules to polynomial factorisation, cryptography and coding theory.

Although the material presented in the survey is somewhat standard, our presentation deviates from the classical ones by being clearly algorithm-oriented.
In particular, we give a prominent role to Anderson motives, as they provide a concrete incarnation of Drinfeld modules with which it is easier to handle computations.
Indeed, unlike Drinfeld modules, Anderson motives are actually modules in the
classical sense, allowing explicit calculations \emph{via} standard linear
algebra methods and polynomial arithmetic.
While emphasising similarities between Drinfeld modules and elliptic curves, this point of view also allows one to touch upon some of their key differences:
Anderson motives embody Grothendieck motives, but unlike them, can
be described using only elementary algebraic definitions. They thus played a key
role in recent algorithmic developments, allowing for solving computational
problems in greater generality (higher rank, notably), and being used in new
coding theory constructions.

For a more comprehensive treatment of Drinfeld modules and their
generalisations (Anderson motives, abelian modules, shtukas), with complementary perspectives, we refer the reader
to~\cite{G98,Rosen_2002,Tha04,brownawell_rapid_2020, poonen_introduction_2022, Pap23} (in chronological order, this list not being exhaustive).

\subsection{Setting and notation}
\label{subsection:setting-and-notations}

Let $p$ be a prime number and $q$ be a power of $p$. Let $\Fq$ be a finite field with $q$ elements. We denote by $A:= \Fq[T]$ the ring of univariate polynomials with coefficients in $\Fq$ and $K=\Fq(T)$ its field of fractions. The field $K$ is a global function field over $\Fq$. 

\begin{remark}
In full generality, Drinfeld modules are defined for $K$ being the field of rational functions over a smooth curve defined over a finite field.
Nonetheless, throughout this text, we will restrict ourselves to the case where $A = \Fq[T]$ (corresponding to the curve $\mathbb P^1$), considering that the theory is already rich enough for our exposition. In order to facilitate the study of Drinfeld modules over general rings, we point out parts where the assumption $A=\Fq[T]$ is crucial.
\end{remark}

The place of $K$ corresponding to $\frac{1}{T}$ is denoted by $\infty$. We set $\Kinf := \Fq((\frac{1}{T}))$, which is a field isomorphic to the completion of $K$ at $\infty$. We equip $\Kinf$ with the absolute value $q^{\deg(\cdot)}$. The completion of a fixed algebraic closure of $\Kinf$ is denoted by $\Cinf$, and the absolute value uniquely extends to $\Cinf$. Both fields $\Kinf$ and $\Cinf$ are non-archimedean. The field $\Cinf$ is complete and algebraically closed. 

A set of classical analogies between the function field and the number field settings, in which $\Z$ and $A$ play the same role, is given in \cref{fig-settingsanalogies}. However, contrary to $\C/\R$, the extension $\Cinf/\Kinf$ has infinite degree. The elements of $\Cinf$ may be difficult to apprehend, especially for computational applications. Note that for a field $K$ of characteristic $0$, the algebraic closure of $K((t))$ is isomorphic to the Newton--Puiseux field $\bigcup_{i=1}^\infty K((t^{1/i}))$. In our case, as $\Kinf$ has positive characteristic, its algebraic closure contains more than Newton--Puiseux series---we refer to \cite{Kedlaya01} for more details.

\begin{figure}
	\begin{tikzpicture}[scale=1.5]
		\node (Z) at (0,0) {$\mathbb{Z}$};
		\node (Q) at (0,1) {$\mathbb{Q}$};
		
		\node (R) at (-1,2) {$\mathbb{R}$};
		\node (C) at (-1,3) {$\mathbb{C}$};
		\node[align=center] (val1) at (-1,4) {complex \\ modulus};

		\node (Ql) at (1,2) {$\Q_\ell$};
		\node (Cl) at (1,3) {$\C_\ell$};
		\node[align=center] (val2) at (1,4) {$\ell^{-v_\ell(\cdot)}$};

		\node at ($(Z)!0.5!(Q)$) {\rotatebox{90}{$\subset$}};
		\draw (Q) -- (R);
		\draw (R) -- node [left] {2} (C) ;
		
		\draw (Q) -- (Ql);
		\draw (Ql) -- node [right] {} (Cl) ;
		
		\draw[<->] (C) -- node [above] {$\sim$} node [below] {as fields} (Cl) ;
	
	\begin{scope}[xshift=3.5cm]
		\node (A) at (0,0) {$A=\Fq[T]$};
		\node (K) at (0,1) {$K=\Fq(T)$};
		\node (Kinf) at (0,2) {$\Kinf=\Fq((\frac{1}{T}))$};
		\node (Cinf) at (0,3) {$\Cinf$};
		\node (val3) at (0,4) {$q^{\deg(\cdot)}$};
	\end{scope}
	
	\node at ($(A)!0.5!(K)$) {\rotatebox{90}{$\subset$}};
	\draw (K) -- (Kinf);
	\draw (Kinf) -- node [right] {$\infty$} (Cinf) ;

	\begin{scope}[xshift=-4cm]
\node (ring) at (0,0) {Ring};
\node[align=center] (ff) at (0,1) {Field of \\fractions};
\node[align=center] (comp) at (0,2) {Completion};
\node[align=center] (alg) at (0,3) {Completion of an\\algebraic closure};
\node[align=center] (val0) at (0,4) {Absolute value};
\end{scope}

\node (accol) at (3,4.3) {};
\draw[decorate,decoration={brace,amplitude=10pt}] (accol.north -| val2.west) -- (accol.north -| val3.east) node[midway,above=10pt] {non--archimedean};

	\end{tikzpicture}
	\caption{An analogy to keep in mind when starting to work with Drinfeld modules.}
\label{fig-settingsanalogies}
\end{figure}
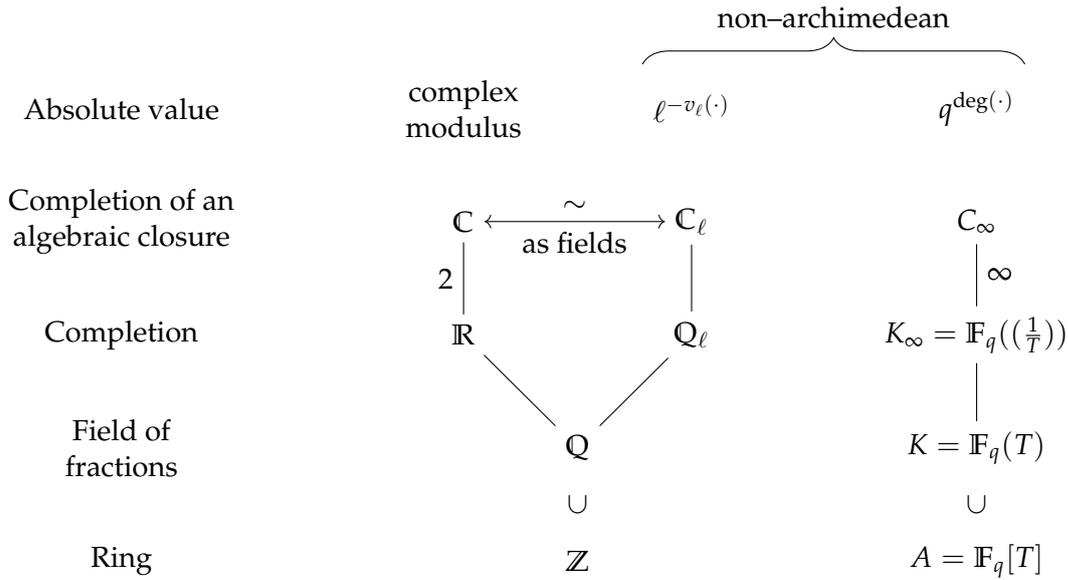

\begin{sageexplain}
Throughout this article, we shall use the base ring $A = \F_7[T]$
for all our examples in SageMath. The code below creates $A$ and
its completion at infinity $\Kinf$.

\begin{sagecolorbox}
\begin{sagecommandline}

  sage: F7 = GF(7)
  sage: A.<T> = F7['T']
  sage: A.completion(infinity)

\end{sagecommandline}
\end{sagecolorbox}
\end{sageexplain}

\paragraph{Euler-Poincaré characteristic.}

We also recall the following classification result over the principal ideal domain $A$:
any finitely generated $A$-module $M$ is isomorphic to a direct sum
of the form
\[
  M \simeq A^n \oplus A/a_1A \oplus \cdots \oplus A/a_mA
\]
where $n$ is a nonnegative integer, often referred to as the \emph{rank}
of $M$, and the $a_i$ lie in $A$. Besides, $n$ is uniquely determined, and the $a_i$ are too, if we impose the condition that
they are monic and that $a_i$ divides $a_{i+1}$ for all $i$.

In particular, if $M$ is finite then $n$ is necessarily zero and $M$ 
becomes isomorphic to $A/a_1A \oplus \cdots \oplus A/a_mA$. The ideal 
generated by the product $a_1 \cdots a_m$ does not depend on this 
presentation: it is the so-called \emph{Euler--Poincaré characteristic} 
or the \emph{Fitting ideal} of~$M$ \cite[Chapter 3, 
\S~8]{lang_algebra_2002}.
We will denote it by $|M|$ throughout this article and use it to measure the size of $M$. In analogy with the number field setting, it simply corresponds to the cardinality.
	
\section{Discovering the Carlitz module}

\label{sec:carlitz}

The first instance of a Drinfeld module was introduced by Leonard Carlitz in 1935 \cite{C35}. At the time, his work received little attention and was mostly forgotten until it was rediscovered in the 1970s, following the work of Drinfeld \cite{drinfeld_commutative_1977}. Carlitz's construction is now called the \emph{Carlitz module}. We motivate its introduction with the \emph{Carlitz exponential} in characteristic~$p$, by building on analogies and differences with classical exponential functions in characteristic~$0$. The following presentation is largely inspired by \cite[\S~2.1]{Tha04}, with additional details.
 
\subsection{Looking for an exponential function}
\label{ssec:lookforexp}

The starting point of the theory for the Carlitz module is the desire to find an analogue of the cyclotomic theory in function fields. Recall that, in classical number theory, the cyclotomic extensions of $\Q$ are the extensions of the form $\Q(\zeta_n)$ where $\zeta_n$ is a primitive $n$-th root of unity, with $n\geq2$ an integer. There are several ways to think of $\zeta_n$. One of them is of an analytic nature, through the classical exponential function $\exp: \C\to\C$: 
$$\zeta_n = \exp\left(\frac{2i\pi}n\right).$$

An approach to construct relevant cyclotomic extensions of $K=\Fq(T)$ is to generalise the above construction. We are then looking for a function $e: \Cinf \rightarrow \Cinf$ which would play the role of an exponential function in our setting. This will also lead us to define an analogue of $\pi$.

A first way to define the classical exponential function $\exp$ is through its differential equation $\exp' = \exp$. Power series solutions 
to this equation would necessarily be of the form $c\sum_{n\geq0} \frac{x^n}{n!}$ for some constant $c$. However, this does not make sense in $\Cinf$ 
since~$n! = 0$ when $n\geq p$.
A second way to look at the exponential is \emph{via} its functional equation 
$$\exp(x+y)=\exp(x)\exp(y).$$
Unfortunately, this fails again. Indeed, any function $e$ defined over $\Cinf$ satisfying 
the above functional equation would also satisfy $e(0)=e(0)^2$, and therefore $e(0) \in\{0,1\}$. 
Since $e(0)=e(px)=e(x)^p$ because $\Cinf$ has characteristic~$p$, 
we would finally derive that $e$ is identically zero or identically one.

The idea, which turns out to work, originally developed by Carlitz, is to take advantage of the existence of additive functions in characteristic~$p$ and to modify the functional equation by replacing the product on the right-hand side by a sum: we are now
looking for functions $e : \Cinf \to \Cinf$ satisfying
$$\forall x, y \in \Cinf, \quad e(x+y) = e(x) + e(y),$$
\ie simply additive functions. Since our setting is by nature 
not only additive but $\Fq$-linear (everything is defined over
$\Fq$), we will focus more specifically on \emph{$\Fq$-linear} 
functions $e$.

As in the case of the classical exponential function, one expects the function $e$ to be \emph{entire}, that is, given by a power series which converges everywhere on $\Cinf$. 
A standard result of ultrametric analysis, recalled in \cite[Proposition~2.7.12]{Pap23}, states that any entire and non-constant function $f:\Cinf \to \Cinf$ is surjective. Thus, our exponential function $e$ would also be surjective.  
Let $\Lambda$ be the set of its zeros. It is a discrete $\Fq$-linear subspace of 
$\Cinf$, meaning that the intersection of $\Lambda$ with any closed ball of positive radius in $\Cinf$ is finite. If such a function $e$ existed, we would get an exact sequence of $\Fq$-vector spaces of the form

\begin{equation}\label{eq:exact-seq-inspi}
0 \longrightarrow \Lambda \longrightarrow \Cinf \overset{e}{\longrightarrow} \Cinf \longrightarrow 0,
\end{equation}
which is to be compared to the exact sequence of $\Z$-modules
for the classical exponential function
\begin{equation}\label{eq:exact-seq-classicalexp}0 \longrightarrow 2i\pi\Z \longrightarrow \C \overset{\exp}{\longrightarrow} \C^\times \longrightarrow 0.\end{equation}
Comparing those two exact sequences provides additional hints of what $\Lambda$ should be. Indeed, recall that
the analogue of $\Z$ in our setting is the ring $A = \Fq[T]$. We
thus aim at expressing $\Lambda$ as a rank-$1$ $A$-lattice $\Lambda = \tilde\pi A$ generated by some constant $\tilde\pi$, that we still need to determine. 
Furthermore, by analogy, we would like \eqref{eq:exact-seq-inspi} to be an exact sequence of \emph{$A$-modules}, for some $A$-module structures to be defined on $\Lambda$ and on both copies of $\Cinf$. First, we need to interpret $e:\Cinf \to \Cinf$ as a homomorphism of $A$-modules. We already know that $e(x+y)=e(x)+e(y)$ for any $x$ and $y$ in $\Cinf$. However, we cannot have $e(az)=a e(z)$ for $a\in A$, as the functions $z \mapsto e(az)$ and $z \mapsto a e(z)$ do not have the same set of zeros. 
The idea of Carlitz, which is central to the theory of Drinfeld modules, is to equip the copy of $\Cinf$ in the codomain with a new structure of $A$-module compatible with $e$.

\begin{remark}
Before proceeding, we underline the similarity with the classical exponential function,
see Equation~\eqref{eq:exact-seq-classicalexp}: 
the codomain of the latter is not $\C$, but $\C^\times$ equipped with its 
multiplicative structure. In the function field setting, changing the 
structure on the codomain is also required, but the modification is 
of a different nature: we keep the additive structure but modify the action of $A$.
\end{remark}

Concretely, for any polynomial $a\in A$, we look for a functional
equation of the form
\begin{equation}\label{eq:functeqcarlitzexp}\forall z \in \Cinf,\quad e(az) = \phi_a (e(z))\end{equation}
where $\phi_a : \Cinf \to \Cinf$ is a function (dependent on $a$) to 
determine, which will define a new structure of $A$-module
on $C_\infty$, denoted  by $\phiact{\Cinf}$, through the rule
$$\forall a \in A, \forall z \in \Cinf, \quad a \star z = \phi_a(z).$$

We will thus get a family of commutative diagrams of $A$-modules:
\begin{center}
\begin{tikzcd}
	0 \arrow[r] & \Lambda \arrow[d,"\text{mult}_a"] \arrow[r] & \Cinf \arrow[d,"\text{mult}_a"] \arrow[r,"e"] &\phiact{\Cinf}\arrow[d,"\phi_a"] \arrow[r] &0 \\
	0 \arrow[r] & \Lambda \arrow[r] & \Cinf \arrow[r,"e"] &\phiact{\Cinf}\arrow[r] &0
\end{tikzcd}
\end{center}
where $\text{mult}_a$ denotes the usual multiplication by $a \in A$.

\subsection{Analytic construction of the Carlitz module}\label{ssec:carlitz}

\subsubsection{Finding a formula for the function \texorpdfstring{$e$}{e}}

We now aim at constructing the function $e$ with the properties highlighted in
\cref{ssec:lookforexp}. Namely, we fix a rank-$1$ $A$-lattice of the form $\Lambda  = cA$, with $c \in \Cinf^\times$, and we are looking for an $\Fq$-linear entire function $e_\Lambda: \Cinf \to \Cinf$ whose kernel is the lattice $\Lambda$.

To build $e_\Lambda$, we recall another standard result of ultrametric analysis: any
entire function $f : \Cinf \to \Cinf$ is determined, up to multiplication by a
nonzero constant, by its set of zeros counted with multiplicity; this is a 
consequence  of an ultrametric version of the Weierstraß ``preparation theorem'' (see~\cite[\S~2.7.2]{Pap23} for more details).

In our case, we end up with the expression
\begin{equation}\label{eq-defe}
e_\Lambda(z) = 
z \prod_{\lambda \in \Lambda \setminus\{0\}} \left(1-\frac{z}{\lambda}\right).
\end{equation}
The fact that the above product indeed converges to an $\Fq$-linear entire function 
$e_\Lambda$ is a consequence of the discreteness of $\Lambda$ in $\Cinf$ 
\cite[Proposition~5.1.3]{Pap23}.

\subsubsection{Setting a new structure of \texorpdfstring{$A$}{e}-module on \texorpdfstring{$\Cinf$}{C\_∞}}
We now explain the construction of the morphisms $\phi_a$ in the functional Equation~\eqref{eq:functeqcarlitzexp}.
First of all, we observe that since $A=\Fq[T]$ is generated by $T$ over $\Fq$, knowing $\phi_T$ for the indeterminate 
$T$ is enough to recover all the functions $\phi_a$. Indeed, we have, for instance,
$$ \phi_{T^2}(e(z)) = e(T^2 z) = e(T(Tz)) = \phi_T(e(Tz)) = (\phi_T \circ \phi_T)(e(z)),$$
which would give $\phi_{T^2}= \phi_T \circ \phi_T$ by identification.
More generally, for any integer $m\geq1$, we get 
$$\phi_{T^m} =\underbrace{\phi_T \circ \cdots \circ \phi_T}_{m \text{ times}}.$$
Using the $\Fq$-linearity of $e_\Lambda$, 
the function $\phi_{a}$ is obtained by taking the suitable $\Fq$-linear
combinations of the $\phi_{T^m}$.

We now focus on the construction of $\phi_T$.
We know that the map $z \mapsto e_\Lambda(Tz)$ defines an entire function 
whose set of zeros is $\frac{1}{T}\Lambda$, and that all zeros are simple. Let us 
look at the map
$$z \mapsto \prod_{\lambda \in \frac{1}{T} \Lambda / \Lambda} (e_\Lambda(z)-e_\Lambda(\lambda)).$$
Recall that $e_\Lambda$ has $\Lambda$ for kernel so $e_\Lambda(\lambda)$ is well defined for $\lambda \in \frac{1}{T} \Lambda / \Lambda$. 
Observing that $\text{Card}\big(\frac{1}{T} \Lambda / \Lambda\big) = q$, we see that the above product is finite and defines a polynomial in $e_\Lambda(z)$ of degree~$q$.
As $e_\Lambda$ is entire, it also defines an entire function.
Moreover, its set of zeros is $\frac{1}{T}\Lambda$ and all its zeros are 
simple. Therefore, there exists $k\in\Cinf^\times$ such that
\begin{equation}
\label{eq-eTz}
e_\Lambda(Tz)= k\displaystyle\prod_{\lambda \in \frac 1 T  \Lambda / \Lambda} (e_\Lambda(z)-e_\Lambda(\lambda)).
\end{equation}
Given that we aim at obtaining $\phi_T(e_\Lambda(z)) = e_\Lambda(Tz)$, we then define
\begin{equation}
\label{eq-phiT}
\phi_T(x) = k \displaystyle\prod_{\lambda \in \frac 1 T  \Lambda / \Lambda} (x-e_\Lambda(\lambda))\in \Cinf[x].
\end{equation}
The $A$-module structure on $C_\infty$ defined by the rule
\begin{equation}
\label{eq-phiCinf}
\forall a \in A, \forall z \in \Cinf, \quad a \star z = \phi_a(z)
\end{equation}
will be denoted by $\phiact\Cinf$, as announced previously.

Since the set of zeros $\frac{1}{T} \Lambda / \Lambda$ is a 
one-dimensional $\Fq$-linear subspace of $\Cinf$, we deduce that
$\phi_T(x)$ is an $\Fq$-\emph{linear} polynomial of degree $q$, 
\ie it has the form $\phi_T(x) = k_1 x + k_2 x^q$
for some $k_1, k_2 \in \Cinf$ with $k_2\neq 0$. 
From Equation~\eqref{eq-defe}, we derive the estimation $e(z) 
= z + O(z^2)$. By identifying the coefficients in $z$ in the functional 
equation $e_\Lambda(Tz) = \phi_T(e_\Lambda(z))$, we then get $k_1 = T$.
Similarly, by identifying the coefficients in $x^q$ in Equation~\eqref{eq-phiT},
we find $k_2 = k$. Therefore, we end up with the simple formula
\[\phi_T(x)=T x + k x^q.\]

\subsubsection{Normalising the lattice \texorpdfstring{$\Lambda$}{Lambda}}

We still do not know what the constant $k$ is, but it turns out that we can freely choose its value by rescaling the lattice $\Lambda$. Indeed, if instead of starting with $\Lambda$, we start with $u\Lambda$ for $u\in\Cinf^\times$, we get the exponential function $e_{u\Lambda}$, which is connected to the original one
$e_{\Lambda}$ by the relation \[e_{u\Lambda}(z) = u e_\Lambda (z/u).\] 
Indeed, on both sides, the functions have the same set of zeros with multiplicities
and the same first term $z$ in their expansion. The final formula we get for the polynomial
$\phi_T^{u\Lambda}$ attached to the lattice $u\Lambda$ reads
\[\phi_T^{u\Lambda}(x)=T x + u^{1-q} k x^q.\]
The standard choice is to normalise our exponential map, to finally obtain the normalised form $\phi_T(x)=T x + x^q$.

The corresponding lattice is usually denoted by $\tilde\pi A$; the constant $\tilde\pi$ is then understood as the function field analogue of $2i\pi$ (see \cite[p.~317]{Pap23} for instance). 
With this normalization, we get the \emph{Carlitz exponential} map, denoted by $e_C$, whose kernel is $\tilde{\pi} A$ and with functional equation 
\begin{equation}
\label{eq:equationeC}
\forall z \in \Cinf, \quad
  e_C(Tz) = \phi_T(e_C(z)) = Te_C(z) + e_C(z)^q.
\end{equation}
Another advantage of this choice is that the coefficients of $e$ are all rational functions, \ie elements of $K$.
They can be computed iteratively by using the functional equation~\eqref{eq:equationeC} (see Subsection~\ref{ssec:drinfeld:uniformization} for more details).

\begin{sageexplain}
SageMath provides direct functionalities for computing the Carlitz exponential.

\begin{sagecolorbox}
\begin{sagecommandline}

  sage: eC = carlitz_exponential(A, prec=100)
  sage: eC

\end{sagecommandline}
\end{sagecolorbox}

We can now check the functional equation:

\begin{sagecolorbox}
\begin{sagecommandline}

  sage: z = parent(eC).gen()
  sage: eC(T*z) == T*eC + eC^7
  True

\end{sagecommandline}
\end{sagecolorbox}

\end{sageexplain}

\noindent
The datum of the degree-$q$ polynomial $\phi_T (x)= Tx + x^q$ 
and all the subsequent functions $\phi_a$ forms the so-called \emph{Carlitz module}.

\begin{sageexplain}
We instantiate the Carlitz module over $\F_7$ and compute $\phi_a$ for
$a = T^2 + T + 1$.

\begin{sagecolorbox}
\begin{sagecommandline}

  sage: phi = CarlitzModule(A)
  sage: phi(T^2 + T + 1)

\end{sagecommandline}
\end{sagecolorbox}

In the output above, $\tau$ represents the Frobenius map $x \mapsto x^q$.
This notation is of primary importance in the theory of Drinfeld modules
and will be explained in more detail in \cref{ssec:drinfeld:ore}.
\end{sageexplain}

From a work of Carlitz~\cite{C35}, one can derive the following 
\begin{equation}
\label{eq:pi}
\tilde{\pi} = T \sqrt[q-1]{-T} \prod_{i=1}^{+\infty} \left ( 1 - \frac{1}{T^{q^i -1}} \right)^{-1}.
\end{equation}
Although we will not need it, we also mention that Wade~\cite{Wade}
showed that $\widetilde{\pi}$ is transcendental over $K$.

\subsection{Application to the cyclotomic theory}
\label{applications_cyclotomic_theory}

One motivation for building the Carlitz exponential
was to find an analogue of the cyclotomic theory in the framework of
function fields. Classically, the cyclotomic fields are those obtained
by adding the $n$-th roots of unity (for an integer $n\geq2$), \ie the kernel of the 
group morphism $\C^\times \to \C^\times, x \mapsto x^n$.
In the function field setting, the analogues of these morphisms are the maps $\phi_a$ defining the Carlitz module.
We are then naturally lead to consider the extension $K(\phi[a])/K$ where, by definition,
$$\phi[a] = \big\{\,z \in \phiact\Cinf \ , \
a \star z = \phi_a(z) = 0\,\big\}.$$
It is the so-called \emph{$a$-torsion} of the Carlitz module, which is the $a$-torsion submodule of $\phiact\Cinf$, therefore an $A$-submodule of $\phiact\Cinf$.

\begin{theorem}[{Carlitz, see \cite[Proposition~12.5]{Rosen_2002}}]
For all $a \in A$, 
the extension $K(\phi[a])/K$ is Galois and its Galois group
is canonically isomorphic to $(A/aA)^\times$.
\end{theorem}

\begin{proof}[Sketch of the proof]
The construction of the map $\alpha : \Gal(K(\phi[a])/K) \to 
(A/aA)^\times$ is similar to the classical case. Let $\sigma$
be an automorphism to $K(\phi[a])$ preserving $K$.
Then $\sigma$ acts on $\phi[a]$, as the latter is the set of
roots of the polynomial $\phi_a$. We then notice that $\phi[a]$
is a free module of rank~$1$ over $A/aA$ and that $\sigma$ acts
linearly. Therefore $\sigma$ must act by multiplication by an element of $A/aA$, which needs to be invertible given that $\sigma$
is an isomorphism. We thus get a uniquely defined scalar in
$(A/aA)^\times$, which is by definition $\alpha(\sigma)$.

Proving the injectivity of $\alpha$ is routine: given that 
$K(\phi[a])$ is by definition generated over $K$ by the elements
of $\phi[a]$, a $K$-automorphism of $K(\phi[a])$ is uniquely
determined by its action on $\phi[a]$. On the contrary, the
surjectivity of $\alpha$ is more difficult and showing it amounts
more or less to studying the factorisation properties of $\phi_a$.
We refer to \cite[Theorem~12.8]{Rosen_2002} for a complete proof of this statement.
\end{proof}

In the case of number fields, a striking result about 
cyclotomic extensions is that they exhaust all abelian extensions 
of $\Q$: it is the famous Kronecker--Weber's theorem. 
Hayes~\cite{Hayes_1974} proved an analogue of this statement in the
context of function fields.
However, this case is a bit more subtle, due to the two following facts, both related to ramification.
First of all, contrary to $\Q$, the base field $K = \Fq(T)$ certainly
admits everywhere unramified extensions: these are the extensions
of the form $F(T)$ where $F$ is a finite extension of $\Fq$.
Those extensions are abelian, and they are not of the form
$K(\phi[a])$.
The second source of difficulties is that Carlitz's construction  $K(\phi[a])$ 
only produces extensions that are tamely ramified at $\infty$ \cite[Theorem 3.1]{Hayes_1974}. In some
sense, this is due to our choice of the integral ring $\Fq[T]$ inside
$\Fq(T)$. For instance, had we initially chosen $\Fq[\frac 1 T]$ instead of $\Fq[T]$, we would have ended up with a different
family of cyclotomic extensions, tamely ramified at $0$. 
In \cite{Hayes_1974}, Hayes proves that combining these two families, as well
as the previously-discussed everywhere unramified extensions, then an analogue of the Kronecker--Weber's theorem does hold.

Hayes' theorem is an outstanding outcome, which definitely demonstrates the pivotal role of the Carlitz module and of the subsequent Carlitz cyclotomic extensions in the arithmetic theory of function fields. Beyond that, Hayes' theorem is often considered the seed of an explicit version of the class group
theory in the context of function fields.

\subsection{The Carlitz zeta function}
We have seen previously that the classical exponential function has an interesting and relevant function field analogue. It turns
out that this is not an isolated example: actually, several arithmetic functions of interest also have a twin in the world of function fields (see \eg\cite[\S~5.4]{Pap23}).

As an example, we present here the function field analogue of the Riemann zeta function, the so-called \emph{Carlitz zeta function}.
It is defined by
\EqFunction
  {\zeta_C}
  {\N}
  {\Kinf}
  {s}
  {\displaystyle \sum_{\a \in A_+} \frac 1{\a^s},}
  {eq:zetaC}
where $A_+$ denotes the subset of $A = \Fq[T]$ consisting of monic
polynomials.

Similarly to the case of number fields, the Carlitz zeta function can be expressed as an Euler product
\begin{equation}
\label{eq:zetaCEuler}
\zeta_C(s) = \prod_\p \frac 1{1 - \p^{-s}},
\end{equation}
where the product runs over all irreducible monic polynomials $\p$ of $A$.
As in the classical case, navigating between the expressions~\eqref{eq:zetaC} 
and~\eqref{eq:zetaCEuler} boils down to writing down the decomposition
of each $\a \in A_+$ as the product of irreducible factors, and factorising the resulting formula.

\begin{sageexplain}
\label{sage:carlitzZeta}
The following code directly computes the value of the Carlitz zeta function.

\begin{sagecolorbox}
\begin{sagecommandline}

  sage: carlitz_zeta(A, s=1, prec=40)
  sage: carlitz_zeta(A, s=2, prec=40)
  sage: carlitz_zeta(A, s=6, prec=100)

\end{sagecommandline}
\end{sagecolorbox}
\end{sageexplain}

Carlitz proved in~\cite{C35} that, whenever $s$ is a multiple of $q-1$,
one has
\begin{equation}
\label{eq:Lseries:carlitzspecial}
\zeta_C(s) = \frac{\BC_s}{\Pi(s)} \cdot \tilde \pi^s,
\end{equation}
which appears to be an analogue of the Euler formula expressing the 
values of the Riemann zeta function at positive even numbers.
In Equation~\eqref{eq:Lseries:carlitzspecial}, the constant $\tilde \pi$ is the Carlitz period we already encountered at the end of 
\cref{ssec:carlitz}. 
The notation $\BC_s$ refers to the Bernoulli--Carlitz numbers, while 
$\Pi(s)$ is the Carlitz factorial \cite[Chapter 9]{G98}.

\begin{sageexplain}
We check Formula~\eqref{eq:Lseries:carlitzspecial} for $s = q{-}1$, 
using the following explicit congruence for $\tilde \pi^{q-1}$ derived
from Equation~\eqref{eq:pi}:
$$\tilde{\pi}^{q-1} 
 = -\frac{T^q}{{\displaystyle \prod_{i=1}^{+\infty}} \left ( 1 - \frac{1}{T^{q^i -1}} \right)^{q-1}}
 \equiv
 -\frac{T^q}{{\displaystyle \prod_{i=1}^{n-1}} \left ( 1 - \frac{1}{T^{q^i -1}} \right)^{q-1}} \pmod{T^{q+1-q^n}}.$$

\begin{sagecolorbox}
\begin{sagecommandline}

  sage: # here $q = 7$
  sage: BC6 = carlitz_bernoulli(A, 6)
  sage: Pi6 = carlitz_factorial(A, 6)
  sage: pi6 = -T^7 / prod((1 - T^(1-7^i))^6 for i in range(1, 3)) + O(1/T^100)
  sage: BC6 / Pi6 * pi6

\end{sagecommandline}
\end{sagecolorbox}
\end{sageexplain}

\section{Drinfeld modules and their Anderson motives}

\label{sec:drinfeld}

In \cref{sec:carlitz}, we have extended the construction of  the exponential function in the function field setting by building
explicitly an isomorphism $\Cinf / \tilde\pi A \to \phiact{\Cinf}$,
mimicking the standard bijection 
$\C / 2i\pi \Z \stackrel\exp\longrightarrow \C^\times$.
In the classical setting, quotienting out $\C$ by lattices of
rank~$2$ is also very fruitful as it leads to the theory of
elliptic curves. The exponential function is then replaced by the
Weierstraß $\wp$ function, inducing the map
\[\begin{array}{rcl}
    \C/\Lambda & \to  & \mathbb P^2(\C)\\
     z & \mapsto & (\wp(z):\wp'(z):1),\\
     0 & \mapsto & (0:1:0),
\end{array}\]
which identifies $\C/\Lambda$ with the complex points of an elliptic curve.
In the world of function fields, a similar construction occurs: one can consider a general lattice $\Lambda \subset \Cinf$, and the quotient $\Cinf/\Lambda$ with its natural structure of
$A$-module. Here, we require $\Lambda$ to be discrete, in order
to get a suitable topology on the quotient, even though it can be of arbitrary rank, given that, contrary to the classical case, 
the extension $\Cinf/K$ is infinite.
Nonetheless, all the constructions we carried out in 
\cref{ssec:carlitz} extend \emph{verbatim}: to each such
lattice $\Lambda$ (of any rank), one can associate an exponential map $e_\Lambda$
defined by
\Function
  {e_\Lambda}
  {\Cinf}
  {\Cinf}
  {z}
  {\displaystyle z \prod_{\lambda \in \Lambda \setminus\{0\}}
  \left(1-\frac{z}{\lambda}\right),}
which sits in the exact sequence
$$0 \longrightarrow \Lambda \longrightarrow \Cinf 
\overset{e_\Lambda}{\longrightarrow} \Cinf \longrightarrow 0.$$
Moreover, for each $a \in A$, the function $e_\Lambda$ satisfies 
the functional equation 
$$e_\Lambda(az) = \phi_a(e_\Lambda(z)),$$
where $\phi_a$ is a polynomial of the form
$\phi_a(x) = a x + c_1 x^q + \cdots + c_r x^{q^r}$ (with $c_i \in
\Cinf$ and $c_r\neq 0$) explicitly defined by
\begin{equation}
\label{eq:phia}
\phi_a(x) := a x 
  \prod_{\lambda \in (\frac 1 a  \Lambda / \Lambda)^\times} 
    \hspace{-1ex} \left(1 - \frac x {e_\Lambda(\lambda)}\right).
\end{equation}
The family $(\phi_a)_{a\in A}$ endows $\Cinf$ with a structure of $A$-module denoted by $\phiact{\Cinf}$ and given, for any $a\in A$ and $z\in\phiact{\Cinf}$, by $a \star z = \phi_{a}(z)$, as in Equation~\eqref{eq-phiCinf}. It is what we call a
\emph{Drinfeld module} of rank~$r$ over $\Cinf$. Via the exponential map $e_{\Lambda}$, we get an isomorphism of $A$-modules $\Cinf/\Lambda \simeq {}^\phi \Cinf$.

Similarly to the case of elliptic curves, a decisive advantage
of the description coming from the $\phi_a$ is its algebraicity: the $\phi_a$ are not general analytic functions, but polynomials.
Consequently, they make sense over a general base and working 
over $\Cinf$ is no longer required.
The general definition of Drinfeld modules follows from this
basic observation. 
We will write it down precisely in \cref{ssec:drinfeld:def}, after some introduction to Ore polynomials carried out
in \cref{ssec:drinfeld:ore}.
We then continue in \cref{ssec:drinfeld:uniformization} by
comparing the algebraic definition with the analytic one. Finally, in \cref{ssec:drinfeld:realizations}, we explain how to associate an Anderson module to a Drinfeld module: this will allow us to use characteristic polynomials and reduction theorems at a later stage.

\subsection{Ore polynomials}
\label{ssec:drinfeld:ore}
As mentioned above, the algebraic definition of Drinfeld modules
will be captured by the data set of polynomials $\{\phi_a\}_{a \in A}$. 
In this preliminary subsection, we introduce the ring formed by those polynomials and collect their basic properties.

\smallskip
Clearly, the functions $\phi_a$ as defined in Equation~\eqref{eq:phia} have
a distinctive form: they define $\Fq$-linear functions, which 
amounts to saying that they only involve terms of the form $c_i x^{q^i}$
with $i \in \N$.
In some sense, the $\phi_a$ are polynomials in the $q$-Frobenius endomorphism $x \mapsto x^q$. 
This property is captured by the notion of Ore polynomials, defined
below.

\begin{definition}[Ore polynomials] \label{def:ore}
Let $F$ be an $\Fq$-algebra.
We denote by $F\{\tau\}$ the noncommutative ring of \emph{Ore polynomials} 
(also known as \emph{skew polynomials}, or \emph{twisted polynomials})
\[F\{\tau\}=\left\{ \sum_{i=0}^n c_i \tau^i \,\Big|\, n \geq 0, c_i \in F\right\},\]
with the classical additive law, and the multiplication defined by $\tau c= c^q \tau$ for every $c \in F$.

Writing $f = \sum_i c_i \tau^i$, we define the \emph{$\tau$-degree} (resp.~\emph{$\tau$-valuation}) of $f$ as the largest (resp.~smallest) integer $i$ such that $c_i \neq 0$.
\end{definition}

In \cref{def:ore}, the letter $\tau$ is just a formal
variable, without further signification. However, it should of
course be understood as the Frobenius endomorphism:
Ore polynomials then correspond to polynomials in the $q$-Frobenius
as claimed previously.
The noncommutative multiplication law is also reminiscent of this 
interpretation; indeed, applying the multiplication by a scalar $c$ and 
then the $q$-Frobenius amounts to applying first the $q$-Frobenius and 
then the multiplication by $c^q$.

\begin{sageexplain}
\label{sage:orepolynomials}
In SageMath, the ring of Ore polynomials can be constructed as follows.

\begin{sagecolorbox}
\begin{sagecommandline}

  sage: F.<z> = F7.extension(2)
  sage: frob = F.frobenius_endomorphism()
  sage: Ftau.<tau> = OrePolynomialRing(F, frob)
  sage: Ftau

\end{sagecommandline}
\end{sagecolorbox}

We showcase the noncommutativity of $\Ftau$:

\begin{sagecolorbox}
\begin{sagecommandline}

  sage: z * tau
  sage: tau * z

\end{sagecommandline}
\end{sagecolorbox}

\end{sageexplain}


A nice feature of Ore polynomials is that their ring, despite being
noncommutative, is left-Euclidean with respect to the \emph{$\tau$-degree}
\cite{Ore33}. This means that right-Euclidean division can be performed,
and that left-ideals all have a unique monic generator. For a nonempty family
$S$ of Ore polynomials, one can compute their \emph{right-greatest common
divisor} (or \emph{right-gcd}, for short) $\rgcd(S)$ and their
\emph{left-least common multiple} (or \emph{left-lcm}, for short) 
$\llcm(S)$. In practice, variants of the long division and Euclidean 
algorithms can be used; complexity statements can be found in 
\cite[\S~3.1]{leudiere_computing_2024}. More advanced primitives have 
been described in \cite{caruso_new_2017, caruso_fast_2017}: they are 
based on faster algorithms for the multiplication of two Ore polynomials 
with coefficients in a finite field.

\begin{sageexplain}
  Right-gcds can be easily computed as follows:

  \begin{sagecolorbox}
  \begin{sagecommandline}

    sage: f = (1 + tau) * (z + tau)
    sage: g = (1 - tau) * (z + tau)
    sage: f.right_gcd(g)

  \end{sagecommandline}
  \end{sagecolorbox}

  We underline that having a common factor on the left does not
  imply the nontriviality of the right-gcd:

  \begin{sagecolorbox}
  \begin{sagecommandline}

    sage: f = (1 + tau) * (z + tau)
    sage: g = (1 + tau) * (z - tau)
    sage: f.right_gcd(g)

  \end{sagecommandline}
  \end{sagecolorbox}

\end{sageexplain}

\subsubsection{Kernels of Ore polynomials}
\label{subsubsection:ore-pols-kernels}

To carry on the identification of $\tau$ to the $q$-Frobenius, we interpret the Ore polynomials in $F\{\tau\}$ as actual endomorphisms, namely on a separable closure $\Fs$ of $F$: to an element $\displaystyle{f=\sum_{i=0}^n c_i \tau^i \in F\{\tau\}}$, we associate the
transformation
\FunctionNoname
  {\Fs}
  {\Fs}
  {x}
  {\displaystyle{\sum_{i=0}^n c_i x^{q^i}},}
that, in a slight abuse of notation, we continue to denote by $f$.
This transformation is $\Fq$-linear and $F$-\emph{algebraic} in the
sense that it is given by a polynomial over $F$.
One proves that any $\Fq$-linear algebraic endomorphism of $\Fs$ 
comes from a uniquely determined Ore polynomial \cite[Lemma~3.1.4]{Pap23}. 
Another nice illustration of this correspondence is given by the following proposition.

\begin{proposition}
\label{prop:Orebijection}
There is a bijection:
$$\begin{array}{rcl}
\left\{ \begin{array}{c}
  \text{\rm Ore polynomials } f \in F\{\tau\}\\
  \text{\rm with constant term } 1
\end{array}\right\} &
\longrightarrow & 
\left\{ \begin{array}{c}
  \text{\rm finite dimensional $F$-linear subspaces} \\
  V \subset \Fs \text{\rm ~stable by } \Gal(\Fs/F)
\end{array}\right\} \smallskip \\
f\phantom{replace au centre} & \longmapsto & \phantom{on replace au centre}\ker f.
\end{array}$$
\end{proposition}

We refer to \cite[\S~1.2]{G98} for a proof of the proposition.
Here, we just mention that the inverse bijection can be explicitly
described as follows. Given a subspace $V \subset \Fs$, we form the
(classical) polynomial
$$\tilde f(x) = x \prod_{v \in V\setminus\{0\}} \left(1 - \frac x v\right).$$

The stability under the Galois action ensures that $\tilde f$ has
coefficients in $F$. Besides, a calculation shows that it defines
an $\Fq$-linear function.
Thus, the classical polynomial $\tilde f \in F[x]$ can be seen as an Ore
polynomial in $F\{\tau\}$. This Ore polynomial is the inverse image of $V$
under the bijection of \cref{prop:Orebijection}.
\begin{proposition}
\label{rem:Orebijection}
The bijection of \cref{prop:Orebijection} is increasing
in the sense that $f$ right-divides $g$ if and only if $\ker f \subset \ker g$. This implies, in particular, that right-gcds (resp.~left-lcms) of
Ore polynomials on the left-hand side correspond to intersections
(resp.~sums) of subspaces on the right-hand side.
\end{proposition}

\begin{remark}
Restricting to Ore polynomials with a constant term equal to $1$ in
\cref{prop:Orebijection} is important for two reasons.
First, it permits normalising the Ore polynomials as, of course,
$f$ and $cf$ (for $c \in F^\times$) have the same kernel. Secondly, it permits discarding all Ore polynomials with vanishing constant
terms. Indeed, we observe that $f$ and $f\tau$ also share the same kernel, given that the $q$-Frobenius induces a bijection on each Galois-stable $\Fq$-linear subspace of $\Fs$.
\end{remark}

\begin{remark}
  The polynomial $\tilde f$ defined above can be expressed in
  terms of elementary symmetric functions. If we look at $\tilde f$ as a
  regular polynomial, most of its coefficients are zero; we would also want
  root-coefficients relations defining $\tilde f$ to be given as expressions in
  the sole kernel basis elements. Consequently, we introduce the following
  objects. Consider an integer $n \geq 0$, and formal variables $X_1, \dots,
  X_n, X$. The \emph{Moore determinant} of the family $(X_1, \dots, X_n)$ is
  defined as
  \[
    \Delta(X_1, \dots, X_n) = \det
    \begin{pmatrix}
      X_1 & \cdots & X_n \\
      X_1^q & \cdots & X_n^q \\
      \vdots & \vdots & \vdots \\
      X_1^{q^{n-1}} & \cdots & X_n^{q^{n-1}}
    \end{pmatrix}.
  \]
  It is a polynomial in $X_1, \dots, X_n$, and was introduced by Moore
  \cite{moore_two-fold_1896} as an $\Fq$-linear analogue of the classical
  \emph{Vandermonde determinant}. This polynomial satisfies (see \cite[Theorem~10]{ore_special_1933} or
  \cite[Proposition~1]{elkies_linearized_1999})
  \[
    \Delta(X_1, \dots, X_n) =
    \prod_{i=1}^n
      \prod_{(k_1, \dots, k_{i-1}) \in \Fq^{i-1}}
        \left(X_i + \sum_j k_j X_j\right),
  \]
  which means that a family $(x_1, \dots, x_n)$ is $\Fq$-linearly dependant, if and only if, $\Delta(x_1, \dots, x_n)=0$. 
  In particular, if $(\omega_1, \dots, \omega_n)$ is a basis of $\ker f$,
  the univariate polynomial $\Delta(x, \omega_1, \ldots, \omega_n)$ vanishes exactly 
  on $\ker f$. After renormalisation, we obtain
  \[
    \tilde f(x) = \frac {\Delta(x, \omega_1, \dots, \omega_n)}
                          {\Delta(\omega_1^q, \dots, \omega_n^q)}.
  \]
  Similarly, the unique monic separable polynomial whose set of zeros is $\ker f$
  is given by
  \[
    \bar{f}(x) := \frac {\Delta(x, \omega_1, \dots, \omega_n)}
                          {\Delta(\omega_1, \dots, \omega_n)}.
  \]
  Expanding the determinants, one proves the existence of universal
  polynomials $\lambda^{(n)}_0, \dots, \lambda^{(n)}_n \in
  \Fq[X_1, \dots, X_n]$ such that $\bar f(x) = \sum_{i=0}^n \lambda^{(n)}_i x^{q^i}$.
  Each $\lambda^{(n)}_i$ is homogeneous of degree $q^n - q^i$.
  Altogether, they form the set of $\Fq$-linear elementary symmetric functions.
\end{remark}

\subsection{Algebraic Drinfeld modules}
\label{ssec:drinfeld:def}
We are now ready to define Drinfeld modules over a general base, 
which might be different from $C_\infty$. The definition is of an algebraic nature. 
Roughly speaking, a Drinfeld module over a field $F$ is the datum of
an algebraic structure of $A$-module on $F$. Here, algebraicity means
that the multiplication by any element $a \in A$ is given by an Ore
polynomial $\phi_a \in F\{\tau\}$.

We start by introducing the base fields we will be interested in.

\begin{definition}
\label{def:Afield}
An \emph{$A$-field} is a field $F$ equipped with a ring homomorphism 
$\gamma : A \to F$.
The \emph{$A$--characteristic} of $F$ is the ideal of $A$ defined by
$\Acar(F)= \ker \gamma$. 
\end{definition}

Throughout the article, we set $z=\gamma(T)$. Since we suppose that $A=\Fq[T]$, the datum of $z$ completely determines $\gamma$. Beyond $F = \Cinf$, the most important $A$-fields considered in the literature are the following.
\begin{itemize}
\item 
$F$ is a finite extension of $\Fq(T)$ and $\gamma : A \to F$ is the canonical injection (so the $A$-characteristic is zero); this choice leads to the theory of \emph{rational} Drinfeld modules.
\item 
$F$ is a finite extension of $\Fq$ and $\gamma : A \to F$ is a presentation of $F$ as a quotient of $A = \Fq[T]$ (so the $A$-characteristic is nonzero, it is the defining polynomial of $F$); this choice leads to the theory of \emph{finite} Drinfeld modules. 
\end{itemize}

We will come back to these two particular bases and Drinfeld
modules over them in \cref{sec:arithmetics}.

\begin{definition}
\label{def:drinfeldmodule}
A \emph{Drinfeld module} over a $A$-field $(F, \gamma)$ is a homomorphism of $\Fq$-algebras
\Function
  {\phi}
  {A}
  {\Ftau}
  {a}
  {\phi_a}
such that:
\begin{enumerate}[label=(\roman{enumi})]
\item \label{cond:nontrivial}
$\phi_a$ is nonconstant for at least one element $a \in A$,
\item \label{cond:tangent}
for all $a \in A$, the constant coefficient of $\phi_a$ is $\gamma(a)$.
\end{enumerate}
\end{definition}

Condition~\ref{cond:nontrivial} in \cref{def:drinfeldmodule}
is not essential: it simply eliminates trivial cases, which may be a source
of problems at some point.
On the other hand, Condition~\ref{cond:tangent} is a normalisation condition which reflects the fact that the derivative of the exponential 
map (at $0$) is~$1$ (compare with Equation~\eqref{eq:phia}). In some 
sense, it says that the action of $\phi_a$ on the tangent space, 
\ie the action of the derivative of $\phi_a(x)$, should be the 
multiplication by $a$. This requirement is a standard fact in the theory 
of elliptic curves (or, more generally, abelian varieties) where the multiplication by $a$ on the curve always acts on the Lie algebra by 
scalar multiplication by $a$. 

\begin{remark}
It is also possible to define Drinfeld modules without referring to
$\gamma$ and define $\gamma$ afterwards by setting $\gamma := \phi
\bmod \tau$. However, it is more standard to proceed as we did in
\cref{def:drinfeldmodule}. This point of view also has
the advantage of underlining the importance of the tangent action.
\end{remark}

Since we assume that $A=\Fq[T]$, a Drinfeld module $\phi$ is
entirely determined by the Ore polynomial $\phi_T \in F\{\tau\}$.
Conversely, any choice of $\phi_T$ with positive degree and
constant term equal to $z$ gives rise to a unique Drinfeld module over $A$.
	
The $\tau$-degree of $\phi_T$ is called the \emph{rank} of $\phi$, and we denote it by $r$.
For any $a\in A$, it can be shown that the degree of $\phi_a$ in $\tau$ is equal to $r\deg a$.
 
\begin{sageexplain}
In SageMath, we create a Drinfeld module by specifying the polynomial ring $A$
and the coefficients of $\phi_T$ (we recall that $z$ was defined in
SageMath Example \ref{sage:orepolynomials} as a generator of $\F_{7^2}$).

\begin{sagecolorbox}
\begin{sagecommandline}

  sage: phi = DrinfeldModule(A, [z, z, z, z+1])
  sage: phi

\end{sagecommandline}
\end{sagecolorbox}

\begin{sagecolorbox}
\begin{sagecommandline}

  sage: phi.rank()
  sage: phi(T + 1)

\end{sagecommandline}
\end{sagecolorbox}

\end{sageexplain}

A Drinfeld module $\phi : A \to F\{\tau\}$ defines a structure of
$A$-module over any $F$-algebra $\mathcal F$ through the formula
$a \star x = \phi_a(x)$ for $a \in A$ and $x \in \mathcal F$. We will use the notation $\phiact{\mathcal F}$ for $\mathcal F$ 
equipped with this exotic action\footnote{In the literature, we often
find the notation $\phi(\mathcal F)$ instead; however, for this survey, 
we prefer using $\phiact{\mathcal F}$ in order to minimize the risk of confusion.}.

\subsection{Analytic uniformisation}
\label{ssec:drinfeld:uniformization}

When $F = \Cinf$, we have two concurrent viewpoints on Drinfeld modules: the first is analytic, as depicted in \cref{sec:carlitz} and the introduction of this section, while the second is algebraic, as described in \cref{ssec:drinfeld:def}.
Here, we aim to reconcile these viewpoints and explain that they both, indeed, define the same objects.

The direction ``analytic to algebraic'' has already been discussed in the introduction of Section \ref{sec:drinfeld}. We summarise it by the following theorem.

\begin{theorem}[\protect{\cite[Proposition~5.2.2]{Pap23}}]
Let $\Lambda$ be a discrete $A$-submodule of $\Cinf$ and let
\Function
  {e_\Lambda}
  {\Cinf}
  {\Cinf}
  {z}
  {\displaystyle z\prod_{\lambda \in \Lambda, \lambda \neq 0} \left (1-\frac{z}{\lambda}\right)}
be the corresponding exponential function.
Then, the function $e_\Lambda$ is entire, surjective and $\Fq$-linear.
Its zeros consist exactly of $\Lambda$, and are all simple.
Moreover, for every $a \in A$, there exists a unique Ore polynomial
$\phi_a^{\Lambda} \in \Cinf\{\tau\}$ such that
\begin{equation}\label{eq:A-mod-structure}
    \forall z \in \Cinf, 
    \quad e_\Lambda (az) = \phi_a^{\Lambda} (e_\Lambda (z)).
\end{equation}
Finally, the map
\EqFunction
  {\phi^\Lambda}
  {A}
  {\Cinf\{\tau\}}
  {a}
  {\phi_a^\Lambda}
  {eq:somelabelnamequinestjamaisutilise}
is a Drinfeld module over $\Cinf$ and its rank is equal to the 
rank of $\Lambda$ as an $A$-module.
\end{theorem}

The opposite direction ``algebraic to analytic'' is the so-called
\emph{analytic uniformization} theorem for Drinfeld modules. It can be
stated as follows.

\begin{theorem}
\label{th:drinfeld:uniformization}
Let $\phi : A \to \Cinf\{\tau\}$ be a Drinfeld module over $\Cinf$.
Then, there exists a lattice $\Lambda$ in $\Cinf$ such that 
$\phi = \phi^\Lambda$.
\end{theorem}

\begin{proof}[Sketch of the proof]
The first step is to reconstruct the exponential function $e_\Lambda$
as an Ore series over $\Cinf$, \ie an element of the form
$$e_\Lambda = \sum_{n=0}^\infty c_n \tau^n \qquad (c_n \in \Cinf).$$
To find the coefficients $c_n$, we come back to the functional
Equation~\eqref{eq:A-mod-structure} with $a = T$. In our language,
it reads $e_\Lambda T = \phi_T e_\Lambda$ as an equality in the ring of formal Ore power series $\Cinf\{\!\{\tau\}\!\}$.
Writing
$$\phi_T = g_0 + g_1 \tau + \cdots + g_r \tau^r \qquad 
(g_i \in \Cinf,\, g_0 =T)$$ and
identifying the coefficients (taking care of the noncommutativity
relation), we end up with the relation
$$c_n = \frac 1 {T^{q^n} - T} 
  \sum_{i=1}^{\min(n,r)} g_i c_{n-i}^{q^i}$$
which allows one to compute recursively all the $c_n$ (starting with
$c_0 = 1$). Finally, once we know the function $e_\Lambda$, we can reconstruct $\Lambda$ by taking its kernel.
The details of the proof can be found in \cite[Theorem~5.2.8]{Pap23}.
\end{proof}

The proof of \cref{th:drinfeld:uniformization} provides an algebraic perspective on the 
exponential function, which extends to any base of $A$-characteristic
zero. Precisely, if $\phi : A \to F\{\tau\}$ is a Drinfeld module 
over $(F, \gamma)$, one defines the exponential function $e_\phi$ 
as the Ore series
$$e_\phi = \sum_{n=0}^\infty c_n \tau^n$$
where the coefficients $c_n\in F$ are determined by the recurrence 
relation
$$c_n = \frac 1 {z^{q^n} - z} 
  \sum_{i=1}^{\min(n,r)} g_i c_{n-i}^{q^i}.$$
The fact that $F$ has $A$-characteristic zero ensures that the
denominator $$z^{q^n} - z = \gamma(T^{q^n} - T)$$ never vanishes.

Similarly, one can define a \emph{logarithm} function $\ell_\phi
= \sum_{n=1}^\infty \ell_n \tau^n \in F\{\!\{\tau\}\!\}$ as the 
solution of the equation 
$\ell_\phi  \phi_T = z  \ell_\phi$. Identifying the coefficients, we now find the relation
$$\ell_n = \frac 1 {z - z^{q^n}}  
  \sum_{i=1}^{\min(n,r)} g_i^{q^{n-i}} \ell_{n-i},$$
which again is enough to compute recursively all the $\ell_n$.
In addition to the identities 
\begin{align*}
e_\phi  \gamma(a) & = \phi_a  e_\phi, \\
\ell_\phi  \phi_a & = \gamma(a)  \ell_\phi,
\end{align*}
which hold true for all $a \in A$, we have the formal relation
$e_\phi \ell_\phi = \ell_\phi e_\phi = 1$ in $F\{\!\{\tau\}\!\}$.

\begin{sageexplain}
The exponential and the logarithm are implemented 
in SageMath. However, since Ore series do not exist in SageMath, the result is given by an actual series where $\tau^n$ is replaced
by $x^{q^n}$.

\begin{sagecolorbox}
\begin{sagecommandline}

  sage: phi = DrinfeldModule(A, [T, T, 1])
  sage: exp = phi.exponential(prec=100, name='x')
  sage: exp
  sage: log = phi.logarithm(prec=100, name='x')
  sage: log

\end{sagecommandline}
\end{sagecolorbox}

We can check that they are inverse to each other.

\begin{sagecolorbox}
\begin{sagecommandline}

  sage: exp(log)
  sage: log(exp)

\end{sagecommandline}
\end{sagecolorbox}

\end{sageexplain}

\subsection{Realisations of Drinfeld modules: Tate modules and Anderson motives}
\label{ssec:drinfeld:realizations}

The definition of Drinfeld modules may sound quite disconcerting
because they are not in any case modules over a ring.
Even worse, they are not ``parents'' in the sense of SageMath, 
meaning that they have no elements we can pick and work with.

In this subsection, we explain how to associate Drinfeld modules with actual objects of linear algebra. Beyond making us more comfortable, this will allow us afterwards to easily import classical tools from
linear algebra, such as characteristic polynomials, reduction theorems, \emph{etc}.

\subsubsection{Torsion points and Tate modules}\label{ss:torsion}

The first construction we are going to present is the Tate module: it is the straightforward analogue in our context of the classical Tate module of an elliptic curve.
We start by defining torsion points.

\begin{definition}
\label{def:torsion}
Let $\phi: A \to F\{\tau\}$ be a Drinfeld module over $(F, \gamma)$.
For an ideal $\mathfrak a \subset A$, the \emph{$\mathfrak a$-torsion} of $\phi$ is defined 
as the $\mathfrak{a}$-torsion of $\phiact\Fs$, namely
\[
    \phi[\mathfrak a] = \big\{ \, x \in \phiact\Fs \mid \forall a \in \mathfrak a , \,
              \phi_a(x) = 0 \big\}.
\]
\end{definition}

\begin{remark}
Since $A = \Fq[T]$ in our setting, every ideal $\mathfrak a$ of $A$ is principal, 
\ie $\mathfrak a = aA$, and the $\mathfrak a$-torsion is simply the set of roots of $\phi_a$.
The polynomial $\phi_a$ is sometimes called the \emph{Ore polynomial of $a$-division} of $\phi$.
\end{remark}

The $\mathfrak a$-torsion $\phi[\mathfrak a]$ is naturally a 
$A$-submodule of $\phiact\Fs$ which is annihilated by $\mathfrak a$, hence, it is a module over
the quotient ring $A/\mathfrak a$. Besides, it is equipped with an action
of the Galois group $\Gal(\Fs/F)$. It then gives rise to a $A/\mathfrak a$-linear
representation of $\Gal(\Fs/F)$.

\begin{proposition}[\protect{\cite[Corollary 3.5.3]{Pap23}}]
\label{prop:torsion}
Let $\phi: A \to F\{\tau\}$ be a Drinfeld module of rank~$r$.
Let $\mathfrak l$ be a prime ideal of $A$, which is different from the $A$-characteristic of $\phi$.
Then, for all nonnegative integer $n$, $\phi[\mathfrak{l}^n]$ is free of rank $r$ over
$A/\mathfrak l^n$.
\end{proposition}

\begin{remark}\label{rk:height}
When $\mathfrak l = \mathfrak p$ equals the $A$-characteristic of $\phi$,
it is still true that $\phi[\mathfrak{p}^n]$ is free over $A/\mathfrak p^n$ but its rank is always smaller than $r$.
The defect $\h(\phi) := r - \rk_{A/\mathfrak p^n}(\phi[\mathfrak{p}^n])$ does
not depend on $n$ and is called the \emph{Frobenius height} of $\phi$,
often referred to simply as \emph{height}.
It can also be interpreted as the $\tau$-valuation of the Ore polynomial 
$\phi_{\mathfrak p}$ divided by the degree of $\p$. This notion of height is different from the other heights attached to Drinfeld modules.
\end{remark}

Passing to the inverse limit, we obtain a family of free modules over $A_{\mathfrak l}$.

\begin{definition}
\label{definition:tate-module}
Let $\phi : A \to F\{\tau\}$ be a Drinfeld module and
let $\mathfrak l$ be a maximal ideal of~$A$.
The \emph{Tate module} of $\phi$ at $\mathfrak l$ is
\[
  \Tl(\phi) = \varprojlim_{n \in \N} \phi[\mathfrak{l}^n].
\]
\end{definition}

It follows that the Tate module of $\phi$ at $\mathfrak l$ has rank $r$ over $A_\mathfrak{l}$ when $\mathfrak l$ is different from the $A$-characteristic of $\phi$, and has rank $r-h(\phi)$ otherwise.

\subsubsection{Anderson motives}
\label{sssec:motives}

The second important construction is that of Anderson
motives. As we shall see, it is a very powerful notion which captures all the information encapsulated in a Drinfeld module and, at the same time, is easy to handle and work with, especially from an algorithmic perspective.
From a more abstract viewpoint, Anderson motives should be
considered as motives of Drinfeld modules in Grothendieck's vision.
So far, no analogue is known in the framework of elliptic curves and
abelian varieties.

\begin{definition}
\label{def:anderson-motive}
Let $\phi: A \to F\{\tau\}$ be a Drinfeld module. Its \emph{Anderson motive} $\MM(\phi)$ is $F\{\tau\}$ endowed with the following extra
structures:
\begin{itemize}
\item a structure of $A$-module given by:
$$\forall a \in A,\, \forall m \in \MM(\phi), \quad
a \bullet m = m \phi_a,$$
where the product on the right-hand side is the multiplication
in $F\{\tau\}$,
\item a structure of $F$-module given by:
$$\forall \lambda \in F,\, \forall m \in \MM(\phi), \quad
\lambda \bullet m = \lambda m,$$
\item an operator $\tau_{\MM(\phi)} : \MM(\phi) \to \MM(\phi)$
defined by $m \mapsto \tau m$.
\end{itemize}
\end{definition}

In what follows, we will omit the sign $\bullet$ and simply write
$am$ or $\lambda m$.
We note that being at the same time a
$A$-module and a $F$-module is equivalent to being a module over
$A \otimes_{\Fq} F$. In addition, since $A = \Fq[T]$ in our setting, we
have $A \otimes_{\Fq} F \simeq F[T]$.
Therefore, the motive $\MM(\phi)$ is no more than a $F[T]$ module
equipped with an additional endomorphism $\tau_{\MM(\phi)}$. We notice, moreover, that the latter is semilinear in the sense that
$$\forall f \in F[T], \, \forall m \in \MM(\phi), \quad
\tau_{\MM(\phi)}(fm) = \tau(f) \, \tau_{\MM(\phi)}(m),$$
where $\tau : F[T] \to F[T]$ is the ring homomorphism acting
trivially on $T$ and raising all the coefficients to the $q$-th power.
\begin{proposition}
\label{prop:basismotive}
If $\phi : A \to F\{\tau\}$ is a Drinfeld module of rank $r$,
then $\MM(\phi)$ is free of rank $r$ over $F[T]$ with basis
$(1, \tau, \ldots, \tau^{r-1})$.
\end{proposition}

\begin{proof}
We explain how to decompose an element $m \in \MM(\phi)$ in the
given basis. We recall that, by definition, $m$ is an Ore polynomial
in $F\{\tau\}$.
The latter being right Euclidean, we can decompose $m$ in the basis $\phi_T$ as follows:
$$m = m_0 + m_1 \phi_T + \cdots + m_k \phi_T^k,$$
for some nonnegative integer $k$ and some uniquely determined
$m_0, \ldots, m_k \in F\{t\}$ of degree at most $r{-}1$ where
$r = \deg_\tau \phi_T$ is the rank of $\phi$.
Indeed, $m_0$ is obtained as the remainder in the right division of
$m$ by $\phi_T$ and the next $m_i$ are obtained similarly by
dividing the successive quotients. Now writing
$$m_i =
\lambda_{i,0} + \lambda_{i,1} \tau + \cdots
\lambda_{i,r-1} \tau^{r-1},$$ we get the formula:
$$m = \sum_{i=0}^k \sum_{j=0}^{r-1} \lambda_{i,j} \tau^j \phi_T^i
    = \sum_{j=0}^{r-1} \left(\sum_{i=0}^k \lambda_{i,j} T^i\right) \tau^j,$$
which turns out to be the decomposition of $m$ in the basis
$(1, \tau, \ldots, \tau^{r-1})$.

The previous calculation shows that the family $(1, \tau, \ldots,
\tau^{r-1})$ generates $\MM(\phi)$. Proving that it is free is left
as an easy exercise for the reader.
\end{proof}

We say that $(1, \tau, \ldots, \tau^{r-1})$ is the
\emph{canonical basis} of $\MM(\phi)$. The proof we have detailed above
shows that writing an element of $\MM(\phi)$ in the canonical basis is
not a difficult task, which can definitely be implemented on a computer.

\begin{sageexplain}
One can create the motive attached to a Drinfeld module as follows.

\begin{sagecolorbox}
\begin{sagecommandline}

  sage: phi = DrinfeldModule(A, [z, z, z, z+1])
  sage: M = phi.anderson_motive()
  sage: M

\end{sagecommandline}
\end{sagecolorbox}

Now, we can create an element of $\MM(\phi)$ by passing in an Ore
polynomial and observe that the software automatically decomposes it
on the canonical basis.

\begin{sagecolorbox}
\begin{sagecommandline}

  sage: tau = phi.ore_variable()
  sage: M(tau)
  sage: M(tau^2)
  sage: M(tau^3)
  sage: M(phi(T))

\end{sagecommandline}
\end{sagecolorbox}
\end{sageexplain}

Having a canonical basis of $\MM(\phi)$ allows us to represent linear
morphisms with matrices. Following the SageMath convention, we shall
write vectors in row---in particular, the matrix $\Mat(f)$ of a
morphism $f$ will be defined as the matrix whose $i$th row are the
coefficients of the image by $f$ of the $i$th element of the basis.

Using the very same definition, it is also possible to attach a
matrix to semilinear operators, and especially to the endomorphism
$\tau_{\MM(\phi)}$. A simple computation shows that, if $\phi_T =
z + g_1 \tau + \cdots + g_r \tau^r$, then
\begin{equation}
\label{eq:motive:matrixtau}
\Mat\big(\tau_{\MM(\phi)}\big) =
\left( \begin{matrix}
0 & 1 \\
  & 0 & 1 \\
  &   & \ddots & \ddots \\
  &   &   & 0 & 1 \\
\frac{T-z}{g_r}  & \frac{-g_1}{g_r} & \cdots &
\frac{-g_{r-2}}{g_r} & \frac{-g_{r-1}}{g_r}
\end{matrix} \right) \in F[T]^{r\times r}.
\end{equation}

\begin{sageexplain}
We have access to the above matrix using the following syntax.

\begin{sagecolorbox}
\begin{sagecommandline}

  sage: M.matrix()

\end{sagecommandline}
\end{sagecolorbox}

\end{sageexplain}

\begin{example}
\label{ex:motive:carlitz}
An essential example is the motive of the Carlitz module (see Section~\ref{ssec:carlitz}). Recall briefly that
the corresponding Carlitz module is the Drinfeld module of rank $1$ given by the morphism $c : A \to F\{\tau\}$, $T
\mapsto z + \tau$.
Its motive $\MM(c)$ is then a $F[T]$-module of rank $1$,
\ie as a module one has $\MM(c) = F[T] \ee$ where we
used the letter $\ee$ to denote the canonical basis of $\MM(c)$.
In addition, applying Equation~\eqref{eq:motive:matrixtau}, one finds
that the matrix of $\tau_{\MM(c)}$ is simply
$\big(\begin{matrix}T - z \end{matrix}\big)$.
Therefore, the operator $\tau_{\MM(c)}$ is explicitly given by
$$\tau_{\MM(c)}(f \, \ee) =
\tau(f)\, \big(T - z\big) \, \ee.$$
\end{example}

\subsubsection{Comparison between Anderson motives and Tate modules}
\label{sssec:AndersonTate}

The two constructions we have presented previously are, in some 
sense, dual to each other. To properly settle this, we consider
the bilinear map
$$\MM(\phi) \times \Fs \rightarrow \Fs, \, (m,z) \mapsto m(z).$$
For all ideals $\mathfrak a$ of $A$, one checks that it induces a second bilinear map at the level of torsion points as follows:
\begin{equation}
\label{eq:motive:pairing}
\MM(\phi) / \mathfrak a \MM(\phi) \times \phi[\mathfrak a] \rightarrow \Fs.
\end{equation}
It then turns out that the above map is a ``perfect duality'' in 
the sense that it makes appear $\phi[\mathfrak a]$ as the dual of 
$\MM(\phi) / \mathfrak a \MM(\phi)$ and \emph{vice versa}.
Precisely, we have the following theorem.

\begin{theorem}
\label{th:motive:torsion}
Let $\phi : A \to F\{\tau\}$ be a Drinfeld module and let 
$\mathfrak a$ be an ideal of $A$. If the $A$-characteristic of $\phi$
is not zero, we assume that $\mathfrak a$ is coprime with it.
Then the bilinear map
\eqref{eq:motive:pairing} induces an isomorphism
$$\phi[\mathfrak a] \simeq
  \Hom_{F, \tau}\big(\MM(\phi)/\mathfrak a \MM(\phi), \, \Fs \big),$$
where $\Hom_{F, \tau}(\cdot,\cdot)$ denotes the $A$-module of $F$-linear maps commuting with the $\tau$-action, where $\tau$ acts on $F^s$ by $x \mapsto x^q$.
\end{theorem}
\begin{proof}
See \cite[\S 5.4]{G98}, or \cite[\S 2.1]{CL23} for a more
pedestrian approach.
\end{proof}

\begin{remark}
\cref{th:motive:torsion} can be seen as a realization
of Katz's famous equivalence of categories between Galois
representations and Frobenius modules \cite[Proposition~4.1.1]{K73}:
it precisely says that $\phi[\mathfrak a]$ and $\MM(\phi)/ \mathfrak a 
\MM(\phi)$ are in correspondence under this equivalence.
\end{remark}

Passing to the limit, we deduce from \cref{th:motive:torsion}
a formula for the Tate module $\TT_\fll(\phi)$ in terms of $\MM(\phi)$.
To write it down in a simple form, it is convenient to
introduce the dual motive $\MM(\phi)^\vee$ of $\MM(\phi)$.
It is defined by the usual formula
$$\MM(\phi)^\vee = \Hom_{F[T]}\big(\MM(\phi), F[T]\big).$$
The $\tau$-action on it (in the canonical dual basis) is defined
by the matrix
\begin{equation}
\label{eq:motive:matrixtaudual}
\Mat\big(\tau_{\MM(\phi)}\big)^{-1} =
\frac 1{T-z} 
\left( \begin{matrix}
g_1 & 1 \\
g_2 & 0 & 1 \\
\vdots  & & \ddots & \ddots \\
g_r &   &  & 0 & 1 \\
\end{matrix} \right).
\end{equation}
We observe, nonetheless, that the latter does not assume coefficients
in $F[T]$ but in $F[T][\frac 1{T-z}]$.
This means that $\tau_{\MM(\phi)^\vee}$ is \emph{not} an endomorphism
of $\MM(\phi)^\vee$ but only defines a mapping
$$\tau_{\MM(\phi)^\vee} :
  \MM(\phi)^\vee \to \MM(\phi)^\vee\left[\frac 1{T-z}\right].$$
We say that $\MM(\phi)^\vee$ is \emph{noneffective}.

For a maximal ideal $\fll$ of $A$, which is different from the
$A$-characteristic of $\phi$, we now have an isomorphism
\begin{equation}
\label{eq:motive:Tate}
\TT_\fll(\phi) \otimes_{\Fq} \Fs \simeq
  \MM(\phi)^\vee \otimes_{A \otimes F} (A_\fll \otimes \Fs),
\end{equation}
which is compatible with the $\tau$-action (after inverting
$T{-}z$ on the codomain).
We finally recover a formula for $\TT_\fll(\phi)$ by taking the
fixed points under the $\tau$-action.

\section{Morphisms of Drinfeld modules}
\label{ssec:drinfeld:morphisms}

In this section, we keep the notation from \cref{ssec:drinfeld:def}. In particular, we recall that $F$ is an $A$-field via the ring homomorphism $\gamma: A \rightarrow F$ and we set $z=\gamma(T)$.
\subsection{The formal definition}
\label{sssec:defmorphisms}

Drinfeld modules over $(F, \gamma)$ 
are designed to model algebraic structures of 
$A$-modules on $F$-algebras, and likewise, morphisms between two Drinfeld modules 
correspond to $A$-linear mappings which are algebraic, \ie encoded by Ore
polynomials.

The formal definition of a morphism between Drinfeld modules only
retains the underlying Ore polynomial and can be phrased as follows.

\begin{definition}
\label{def:morphism}
Let $\phi, \psi:A \to F\{\tau\}$ be two $A$-Drinfeld modules over 
$(F,\gamma)$.
A \emph{morphism} $u : \phi\to\psi$ is an Ore polynomial $u\in F\{\tau\}$ 
such that $u\phi_a=\psi_a u$ for all $a\in A$. An \emph{isogeny} is a 
nonzero morphism.

We let $\Hom(\phi,\psi)$ denote the set of morphisms from $\phi$
to $\psi$. When $\phi = \psi$, we simply write $\End(\phi)$ for
$\Hom(\phi,\phi)$.
\end{definition}

Since $A$ is $\Fq[T]$, the condition of \cref{def:morphism} is 
equivalent to requiring that $u\phi_T = \psi_T u$, in other words, it is sufficient to check the condition for $a = T$ only. Another remarkable fact is that two isogenous Drinfeld modules necessarily have the same rank. Indeed, the relation $u \phi_T =\psi_T u$ implies on the $\tau$-degrees 
that $\deg u + \deg \phi_T = \deg \psi_T + \deg u$, and thus that ${\deg 
\phi_T= \deg \psi_T}$.

\begin{remark}
\label{rem:notation-hom-end-cloture}
  We warn the reader that multiple conventions are possible for denoting the
  sets of morphisms and of endomorphisms. For example, when one has to distinguish between ordinary and supersingular Drinfeld modules (which we will define in \cref{subsubsec:super-ordi}), it is useful to consider morphisms defined on
  the separable closure, \ie elements $u \in \Fs\{\tau\}$ such that $u
  \phi_T = \psi_T u$. We would then write $\Hom_{\Fs}(\phi, \psi)$ for their
  set, and $\End_{\Fs}(\phi) = \Hom_{\Fs}(\phi, \phi)$. This notation differs
  from the one introduced in \cite[Proposition~3.3.10]{Pap23}, but aligns with the
  SageMath implementation (see SageMath Example~\ref{sage:supersingular-end}).
\end{remark}

We observe that, when only $\phi$ and $u$ are given, it is easy to infer the codomain $\psi$ by solving the equation $u \phi_T = 
\psi_T u$. This remark also leads to the following proposition.

\begin{proposition}
\label{prop:velu-ore}
  Let $\phi : A\to F\{\tau\}$ be a Drinfeld module.
  Let $u \in \Ftau$ be a nonzero Ore polynomial.
  Then $u$ defines an isogeny with domain $\phi$ if, and only if, 
  $u$ right-divides $u \phi_T$.
\end{proposition}

\begin{proof}
We assume that the condition of the proposition is fulfilled.
We define the Ore polynomial $\psi_T$ by the relation $u \phi_T = 
\psi_T u$. Looking at the coefficient of least degree on the right-hand side, we find that the constant coefficient
of $\psi_T$ is $z$. Hence $\psi_T$ can be prolonged to a
Drinfeld module $\psi : A \to F\{\tau\}$ and $u$ defines an
isogeny from $\phi$ to $\psi$.
The converse is proved similarly.
\end{proof}

\begin{sageexplain}
  One can create a morphism by simply passing the Ore
  polynomial defining it: if the codomain is not given, it is 
  automatically computed.

  \begin{sagecolorbox}
  \begin{sagecommandline}

    sage: phi = DrinfeldModule(A, [z, z, 1])
    sage: u = phi.hom(tau + 1)
    sage: u
    sage: psi = u.codomain()
    sage: psi
    sage: (tau + 1) * phi(T) == psi(T) * (tau + 1)

  \end{sagecommandline}
  \end{sagecolorbox}

\end{sageexplain}

\paragraph{Arithmetic operations on morphisms.}
All basic arithmetic operations on morphisms are defined by importing the standard arithmetic on Ore polynomials.
Precisely, morphisms can be:
\begin{itemize}
  \item Summed. If $u, v: \phi \to \psi$ are two morphisms, then the Ore
    polynomial $u{+}v$ also defines a morphism. Indeed, we have $$(u+v) \cdot
    \phi_T = u \phi_T + v \phi_T = \psi_T u + \psi_T v = \psi_T \cdot (u+v).$$

  \item Multiplied by scalars in $\Fq$. Note that since $\Fq$ is in the centre
    of $\Ftau$, for any isogeny $u$ on $\phi$, and any scalar $\lambda \in
    \Fq$, then $\lambda u$ and $u\lambda$ define the same isogeny.

  \item Multiplied by elements in $A$. Given $u : \phi \to \psi$ and $a \in A$,
    the product $a u$ is defined by the Ore polynomial $u \phi_a = \psi_a u$;
    we easily check that it is again an isogeny from $\phi$ to $\psi$.

  \item Composed. Given $u : \phi \to \psi$ and $v : \psi \to \mu$, the product
    of Ore polynomials $uv$ defines a morphism from $\phi$ to $\mu$. Indeed,
    one instantly checks that $uv \phi_T = u \psi_T v = \mu_T uv$. The identity
    morphism also makes sense: it is represented by the constant Ore polynomial $1$.
\end{itemize}
One can summarize the above properties by saying that Drinfeld modules over a fixed base $(F, \gamma)$ form a \emph{$A$-linear category}, meaning
in particular that the sets $\Hom(\phi,\psi)$ are modules over $A$ and
that the $\End(\phi)$ are algebras over $A$.

\begin{sageexplain}\label{sage:dualexample}
We create another isogeny $v$ and compose it with $u$.

\begin{sagecolorbox}
\begin{sagecommandline}

  sage: v = psi.hom(tau + z - 1)
  sage: v
  sage: u * v
  sage: v * u
  sage: u * v == T - 1
  sage: v * u == T - 1

\end{sagecommandline}
\end{sagecolorbox}

\noindent
In the code above, the last two lines check that the composites 
$u \circ v$ and $v \circ u$ are both the scalar multiplications by
$T{-}1$.
We observe, nevertheless, that the Ore polynomials defining them
are different. Indeed, the first one is $\psi_{T-1}$ (since $u \circ v$
is an endomorphism of~$\psi$), while the second one is $\phi_{T-1}$.
We shall see later on in \cref{sssec:norm-dual}
that $T{-}1$ is the norm of $u$, and $v$ is its dual isogeny.
\end{sageexplain}

\begin{proposition}
\label{proposition:rgcd-isogeny}
Let $\phi : A\to F\{\tau\}$ be a Drinfeld module.
Let $u_1, \dots, u_n$ be isogenies with domain $\phi$. Then, the Ore polynomials $\rgcd(u_1, \dots, u_n)$ and 
$\llcm(u_1, \ldots, u_n)$ both define isogenies with domain $\phi$.
\end{proposition}

\begin{proof}
  We let $\psi_i$ denote the codomain of each isogeny $u_i$ ($1 \leq i \leq n$).
  For every index $i$, we have the relation $u_i \phi_T = \psi_{i,T} u_i$.

  We now write $u = \rgcd(u_1, \dots, u_n)$.
  Given that $\Ftau$ is left-euclidean, there exists a Bézout relation
  $u = \sum_{i=1}^n m_i u_i$ with $m_i \in F\{\tau\}$.
  We deduce that $u \phi_T = \sum_{i=1}^n m_i \psi_{i,T} u_i$. Hence
  $u \phi_T$ is in the left-ideal generated by the $u_i$, and so
  that it is right-divided by their right-gcd~$u$. 
  We conclude that $u$ defines an isogeny with domain $\phi$ by
  invoking \cref{prop:velu-ore}.

  Let now $v = \llcm(u_1, \ldots, u_n)$.
  By definition $u_i$ right-divides $v$. Thus, we can write 
  $v = m_i u_i$ with $m_i \in F\{\tau\}$ and obtain the relation
  $v \phi_T = m_i u_i \phi_T = m_i \psi_{T,i} u_i$ which proves
  that $u_i$ also right-divides $v \phi_T$. Since this holds for
  all $u$, we find that $v$ also right-divides $v \phi_T$. Again,
  we conclude using \cref{prop:velu-ore}.
\end{proof}

\paragraph{The action on the \texorpdfstring{$\Fs$}{F\^s}-points.}

Another option to approach morphisms of Drinfeld modules is to study their action on $F$-algebras.
Indeed, any morphism of Drinfeld modules $u: \phi \to \psi$ in the previous sense induces maps, which we will still denote by $u$
$$\begin{array}{rcl}
\phiact{\mathcal F} & \to & \psiact{\mathcal F} \\
x & \mapsto &  u(x),
\end{array}$$
for every $F$-algebra $\mathcal F$. 
This construction is particularly relevant when $\mathcal F$ is $\Fs$. Indeed, one can prove that two isogenies $u, v : \phi \to \psi$ are 
equal if, and only if, they induce the same morphisms $\phiact{\Fs}\to 
\psiact{\Fs}$.

We will elaborate more on this viewpoint
in \cref{tate-module-functor} below.

\subsection{Separability}
Regarding isogenies of elliptic curves, there is an important distinction between separable and nonseparable ones: the former correspond to finite subgroups, whereas the latter are typically modelled by the Frobenius morphism. Such a dichotomy is also relevant in the case of Drinfeld modules, as we will discuss now.

\subsubsection{Separable isogenies and their kernels}
We start by defining the notion of separability in the 
context of Drinfeld modules.
\begin{definition}
\label{def:separable}
An isogeny $u : \phi \to \psi$ is \emph{separable} if its 
$\tau$-valuation is zero, \ie its constant coefficient does not vanish 
when viewed as an Ore polynomial.
\end{definition}

We recall from \cref{prop:Orebijection}
that we have defined the kernel of a morphism $u$ by
$$\ker u = \big\{\, x \in \phiact\Fs \mid u(x) = 0 \,\big\}.$$
It is an $A$-submodule of $\phiact \Fs$, which is finite if, and
only if, $u$ is an isogeny. Besides, as a subset of $\Fs$, it is
equipped with an action of $\Gal(\Fs/F)$.

\begin{proposition}
\label{prop:velu}
Let $\phi : A\to F\{\tau\}$ be a Drinfeld module.
Let $G$ be an $A$-submodule of $\phiact\Fs$ which is finite and stable by 
the action of $\Gal(\Fs/F)$. Then, there exists a Drinfeld module $\psi :
A \to F\{\tau\}$ and a separable isogeny $u : \phi \to \psi$ such that 
$\ker u = G$.
\end{proposition}
\begin{proof}
Applying \cref{prop:Orebijection}, we find an Ore polynomial 
$u \in F\{\tau\}$ with constant term~$1$ such that $\ker u = G$.
Moreover, the fact that $G$ is a $A$-submodule ensures that 
$\phi_T(G) \subset G$, from what we derive $\ker u 
\subset \ker(u \phi_T)$. Thus, it follows from 
Remark~\ref{rem:Orebijection} that $u$ right-divides $u \phi_T$, 
\ie there exists an Ore polynomial $\psi_T \in F\{\tau\}$ such 
that $u \phi_T = \psi_T u$. This Ore polynomial $\psi_T$ defines a 
Drinfeld module $\psi$ for which $u$ becomes an isogeny $\phi \to \psi$.
\end{proof}

\begin{remark}
The isogeny $\psi$ of \cref{prop:velu} is uniquely determined up to an isomorphism of $\psi$, \ie multiplication
by a nonzero element of $\Fq$ (see \cref{ssec:drinfeld:isomorphisms}).
We can then normalise it by requiring either that it has a constant
coefficient $1$ (as in \cref{prop:Orebijection}) or that
it is monic.
\end{remark}

It is interesting to reinterpret 
\cref{proposition:rgcd-isogeny} in light of what precedes. Indeed, we 
recall from \cref{rem:Orebijection} that right-gcds (resp.~left-lcms) 
correspond to taking intersections (resp.~sums). Therefore, when the 
isogenies $u_i$ are separable, the Ore polynomials $\rgcd(u_1, \dots, 
u_n)$ and $\llcm(u_1, \ldots, u_n)$ define the isogenies having 
$\ker(u_1) \cap \cdots \cap \ker(u_n)$ and $\ker(u_1) + \cdots + 
\ker(u_n)$ as kernel, respectively.

\subsubsection{Nonseparable isogenies and the Frobenius morphism}
\label{sssec:Frobenius}

When the $A$-characteristic of $F$ is zero, one checks that all isogenies
are separable and thus correspond to finite $A$-submodules of $\phiact{F^s}$.

On the contrary, when the $A$-characteristic of $F$ is a maximal ideal
$\pp$, inseparable isogenies do exist. Indeed, from the fact that 
$\Acar(F) = \pp$, we derive that $z^{q^{\deg\pp}} = z$. Therefore, 
if we write $\phi_T = z + g_1 \tau + \cdots + g_r \tau^r$, the Ore polynomial
$\tau^{\deg\pp} \in \Ftau$ satisfies the following commutation relation:
\[
  \tau^{\deg\pp} \phi_T
  = z \tau^{\deg\pp} + \sum_{i = 1}^r g_i^{q^{\deg\pp}} \tau^{i + \deg\pp}
  = \phi'_T \tau^{\deg\pp}
\]
where
\begin{equation}
\label{eq:frobcodomain}
  \phi'_T = z + g_1^{q^{\deg\pp}} \tau + \cdots + g_r^{q^{\deg\pp}} \tau^r.
\end{equation}
In other words, $\tau^{\deg\pp}$ defines a nonseparable isogeny from 
$\phi$ to the Drinfeld module $\phi'$ defined by
Equation~\eqref{eq:frobcodomain}.

\begin{definition}
\label{def:relativeFrobenius}
  The isogeny $\tau^{\deg(\p)}:\phi\to\phi'$ is called the 
  \emph{relative Frobenius morphism} of~$\phi$.
\end{definition}

\begin{sageexplain}
In the example below, we have $\gamma(T) = 1$, so the relative Frobenius
is represented by the Ore polynomial $\tau$.

\begin{sagecolorbox}
\begin{sagecommandline}

  sage: phi = DrinfeldModule(A, [1, z, z])
  sage: phi.frobenius_relative()

\end{sagecommandline}
\end{sagecolorbox}

On the contrary, when $\gamma(T)$ generates $F$, the relative Frobenius
is an endomorphism. It is the so-called \emph{Frobenius endomorphism},
which will be discussed in more detail in \cref{sssec:frobendomorphism}.

\begin{sagecolorbox}
\begin{sagecommandline}

  sage: phi = DrinfeldModule(A, [z, z, 1])
  sage: phi.frobenius_relative()

\end{sagecommandline}
\end{sagecolorbox}

\end{sageexplain}

It turns out that the Frobenius morphism is the prototypical example of
nonseparable isogenies, as underlined by the following proposition.

\begin{proposition}
\label{prop:facto-separable-inseparable}
  Any nonseparable isogeny
  $u: \phi \to \psi$ between Drinfeld modules of $A$-characteristic $\pp$ factors through the relative Frobenius of $\phi$, \ie 
  there exists an isogeny $u' : \phi' \to \psi$ such that
  $u = u' \tau^{\deg \pp}$.
  
  Any isogeny $u: \phi \to \psi$ can be written $u = v 
\tau^{m \deg\pp}$ where $v$ is separable and $m$ is a nonnegative integer.
\end{proposition}
\begin{proof}
It is enough to show that the $\tau$-valuation of $u$ is at least 
$\deg \pp$. This follows by writing down the commutation relation
$u \phi_T = \psi_T u$ and comparing the coefficients of the smallest degree. The second statement follows by repeatedly applying the first one. 
\end{proof}

\begin{remark}
We stress that the factorisation in the opposite direction, namely
$u = \tau^{\deg\pp} u'$ may fail in full generality.
However, it does always exist when $F$ is a finite field.
\end{remark}

As in the case of elliptic curves, the Frobenius morphism plays a 
primary role in the study of Drinfeld modules over finite fields.
We will elaborate more on this in \cref{subsec:finite-fields}.

\subsection{Isomorphisms and  \texorpdfstring{$j$}{j}-invariants}
\label{ssec:drinfeld:isomorphisms}

An important family of morphisms is, of course, that of isomorphisms.
Coming back to the definition, we see that an isomorphism $u: \phi
\to \psi$ is encoded by an invertible Ore polynomial $u$. In virtue of the additivity of the degrees, it turns out that
invertible Ore polynomials are just nonzero constant ones. Hence, isomorphisms 
are simply given by elements in $F^\times$. That being said, it is possible to
consider isomorphisms over any extension $F'/F$: constant Ore polynomials $u$ of
$F'\{\tau\}$ such that $u \phi_T = \psi_T u$. 
In this subsection, we describe a procedure to test
isomorphism, and give invariants to encode isomorphism classes (over $\Fs$).

\subsubsection{Deciding whether two Drinfeld modules are isomorphic}
\label{sssec:drinfeld:decideisom}

Let us fix a Drinfeld module $\phi : A \to F\{\tau\}$ of rank
$r$ and write $$\phi_T = z + g_1 \tau + g_2 \tau^2 + \cdots + g_r \tau^r.$$

Any element $u \in (\Fs)^\times$ now defines an isogeny $u : \phi \to \psi$ where
$\psi$ is the Drinfeld module of rank $r$ which is explicitely defined by
$$\psi_T = z + u^{1-q} g_1 \tau + u^{1-q^2} g_2 \tau^2 + \cdots + u^{1-q^r} g_r
\tau^r.$$ Another consequence of this calculation is a procedure to decide if two Drinfeld modules $\phi$ and $\psi$ are isomorphic: it is the case if, and only if, the coefficients of $\psi_T$ can be deduced by those of $\phi_T$ by
multiplying by $u^{1-q^i}$ (where $i$ is the corresponding $\tau$-degree).
Although this criterion looks quite easy to check, one needs to be careful with
the possible vanishing of the coefficients. However, paying attention to
properly handling this somewhat annoying case, one ends up with an efficient
algorithm to check isomorphism (see also
\cite[\S~3.2.3.5]{leudiere_morphisms_2024} for additional details). 

\begin{sageexplain}
\label{sage:isom-1}
The isomorphism test we mentioned is implemented.
This method
even allows us to decide if the Drinfeld modules are isomorphic over the base
field, or over the separable closure, the former being the default.

\begin{sagecolorbox}
\begin{sagecommandline}

  sage: phi = DrinfeldModule(A, [z, 1, 1])
  sage: psi = DrinfeldModule(A, [z, 2*z + 1, 1])
  sage: phi.is_isomorphic(psi)

\end{sagecommandline}
\end{sagecolorbox}

Here, we can check by hand that the isomorphism is given by the 
constant Ore polynomial $z + 1$.

\begin{sagecolorbox}
\begin{sagecommandline}

  sage: u = phi.hom(z + 1)
  sage: u.codomain() is psi
  sage: u.is_isomorphism()

\end{sagecommandline}
\end{sagecolorbox}

Below is a second example where $\phi$ and $\psi$ are not isomorphic
over the ground $A$-field~$F$, but are over $\Fs$.

\begin{sagecolorbox}
\begin{sagecommandline}

  sage: psi = DrinfeldModule(A, [z, z, 3])
  sage: phi.is_isomorphic(psi)
  sage: phi.is_isomorphic(psi, absolutely=True)

\end{sagecommandline}
\end{sagecolorbox}
\end{sageexplain}

\subsubsection{\texorpdfstring{$j$}{j}-invariants}
\label{subsec:j-inv}

An important fact is that isomorphism classes over $\Fs$
can be captured by algebraic invariants.
In rank $1$, the situation is trivial: all Drinfeld modules of
rank $1$ are isomorphism to the Carlitz module $T \mapsto z+\tau$
over $\Fs$.

\paragraph{Rank-$2$ Drinfeld modules.}
It follows from what we did in Subsection~\ref{sssec:drinfeld:decideisom}
that two Drinfeld modules $\phi$ and $\psi$ of rank $2$ defined by
\begin{align*}
\phi_T & = z + g_1 \tau + g_2 \tau^2 \qquad (g_2 \neq 0), \\
\psi_T & = z + h_1 \tau + h_2 \tau^2 \qquad (h_2 \neq 0),
\end{align*}
are $\Fs$-isomorphic if, and only if, there exists $u \in (\Fs)^\times$ 
such that $h_1 = u^{q-1} g_1$ and $h_2 = u^{q^2-1} g_2$. Raising
the first equation to the power $q{+}1$, we get
$h_1^{q+1} = u^{q^2-1} g_1^{q+1}$,  and combining now with the second equation, we get the necessary condition
\begin{equation}
\label{eq:jinv2}
\frac{g_1^{q+1}}{g_2} = \frac{h_1^{q+1}}{h_2}.
\end{equation}
It turns out that this condition is also sufficient. Indeed, if $g_1 \neq 0$ (and so
$h_1 \neq 0$ as well), we solve the first equation $h_1 = 
u^{q-1} g_1$ and verify that any solution is also a solution of the second one. On the contrary, if $g_1 = h_1 = 0$, we just solve the second equation, the first one being automatically
fulfilled.
\begin{definition}\label{def:jinvariant}
The quantity $g_1^{q+1}/g_2$ appearing in
Equation~\eqref{eq:jinv2} is called the \emph{$j$-invariant}
of $\phi$ and is denoted by $j(\phi)$.
    
\end{definition}
It characterises the isomorphism class of $\phi$ over $\Fs$.

\begin{sageexplain}
We reuse the isomorphic (over the separable closure, but not the base field) Drinfeld modules of SageMath Example~\ref{sage:isom-1}, and verify that they have the same $j$-invariant.

\begin{sagecolorbox}
\begin{sagecommandline}

  sage: phi.j_invariant()
  sage: psi.j_invariant()

\end{sagecommandline}
\end{sagecolorbox}
\end{sageexplain}

\begin{remark}
\label{rem:j-inv-representative}
  
  For any $j$ in K, let $\phi_T = z + \tau^2$ if $j=0$ or $\phi_T = z + \tau +
  j^{-1}\tau^2$ otherwise. Then the Drinfeld module $A \to F\{\tau\}$ defined
  by $\phi_T$ has $j$-invariant $j$.

\end{remark}

\paragraph{Higher rank Drinfeld modules.}

Isomorphism invariants for Drinfeld modules of higher
ranks were introduced by Potemine \cite{Potemine_1998}.

\begin{definition}
\label{def:jinvariants}
Let $\phi:A\to F\{\tau\}$ be a Drinfeld module of rank $r$. Write $\phi_T = z + g_1\tau+\dots+ g_r \tau^r$ with $g_r\neq 0$. 
Let $\ell < r$ and $\bk = (k_1, \ldots, k_\ell)$ be a multi-index with $1\leq k_1\leq\dots\leq k_\ell\leq r-1$ 
and let $\bs = (s_1,\dots,s_\ell, s_r)$ be a $(\ell{+}1)$--tuple of integers such that the following hold:
\begin{itemize}
\item $0\leq s_i \leq \dfrac{q^r-1}{q^{\gcd(k_i,r)}-1}$ for every $i\in\{1,\dots,\ell\}$,
\item $s_1\left(q^{k_1}-1\right)+\dots + s_\ell \left(q^{k_\ell}-1\right) = s_r\left(q^r-1\right)$.
\end{itemize}
The \emph{Potemine $J$-invariant} of $\phi$ of index $(\bk, \bs)$ is
\[J_{\bk}^{\bs}(\phi) = 
  \frac{g_{k_1}^{s_1} \cdots g_{k_\ell}^{s_\ell}}{g_r^{s_r}}.\]
\end{definition}

There is \emph{a priori} an infinite number of acceptable pairs $(\bk, \bs)$.
Nevertheless, changing $\bs$ into $n \bs$ (with $n \in \N$) results in
raising the corresponding $J$-invariant to the power $n$. For this reason,
it is safe to restrict ourselves to the so-called \emph{basic} parameters,
that are the parameters $(\bk, \bs)$ satisfying the extra condition $\gcd(s_1, \ldots, s_\ell, s_r) = 1$.

Although the list of basic parameters gets rapidly very long when $r$
increases, it always remains finite. As an extreme case, when $r = 2$,
there is one unique basic parameter, namely $\bk=(1)$ and $\bs = (q{+}1, 1)$. 
The corresponding Potemine $J$-invariant is the $j$-invariant we 
have introduced in \cref{def:jinvariant}.

\begin{sageexplain}
SageMath has methods to build the list of basic parameters and to 
compute the corresponding $J$-invariants in any rank. We already see in the example below that the complete list of Potemine $J$-invariants can
be very long, even in rank~$3$.

\begin{sagecolorbox}
\begin{sagecommandline}

  sage: phi = DrinfeldModule(A, [z, z+1, z+2, z+3])
  sage: phi.basic_j_invariants()

\end{sagecommandline}
\end{sagecolorbox}
\end{sageexplain}

\begin{remark}
We warn the reader that different naming conventions may be employed for the invariants. We chose to follow that of Potemine, which is also used in SageMath. Papikian uses a different notation:
the $j$-invariants, respectively the  \emph{basic} $j$-invariants, of  \cite{Pap23} correspond to the basic $j$-invariants
of Potemine, respectively to the Potemine $J$-invariant associated to the parameters
$\bk = (k)$ and $\bs = \frac 1 {q^{\gcd(k, r) - 1}} \cdot (q^r - 1, q^k - 1)$,
namely
$$\frac{g_k^{(q^r - 1) / (q^{\gcd(k, r) - 1})}}{g_r^{(q^k - 1) / (q^{\gcd(k, r) - 1})}}.$$
\end{remark}

These invariants characterise the isomorphism classes of Drinfeld modules over the algebraic closure, as can be seen in the following theorem.
\begin{theorem}[\protect{\cite[Theorem 3.8.11]{Pap23}}]
Let $\phi,\psi: A \to F\{\tau\}$ be two Drinfeld modules of the same rank.
We assume that $F$ is separably closed.
Then $\phi$ and $\psi$ are isomorphic if and only if they have the same basic Potemine $J$-invariants.
\end{theorem}

\subsection{Action on Anderson motives and consequences}
\label{tate-module-functor}

We have already said previously that a morphism $u: \phi\to\psi$
between Drinfeld modules can be realised as an actual $A$-linear map
$u : \phiact{\Fs} \to \psiact{\Fs}$. However, the structure of $A$-module of $\phiact{\Fs}$ is
not easy to describe and to work with; for example, it is usually not
of finite type (see~\cite{Poonen95}).
One can work around this issue by substituting the Tate module 
$\TT_{\mathfrak l}(\phi)$ to $\phiact{\Fs}$. Choosing correctly the 
prime ideal $\mathfrak l$, we then know that $\TT_{\mathfrak l}(\phi)$ 
is a free module over $A_{\mathfrak l}$ of rank $r := 
\text{rank}(\phi)$.
Although this modification definitely allows for many nice applications, 
it still has the unpleasant disadvantage of introducing an auxiliary
prime $\mathfrak l$ and the associated completion $A_{\mathfrak l}$.
In what follows, we will show that considering Anderson motives (introduced in \cref{sssec:motives}) instead of Tate modules elegantly resolves all these issues.

Given an isogeny $u : \phi \to \psi$, we define
$$\begin{array}{rcl}
\MM(u) : \quad \MM(\psi) & \longrightarrow & \MM(\phi) \\
m & \mapsto & mu
\end{array},$$
where the product $mu$ is computed in $F\{\tau\}$.
One readily checks that $\MM(u)$ is $F[T]$-linear and commutes with 
the $\tau$-action, \ie it satisfies
$$\MM(u) \circ \tau_{\MM(\psi)} = \tau_{\MM(\phi)} \circ \MM(u).$$
We emphasize that the construction is contravariant, \ie the direction of the arrows is reversed: if $u$ goes from $\phi$
to $\psi$, then $\MM(u)$ goes from $\MM(\psi)$ to $\MM(\phi)$.

Given that $\MM(u)$ is a linear map, one can consider its matrix
in the canonical bases of $\MM(\psi)$ and $\MM(\phi)$, respectively.
Concretely, the $i$th row of $\Mat(\MM(u))$ is formed by the
coefficients of $\tau^i u$ in $\MM(\phi)$.

\begin{sageexplain}
\label{sage:anderson-matrix}

\begin{sagecolorbox}
\begin{sagecommandline}

  sage: phi = DrinfeldModule(A, [z, z, z, 1])
  sage: u = phi.hom(tau + 5)
  sage: Mu = u.anderson_motive()
  sage: Mu
  sage: Mu.matrix()

\end{sagecommandline}
\end{sagecolorbox}
\end{sageexplain}

\subsubsection{Full faithfulness theorems}

The first important result about Anderson motives is the following
theorem, which tells that the functor $\MM$ is fully faithful. 

\begin{theorem}
\label{th:motive:fullyfaithful}
Let $\phi, \psi : A \to F\{\tau\}$ be two Drinfeld modules.
Then, there is a canonical bijection
$$\begin{array}{rcl}
\Hom(\phi,\psi) & \stackrel\sim\longrightarrow &
  \Hom_{F[T], \tau}\big(\MM(\psi),\,\MM(\phi)\big) \\
u & \mapsto & \MM(u)
\end{array},$$
where $\Hom_{F[T], \tau}$ consists of all
$F[T]$-linear maps commuting with the $\tau$-action.
\end{theorem}

\begin{proof}
It suffices to prove that the inverse map is given by $f \mapsto
f(1)$, which is a straightforward verification.
\end{proof}

Using the fact that the Anderson motive determines the Tate
modules (see Equation~\eqref{eq:motive:Tate}), we deduce from \cref{th:motive:fullyfaithful} a full faithfulness result at the level of Tate modules, which can be seen as an analogue of Tate's theorem on abelian varieties.

\begin{theorem}
\label{theorem:fidelite-tate}
  Let $\mathfrak l$ be a maximal ideal of $A$ which is different
  from the $A$-characteristic of $(F, \gamma)$.
  Then, for all Drinfeld modules $\phi, \psi : A \to F\{\tau\}$,
  the canonical map
  \[
    A_{\mathfrak l} \otimes_A \Hom(\phi, \psi) \to 
      \Hom_{A_{\mathfrak l}}\big(\TT_{\mathfrak l}(\phi), \TT_{\mathfrak l}(\psi)\big)
  \]
  is injective and its image consists exactly of the morphisms
  $\TT_{\mathfrak l}(\phi) \to \TT_{\mathfrak l}(\psi)$ which are
  equivariant under the action of $\Gal(\Fs/F)$.
\end{theorem}

\subsubsection{Characteristic polynomials}
\label{subsubsec:charpoly}

Anderson modules are also quite useful to import classical
constructions of linear algebra and attach meaningful invariants
to morphisms of Drinfeld modules.
A typical example in this line is the construction of the
characteristic polynomial of an endomorphism of a Drinfeld module,
which can be defined as follows.

\begin{definition}\label{def:charpoly}
  Let $u$ be an endomorphism of a Drinfeld module $\phi$.
  The \emph{characteristic polynomial} of $u$, denoted by $\chi_u$,
  is defined as the characteristic polynomial of the $F[T]$-linear
  map $\MM(u) : \MM(\phi) \to \MM(\phi)$, \ie
  $$\chi_u(X) = \det\big(X{\cdot}\text{id} - \MM(u)\big).$$
\end{definition}

It follows from the definition that $\chi_u(X)$ is a polynomial of
degree $r$, the rank of $\phi$. It has \emph{a priori} coefficients
in $F[T] \simeq A\otimes_{\Fq} F$, but one actually derives from the fact that $\MM(u)$ commutes with the $\tau$-action that $\chi_u(X)$ takes
coefficients in $A$. 
Besides, the Cayley--Hamilton theorem asserts that $\chi_u(\MM(u))=0$, 
which implies that $\chi_u(u) = 0$ by \cref{th:motive:fullyfaithful}.

\cref{def:charpoly}, coupled with \cref{prop:basismotive}, directly leads to
an algorithm to compute characteristic polynomials of endomorphisms.
This is currently the default method in SageMath.
More details, including complexity statements and various optimisations,
are provided in \cite{CL23}.

\begin{sageexplain}
\label{sageexplain:charpoly}

\begin{sagecolorbox}
\begin{sagecommandline}

  sage: phi = DrinfeldModule(A, [z, z, 1])
  sage: u = phi.hom(tau^4 + z*tau^3 + (z + 1)*tau^2)
  sage: chi = u.characteristic_polynomial()
  sage: chi

\end{sagecommandline}
\end{sagecolorbox}

Below, we decompose the computation of $\chi$ in two steps: first, we compute
the matrix of the Frobenius endomorphism acting on $\MM(\phi)$ and second, we
compute its characteristic polynomial.

\begin{sagecolorbox}
\begin{sagecommandline}

  sage: Mu = u.anderson_motive()
  sage: Mu.matrix()
  sage: Mu.matrix().charpoly()

\end{sagecommandline}
\end{sagecolorbox}
One can check on this example that the characteristic polynomial
$\chi$ indeed annihilates $u$.

\begin{sagecolorbox}
\begin{sagecommandline}

  sage: chi(u)
  sage: u^2 + (5*T^2 + 3*T + 5)*u + T^4 + T^3 + 2*T^2 + 5*T + 3
  sage: phiT = phi.gen()
  sage: mu = u.ore_polynomial()
  sage: mu^2 + (5*mu*phiT^2 + 3*mu*phiT + 5*mu) + (phiT^4 + phiT^3 + 2*phiT^2 + 5*phiT + 3)

\end{sagecommandline}
\end{sagecolorbox}
\end{sageexplain}

\begin{remark}
\label{remark:calcul-charpoly}
  Musleh and Schost also proposed a general method for computing characteristic
  polynomials of endomorphisms \cite{musleh_computing_2023}. Instead of relying
  on Anderson motives, they rely on the \emph{crystalline cohomology} of the
  Drinfeld module~\cite{angles_characteristic_1997} (see also
  \cite[Chapter~5]{musleh_algorithms_2023}).
  Both methods involve computing the characteristic polynomial as the classical characteristic polynomial of a polynomial matrix.
  These two families of algorithms were the first to achieve
  computation of characteristic polynomials of endomorphisms of Drinfeld
  modules in a relatively large context, whereas previous methods used to restrict to the sole  Frobenius endomorphism case, as we will discuss in
  \cref{sec:review-frobenius-charpoly}.
\end{remark}

After Equation~\eqref{eq:motive:Tate} comparing Anderson motives and
Tate modules, one also finds that, for any prime ideal $\mathfrak l$
different from the $A$-characteristic, $\chi_u(X)$ equals the 
characteristic polynomial of $\TT_{\mathfrak l}(u)$ acting on the 
Tate module $\TT_{\mathfrak l}(\phi)$.
This reinterpretation is actually the classical definition in the framework of elliptic curves, where the notion of Anderson motive is
missing. We then see clearly here the benefit of Anderson motives: they allow for a simpler definition which, on the one hand, does
not depend on an auxiliary prime $\mathfrak l$ (and so avoids 
proving independence results on $\mathfrak l$) and, on the other
hand, is much more suitable for computations. In particular, it 
does not involve the separable closure of $F$, which cannot be
easily handled on computers.

\subsubsection{Norms and duals of isogenies}
\label{sssec:norm-dual}

In a similar fashion, one can also define norms and dual isogenies.

\begin{definition}
\label{def:normisogeny}
Let $u : \phi \to \psi$ be a morphism of Drinfeld modules over
$(F,\gamma)$. The \emph{norm} of $u$ is defined by
$$\text{norm}(u) := \frac{\det \MM(u)}{\text{\text{lc}}(\det \MM(u))}$$
where the determinant is computed over $F[T]$ and $\text{lc}$ refers to
the leading coefficient.
\end{definition}

We underline that $\MM(u)$ goes from $\MM(\psi)$ to $\MM(\phi)$, so
it is not an endomorphism in general. Hence, taking its determinant
requires some precaution since its value depends \emph{a priori} on
the choices of bases of the domain and the codomain.
However, changing the bases only modifies the determinant by 
multiplication by an element of $F[T]^\times = F^\times$. Then the quotient of the determinant by its leading coefficient is well-defined.
In practice, we can carry out the computations by picking the
canonical bases of $\MM(\phi)$ and $\MM(\psi)$ given by
\cref{prop:basismotive}.

\begin{remark}
\label{rem:normideal}
When $A$ is not $\Fq[T]$, it is no longer possible to divide by the
leading coefficient of $\det \MM(u))$. Instead, one may consider the
ideal of $A$ generated by $\det \MM(u)$, which remains canonically
defined.
\end{remark}

As in the case of characteristic polynomials, one proves that
the norm of $u$ always lies in $A$.
Moreover, the norm is multiplicative with respect to the composition of isogenies.

Besides, in the same manner as for characteristic polynomials, 
\cref{def:normisogeny} directly leads to an algorithm for computing norms.
We again refer to \cite{CL23} for detailed complexity statements.

\begin{sageexplain}
  We compute the norm of $u$ defined previously and
  observe that it is the constant coefficient of $\chi_u$.
  We note that the function \texttt{norm} of SageMath returns
  the ideal, and not a generator (see Remark~\ref{rem:normideal}).

  \begin{sagecolorbox}
  \begin{sagecommandline}

    sage: u.norm()
    sage: chi.constant_coefficient()

  \end{sagecommandline}
  \end{sagecolorbox}

  Contrary to characteristic polynomials, the norm continues
  to make sense for isogenies between different Drinfeld modules.
  In the example below, we recover the computation we did by hand
  in the SageMath Example \ref{sage:dualexample}.

  \begin{sagecolorbox}
  \begin{sagecommandline}

    sage: phi = DrinfeldModule(A, [z, z, 1])
    sage: u = phi.hom(tau + 1)
    sage: u
    sage: u.norm()

  \end{sagecommandline}
  \end{sagecolorbox}
\end{sageexplain}

As in the case of elliptic curves,
the norm of an isogeny $u$ is related to its kernel 
\[
  \ker u = \big\{\, x \in \Fs \mid u(x) = 0 \,\big\}.
\]
and, more precisely, to its Fitting ideal $|\ker u|$ (see
\cref{subsection:setting-and-notations} for the definition).

\begin{proposition}
  Let $u$ be an isogeny of Drinfeld modules defined over $(F, \gamma)$.
  \begin{enumerate}
  \item If $u$ is separable, then we have an equality of ideals
  \[\mathrm{norm}(u) A = |\ker u|.\]
  \item If $u$ is not separable, then 
  \[\mathrm{norm}(u) A = \p^{h/\deg\pp} \cdot |\ker u|,\]
  where $h$ is the $\tau$-valuation of $u$, and $\p$ is the
  $A$-characteristic of $(F, \gamma)$.
  \end{enumerate}
\end{proposition}

\begin{proof}[Sketch of the proof]
When $u$ is separable, the formula follows from
\cref{th:motive:torsion} applied with an ideal $\a$ annihilating 
$\ker u$.
The general case is deduced from the previous one by factoring
out as many times as possible $\tau^{\deg \p}$ in $u$ (see also
\cref{prop:facto-separable-inseparable} in the next section) and
showing independently that $\mathrm{norm}(\tau^{\deg \p}) = \p$.
We refer to \cite[Theorem~3.2]{CL23} for more details.
\end{proof}

\begin{proposition}
\label{prop:isogenie-duale-norme}
  Let $\phi$ and $\psi$ be two Drinfeld modules over $(F, \gamma)$, and $u:
  \phi \to \psi$ be an isogeny with norm $a$. Then there exists an isogeny
  $\hat u: \psi \to \phi$ such that $\hat u u = \phi_a$ and $u \hat u = \psi_a$.
\end{proposition}

\begin{proof}[Sketch of the proof]
The isogeny $\hat u$ is defined using \cref{th:motive:fullyfaithful}
by letting $\MM(\hat u)$ be the adjoint (that is, the transpose of 
the matrix of cofactors) of $\MM(u)$.
\end{proof}

The isogeny $\hat u$ is called the \emph{dual isogeny} of $u$.
Again, computing it is easily performed: one computes the norm and recovers
the dual isogeny via Ore Euclidean division.

\begin{sageexplain}
  We compute the dual isogeny of the isogeny $u$ of SageMath Example \ref{sageexplain:charpoly}
  and we recognize the isogeny $v$ we already introduced in the SageMath Example~\ref{sage:dualexample}.

  \begin{sagecolorbox}
  \begin{sagecommandline}

    sage: u.dual_isogeny()

  \end{sagecommandline}
  \end{sagecolorbox}
\end{sageexplain}

\begin{remark}

  The dual isogeny is the analogue to the classical \emph{dual isogeny} in the
  context of elliptic curves: if $\varphi: E \to E'$ is an isogeny with norm
  $n$ between two elliptic curves, there exists an isogeny $\hat{\varphi}: E'
  \to E$ that has norm $n$ such that $\hat{\varphi}\varphi$ (resp.~$\varphi
  \hat\varphi$) is the endomorphism of multiplication by $n$ on $E$
  (resp.~$E')$.

\end{remark}

We conclude this subsection by noticing that the existence of dual isogenies ensures that ``being isogenous'' is an equivalence relation.

\subsubsection{Structure of Hom spaces and endomorphism rings}
\label{subsec:hom-spaces}

Anderson motives also help to describe the structure of the
spaces of morphisms between two Drinfeld modules.

\begin{theorem}[{\cite[Theorem~3.4.1]{Pap23}}]
\label{theorem:structure-homset}
  Let $\phi$ and $\psi$ be two Drinfeld modules over a field $F$, both
  of rank $r$. Then, $\Hom(\phi, \psi)$ is a free $A$-module of rank
  at most $r^2$.
\end{theorem}

\begin{proof}[Sketch of the proof]
We know from \cref{th:motive:fullyfaithful} that 
$\Hom(\phi, \psi)$ is isomorphic to
$\Hom_{F[T], \tau}\big(\MM(\psi),\,\MM(\phi)\big)$, which itself
lives in $\Hom_{F[T]}\big(\MM(\psi),\,\MM(\phi)\big)$ which is a
free module of rank $r^2$ over $F[T]$.
In order to descend from $F[T]$ to $A = \Fq[T]$, the key observation is that the natural morphism
$$F \otimes_{\Fq} \Hom_{F[T], \tau}\big(\MM(\psi),\,\MM(\phi)\big)
\to \Hom_{F[T]}\big(\MM(\psi),\,\MM(\phi)\big)$$
is injective, which is a classical general result about
$\tau$-modules.
We refer to \cite[Proposition~3.4.6]{Pap23} for more details.
\end{proof}

\begin{remark}
\label{rem:computeHom}
Not only does the above proof elucidate the structure of $\Hom(\phi, \psi)$, but it also translates to efficient algorithmic methods for computing this hom space.
Indeed, representing $F[T]$-linear morphisms $\MM(\psi) \to \MM(\phi)$ by $r \times r$ matrices, we are reduced to solving the equation
\begin{equation}\label{eq:computeHom}
M \cdot \Mat\big(\tau_{\MM(\psi)}\big) 
= \Mat\big(\tau_{\MM(\phi)}\big) \cdot \tau(M),
\end{equation}
where $M \in F[T]^{r \times r}$ is the unknown and the two other
matrices are given by Equation~\eqref{eq:motive:matrixtau}. When $F$ is a finite field, it provides efficient algorithms for computing this Hom space.
We will give more details on this in \cref{computing_isogenies}.
\end{remark}
\begin{definition}
\label{def:algebra-of-endomorphisms}
Given a Drinfeld module $\phi$, we set
\begin{align*}
  \End^0(\phi) &= K \otimes_A \End(\phi),\\
  \End_{\Fs}^0(\phi) &= K \otimes_A \End_{\Fs}(\phi),
\end{align*}   
where we recall that $K = \Fq(T)$ is the fraction field of $A$.
\end{definition}

The existence of dual isogenies ensures that $\End^0(\phi)$ is a division
algebra; indeed, the inverse in $\End^0(\phi)$ of an isogeny $u : \phi \to
\phi$ having norm $a$ and dual $\hat u$ is simply $a^{-1} \hat u$. Moreover,
\cref{theorem:structure-homset} tells us that $\End^0(\phi)$ has
dimension at most $r^2$ over $K$, and that $\End(\phi)$ can be seen as an order
inside it. The same is true for $\End_{\Fs}^0(\phi)$. We will go in further
details on the structures of $\End^0(\phi)$ and $\End_{\Fs}^0(\phi)$ when $F$
is a finite field in \cref{subsection:endomorphism-ring}.

\section{Arithmetic aspects of Drinfeld modules}

\label{sec:arithmetics}

The study of elliptic curves follows quite different branches depending on the base we work on. Roughly speaking,
there are three main cases.
\begin{itemize}
\item Over $\C$, the main tools are the Weierstraß uniformisation theorem and analytic geometry.
\item Over finite fields, the Frobenius endomorphism plays a central role.
\item Over number fields, an elliptic curve can be regarded
as an elliptic curve over $\C$, but it can also
be reduced modulo primes to give elliptic curves over finite fields. This case thus takes advantage of the two previous ones.
\end{itemize}

In the theory of Drinfeld modules, a similar trichotomy occurs.
We have already discussed the analytic theory (see \cref{ssec:drinfeld:uniformization}), \ie Drinfeld modules over $\Cinf$.
In this section, we develop the theory in the two remaining cases: Drinfeld
modules over finite fields and Drinfeld modules
over function fields.

\subsection{Over finite fields}
\label{subsec:finite-fields}

We assume that $F$ is a finite extension of the finite field $\Fq$.
This implies that the morphism $\gamma: A \to F$ associated to the $A$-characteristic (see \cref{def:Afield}) is not
injective. We write $\pp=\Acar(F)$, which is a prime ideal. 
We denote by $\Fpp$ the quotient $A/\pp$, which is a finite field with $q^{\deg
\pp}$ elements, and we have a tower of extensions
\[\begin{tikzcd}[ampersand replacement=\&]
	\Fq \& \Fpp \& F \& \Fs
	\arrow["{\pi_\pp}", from=1-1, to=1-2]
	\arrow["\gamma"', curve={height=18pt}, from=1-1, to=1-3]
	\arrow["{\iota_\pp}", from=1-2, to=1-3]
	\arrow["{\iota_F}", from=1-3, to=1-4]
\end{tikzcd},\]
where $\iota_\pp$ and $\iota_F$ are injections. 
We also let $d$ denote the degree of $F$ over $\Fq$. It is then a multiple of $\deg \pp$.
Throughout this subsection, we also fix a rank-$r$ Drinfeld module 
$\phi: A \to F\{\tau\}$ defined by $\phi_T = z + g_1\tau + \dots + g_r \tau^r$
with $g_i \in F$.

\subsubsection{The Frobenius endomorphism and its characteristic polynomial}
\label{sssec:frobendomorphism}

We recall that the Ore polynomial $\tau^{\deg\pp}$ defines an isogeny with domain $\phi$, which is called the relative Frobenius of $\phi$ (see
\cref{def:relativeFrobenius}).
\begin{proposition}
    The isogeny $\tau^d$ is an endomorphism of $\phi$, which is called the \emph{Frobenius endomorphism} of $\phi$.
\end{proposition}
The endomorphism $\tau^d$ is generally denoted by $\pi$.

As we shall see, one invariant of primary importance attached to $\phi$ is the characteristic polynomial of $\pi$. With a slight abuse of notation, due to its central role, we will denote it by $\chi_{\phi}(X)$ instead of $\chi_{\pi}(X)$ (see \cref{def:charpoly}).

\begin{example}
\label{ex:chiFrob:carlitz}
We recall that the Carlitz module $c$ over $(F, \gamma)$ is defined by
$c_T = z + \tau$ where, as usual, $z = \gamma(T)$.
We have seen in \cref{ex:motive:carlitz} that the Anderson motive of 
$c$ is one-dimensional over $F[T]$, namely $\MM(c) = F[T]{\cdot}\ee$,
with the $\tau$-action explicitly given by the formula
$\tau_{\MM(c)} (\ee) = (T{-}z) \ee$.
Using the semi-linearity of $\tau_{\MM(c)}$, we deduce by induction
that
\[
  \tau_{\MM(c)}^n (\ee) = (T-z)(T-z^q) \cdots (T-z^{q^{n-1}}) \cdot \ee
\]
for all positive integers $n$. Therefore
\[
  \chi_{c}(X) = X - (T-z)(T-z^q) \cdots (T-z^{q^{d-1}}),
\]
where $d=[F:\Fq]$.
In particular, if $z$ generates $F$ over $\Fq$, we find that
$\chi_{c}(X) = X - \p$ where $\p$ is a generator of the 
$A$-characteristic of $(F, \gamma)$. More generally, we always
have $\chi_c(X) = X - \p^m$ with $m = [F : \Fq(z)]$.

We observe, in particular, that $\chi_c(1)$ is somehow related to 
the local factors defining the Carlitz zeta function (see 
Equation~\eqref{eq:zetaCEuler}). We will come back to this observation later on, when we will present the theory of $L$-series (see
\cref{sssec:arith:Lseries}).
\end{example}

\begin{sageexplain}
  We compute the characteristic polynomial of the Frobenius of the
  Carlitz module and observe that we indeed find the formula found in Example~\ref{ex:chiFrob:carlitz} above.

  \begin{sagecolorbox}
  \begin{sagecommandline}

    sage: c = CarlitzModule(A, F)
    sage: c.characteristic()
    sage: c.frobenius_charpoly()

  \end{sagecommandline}
  \end{sagecolorbox}

  Below is another example with a Drinfeld module of higher rank.
  We observe that the $T$-degrees of the coefficients increase
  slowly; this is a general phenomenon, as stated below in
  \cref{th:bounds-chipi}.

  \begin{sagecolorbox}
  \begin{sagecommandline}

    sage: phi = DrinfeldModule(A, [z, z^2, z^3, z^4, z^5, z^6])
    sage: chi = phi.frobenius_charpoly()
    sage: chi

  \end{sagecommandline}
  \end{sagecolorbox}
\end{sageexplain}

The three following theorems underline the importance of
$\chi_\phi$.
The first one is an analogue of a famous theorem on elliptic curves 
stating that the number of rational points is obtained by evaluating the 
characteristic polynomial of the Frobenius endomorphism at $1$ 
\cite[Chapter V, Theorem~2.3.1]{silverman_arithmetic_2009}.
In the case of Drinfeld modules, the role of the group of rational points is
played by the $A$-module $\phiact F$. Contrary to abelian varieties, the
underlying subset is always the same, that is, $F$ itself. However, its
structure as an $A$-module can vary.

\begin{theorem}[\protect{\cite[Theorem~5.1]{Gekele_1991}}]
\label{th:number-points}
  The Euler-Poincaré characteristic $|\phiact F|$ is the
  principal ideal generated by $\chi_\phi(1)$.
\end{theorem}

\begin{sageexplain}
  \cref{th:number-points} implies in particular that all
  elements in $\phiact F$ are of $\chi_\phi(1)$-torsion.
  We check it below with an example.

  \begin{sagecolorbox}
  \begin{sagecommandline}

    sage: a = chi(1)
    sage: a
    sage: [phi(a)(x) for x in F]

  \end{sagecommandline}
  \end{sagecolorbox}
\end{sageexplain}

\noindent
The second theorem can be interpreted as an analogue of Hasse's theorem
for elliptic curves \cite[Chapter V, Theorem~1.1]{silverman_arithmetic_2009}.

\begin{theorem}
\label{th:bounds-chipi}
  Write $\displaystyle{\chi_\phi(X) = \sum_{i=0}^r a_i X^r}$, with $a_0, \dots, a_r\in A=\Fq[T]$. For any positive integer $i < r$, we have
  \[
    \deg(a_i) \leqslant \frac{r-i}{r}\cdot d.
  \]
  In the particular case of the \emph{Frobenius norm}, that is, the coefficient
  $a_0$, we have an explicit formula: if $p \in A$ denotes the monic generator 
  of $\pp$, we have
  \[
    a_0 = (-1)^{rd - d - r} \frac{p^{d/\deg(\pp)}}{\mathrm{N}_{F/\Fq}(g_r)}.
  \]
\end{theorem}
It should be noted that, contrary to the case of elliptic curves, 
\cref{th:bounds-chipi} is quite easy to establish in the framework of 
Drinfeld modules, since it follows directly from writing down explicitly 
the matrix of $\MM(\tau^d)$ acting on $\MM(\phi)$ and estimating the 
degrees of its coefficients.
Besides, combining \cref{th:number-points} 
and \cref{th:bounds-chipi} gives the bound
\[
  \deg\big(|\phiact F| - \p\big) \leq \frac {\deg \p} 2
\]
for any Drinfeld module $\phi$ of rank $2$ of the form
$\phi_T = z + g_1 \tau + \tau^2$ with $F = \Fq(z)$.
This is the analogue of Weil's
bound on the number of rational points of a curve of genus $1$.

\begin{theorem}[{\cite[Theorem~4.3.2]{Pap23}}]
\label{th:isogenous-charpoly}
  Two Drinfeld modules over $(F, \gamma)$ are $F$-isogenous if, and only if, they
  have the same characteristic polynomial of the Frobenius endomorphism.
\end{theorem}

\subsubsection{Review of methods for computing the characteristic polynomial of
the Frobenius endomorphism}
\label{sec:review-frobenius-charpoly}

We now reflect on the computation of characteristic polynomials of
endomorphisms, with a focus on the Frobenius case. This problem is a staple of elliptic
curve and isogeny-based cryptography. For elliptic curves over a finite field,
the number of rational points is given by the trace of the Frobenius
endomorphism. The first deterministic algorithm to compute this
quantity was proposed by Schoof~\cite{schoof_elliptic_1985}. Schoof's algorithm works by computing
the trace modulo various distinct prime numbers with division polynomials and
recovering it via the Chinese remainder theorem.

A Drinfeld module analogue of this method was proposed, among others, by Musleh
and Schost~\cite{musleh_computing_2019, musleh_fast_2018, musleh_algorithms_2023}. They describe two versions: a deterministic method, as well as a Monte--Carlo method
that is more efficient. We stress that these only work for rank-$2$ Drinfeld
modules. This restriction allows us to directly transpose algorithms from the
classical case of elliptic curves.

Previous methods also include one by Gekeler, which involved solving a linear
system directly coming from the identity $\chi_\phi(\pi) = 0$
\cite{gekeler_frobenius_2008}, where $\pi$ is the Frobenius endomorphism of $\phi$. It was generalised and implemented by Musleh for
higher ranks in \cite[\S~4.4.1]{musleh_algorithms_2023}. On the other hand,
Narayanan proposed a method based on the minimal polynomial of sequences
\cite{narayanan_polynomial_2018}. Musleh and Schost expressed doubts on the
validity of some of the assumptions made by Narayanan, and suggested a
workaround in \cite[\S~5]{musleh_computing_2019}.

In \cite[\S~4]{CL23}, the authors proposed a method for the Frobenius
endomorphism in arbitrary rank. It is based on the observation
that the Frobenius
endomorphism is not only a remarkable element of $\phi$, but one of
$F\{\tau\}$ as well: setting $d = [F: \Fq]$, the element $\tau^d$ is central in $F\{\tau\}$, and
the center of $F\{\tau\}$ is given by $\Fq\{\tau\}$. One can then show that the
characteristic polynomial of $\pi$ is the reduced characteristic polynomial of
$\phi_T \in F\{\tau\}$, where $F\{\tau\}$ is embedded in a central simple algebra
of degree $d^2$ over its centre.

The methods mentioned in \cref{remark:calcul-charpoly} also provide a way to
compute the characteristic polynomial of the Frobenius endomorphism using Anderson motives.
Interestingly, with their optimisations and subsequent variants, all the methods described above have their own benefits, depending on the relative sizes of the input
parameters. They are compared in \cite[Figure~1]{CL23}.

One takeaway message from these computations is that while elliptic curves
shed a familiar light on Drinfeld modules, understanding Drinfeld modules as a theory on
their own was the key to developing algorithms in much greater generality than the
restricted case of the Frobenius endomorphism in rank $2$. 

Furthermore, those same Ore polynomials give a structure of a central simple algebra, allowing for the computation of the characteristic polynomial of the
Frobenius. These objects bypass the need for the computation of torsion elements
and Tate modules, thanks to \cref{th:motive:torsion}.

For more extensive studies on the subject, we refer to Section 4.4 (Frobenius
endomorphism in the rank-$2$ case only) or Section 6.4 (general case) of
\cite{musleh_algorithms_2023} or \cite[\S~4]{musleh_computing_2019}. We refer
to \cite[Appendix A]{CL23} for direct comparison between asymptotic
complexities, and to \cite[\S~4.4.2, 5.4.2, 6.4.1]{leudiere_computing_2024} and
\cite{ayotte_drinfeld_2023} for reproducible benchmarks.

\subsubsection{Computing isogenies}
\label{computing_isogenies}

In contrast with the situation of elliptic curves, isogenies of Drinfeld modules over a finite field can be computed in polynomial time. Indeed, as highlighted by Equation~\eqref{eq:computeHom}, computing $\Hom(\phi,\psi)$ reduces to solving a linear system over $A$, which is finite dimensional when $F$ is a finite field. Wesolowski \cite{W22} and Musleh \cite[\S~7.3]{musleh_algorithms_2023} develop this idea, leading to the following theorem.

\begin{theorem}[\protect{\cite[Theorem 7.3.1]{musleh_algorithms_2023}}]
    Let $\phi, \psi$ be two Drinfeld modules of rank $r$. There exist algorithms for computing
    \begin{enumerate}
        \item an $\mathbb{F}_q$-basis of isogenies of degree at most $N$ in $\Hom(\phi,\psi)$ with quasi-cubic complexity in $N$ and in $d$, and quasi-linear complexity in $r$,
        \item an $\mathbb{F}_q[\tau^d]$-basis of $\Hom(\phi,\psi)$ with quasi-sextic complexity in $d$, and quasi-linear complexity in $r$,
        \item an $A$-basis of $\Hom(\phi,\psi)$ with quasi-cubic complexity in $d$ and in $r$.
    \end{enumerate} 
\end{theorem}

\begin{remark}
     Musleh \cite{musleh_algorithms_2023} gives, in fact, slightly better complexity involving the exponent of matrix multiplication.
\end{remark}

\begin{sageexplain}
  We compute the hom space between two Drinfeld modules $\phi$
  and $\psi$. In this case, we further observe that the hom space
  contains an isogeny defined by a constant Ore polynomial, so $\phi$
  and $\psi$ are isomorphic.

  \begin{sagecolorbox}
  \begin{sagecommandline}

    sage: phi = DrinfeldModule(A, [z, z, 1])
    sage: psi = DrinfeldModule(A, [z, 1-z, 1])
    sage: H = Hom(phi, psi)
    sage: H.basis()

  \end{sagecommandline}
  \end{sagecolorbox}

  \begin{sagecolorbox}
  \begin{sagecommandline}

    sage: H.basis_over_frobenius()

  \end{sagecommandline}
  \end{sagecolorbox}

  \begin{sagecolorbox}
  \begin{sagecommandline}

    sage: H.basis(degree=5)

  \end{sagecommandline}
  \end{sagecolorbox}

\end{sageexplain}

\subsubsection{Supersingularity}
\label{subsubsec:super-ordi}

Let $p$ again denote the monic generator of $\pp=\ker \gamma$. Recall from \cref{rk:height}  and \cref{definition:tate-module} that the Tate module of $\phi$ at $\pp$ has rank $r -
\h(\phi)$ over $A_\pp$, where $r$ is the rank of $\phi$, and $\h(\phi)$ is the Frobenius height.

\begin{definition}\label{Frobenius height}
  One says that $\phi$ is \emph{ordinary} when $\h(\phi) = 1$, and that $\phi$
  is \emph{supersingular} when $\h(\phi)=r$.
\end{definition}

\begin{theorem}[{\cite[Corollary~4.2.14, Theorem~4.4.1]{Pap23}}]
\label{prop_ordin_frobenius_trace}
  The following properties are all equivalent:
  \begin{itemize}
  \setlength\itemsep{0em}
    \item $\phi$ is supersingular;
    \item $\phi[\p]$ is trivial;
    \item $\phi_p$ is purely inseparable;
    \item $\End_{\Fs}(\phi)$ has maximal $A$-rank, \ie $r^2$;
    \item $\chi_\phi \equiv X^r \pmod {\p}$.
  \end{itemize}
\end{theorem}

\begin{sageexplain}
\label{sage:supersingular-end}
We illustrate the fact that 
\cref{prop_ordin_frobenius_trace} does not hold if one replaces 
$\End_{\Fs}(\phi)$ by $\End(\phi)$ in the fourth item.

  \begin{sagecolorbox}
  \begin{sagecommandline}

    sage: phi = DrinfeldModule(A, [1, 0, z])
    sage: phi.is_supersingular()

  \end{sagecommandline}
  \end{sagecolorbox}

We first compute $\End(\phi)$ and see that it has dimension $2$.

  \begin{sagecolorbox}
  \begin{sagecommandline}

    sage: End(phi).basis()

  \end{sagecommandline}
  \end{sagecolorbox}

However, more isogenies do exist over an extension. In our example,
they show up over $L = F[\sqrt[8] z] = \mathbb F_{7^{16}}$.

  \begin{sagecolorbox}
  \begin{sagecommandline}

    sage: L = F.extension(8)
    sage: phi_L = phi.change_A_field(L)
    sage: End(phi_L).basis()

  \end{sagecommandline}
  \end{sagecolorbox}

\end{sageexplain}

\begin{remark}
\label{rem:hasseinvariant}
In rank two, the third item of \cref{prop_ordin_frobenius_trace}
shows that $\phi$ is supersingular if, and only if, the coefficient of
$\tau^{\deg(p)}$ in $\phi_p$ vanishes.
By analogy with the theory of elliptic curves, this coefficient is sometimes called the \emph{Hasse invariant} of $\phi$.
In \cite[Equation~3.6 and Proposition~3.7]{gekeler_frobenius_2008}, Gekeler observed that the Hasse invariant can be computed
\emph{via} a recursive procedure as follows. Writing $\phi_T = z + g \tau + \Delta \tau^2$, we set
$r_{\phi, 0} = 1$, $r_{\phi, 1} = g$, and for any integer $k \geqslant 2$, we recursively define
\[
  r_{\phi, k+2} =   g^{q^{k+1}} r_{\phi, k+1}
                  - \big(z^{q^{k+1}} - z\big) \Delta^{q^k} r_{\phi, k}.
\]
Then the Hasse invariant of $\phi$ is $r_{\phi, \deg(p)}$.
\end{remark}

In any rank, the characteristic polynomial of a supersingular Drinfeld 
module $\phi$ has a very particular shape: over $\overline{\F}_q[T,X]$, 
we have the factorization
$$\chi_\phi(X) = \prod_{i=1}^{r/a} \big(X^a - c_i p^b\big)$$
where $\frac a b$ is the irreducible form of the fraction $\frac r {d \deg(\p)}$
and the $c_i$ lie in $\overline{\F}_q$.
After \cref{th:isogenous-charpoly}, this implies in particular
that two supersingular Drinfeld modules of the same rank become isogenous
over $\Fs$.

\begin{sageexplain}
We continue the previous example and check that the characteristic
polynomial of $\phi$ has the expected form (here $a = b = 1$, 
$p = T - 1$, $c_1 = 5z$ and $c_2 = 2z + 5$).

  \begin{sagecolorbox}
  \begin{sagecommandline}

    sage: chi = phi.frobenius_charpoly()
    sage: chi = chi.change_ring(F['T'])  # we extend scalars to $F$
    sage: chi.factor()

  \end{sagecommandline}
  \end{sagecolorbox}

We pick another supersingular Drinfeld module which is not
isogenous to $\phi$.

  \begin{sagecolorbox}
  \begin{sagecommandline}

    sage: psi = DrinfeldModule(A, [1, 0, z^2])
    sage: phi.is_isogenous(psi)

  \end{sagecommandline}
  \end{sagecolorbox}

But we check that they become isogenous over the extension of $F$ of
degree $q^2 - 1 = 48$.

  \begin{sagecolorbox}
  \begin{sagecommandline}

    sage: L = F.extension(48)
    sage: phi_L = phi.change_A_field(L)
    sage: psi_L = psi.change_A_field(L)
    sage: phi_L.is_isogenous(psi_L)

  \end{sagecommandline}
  \end{sagecolorbox}

\end{sageexplain}

To conclude, we mention the following proposition, which implies in particular that there are only finitely many supersingular Drinfeld modules of a given rank over $\Fs$.

\begin{proposition}[{\cite[Proposition~4.2]{Gekele_1991}}]
  Let $\phi : A \to \Fs\{\tau\}$ be a supersingular Drinfeld module of
  rank $r$. Then $\phi$ is $\Fs$-isomorphic to a Drinfeld module defined
  over a degree $r$-extension of $\Fpp$.
\end{proposition}

\subsubsection{The endomorphism ring in rank two}
\label{subsection:endomorphism-ring}

Assuming rank $2$, information on supersingularity can also be read on
the structure of the endomorphism ring. An elliptic curve over a finite field
is ordinary if, and only if, its endomorphism ring over the algebraic closure is commutative
\cite[Chapter V, \S~3]{silverman_arithmetic_2009}. In that case, the endomorphism ring of
the elliptic curve is an order in a quadratic number field. Otherwise,
it is a maximal order in a quaternion algebra. This is also true for rank-$2$ Drinfeld modules.
Recall that $K=\mathrm{Frac}(A)$ and $\End^0(\phi)=K \otimes_A
\End(\phi)$ (\cref{def:algebra-of-endomorphisms}).

\begin{theorem}[\protect{\cite[Theorem~6.4.2]{caranay_computing_2018}}]
\label{theo:endordinary}
Let $\phi : A \to F\{\tau\}$ be an ordinary Drinfeld module of rank $2$.
Then $\End^0(\phi)$ is an imaginary quadratic function field generated by
the Frobenius endomorphism of $\phi$. In other words, we have an
isomorphism of $A$-algebras
    \[
      \End^0(\phi) \simeq K[X]/(\chi_\phi(X)).
    \]
\end{theorem}

After Theorem~\ref{theo:endordinary}, we see that, in the ordinary case,
$\End(\phi)$ appears as an order in an imaginary quadratic function field.
We recall that such fields contain a unique maximal order: their
ring of integers, which is also the unique order that is a Dedekind domain.
If we denote it by $\mathcal O_{\End^0(\phi)}$, the other orders are exactly
the sets of the form $\mathcal O = A + f\mathcal{O}_{\End^0(\phi)}$ for $f$ in $A$.
The polynomial $f$ is called the \emph{conductor} of $\mathcal O$. It is such that
$\mathcal O$ has index $q^{\deg(f)}$ in $\mathcal{O}_{\End^0(\phi)}$.

\begin{sageexplain}
\label{sage:endordinary}
We consider the following ordinary Drinfeld module of rank $2$

  \begin{sagecolorbox}
  \begin{sagecommandline}

    sage: phi = DrinfeldModule(A, [z, z, 1])
    sage: phi
    sage: phi.is_ordinary()

  \end{sagecommandline}
  \end{sagecolorbox}

and compute its endomorphism ring

  \begin{sagecolorbox}
  \begin{sagecommandline}

    sage: End(phi).basis()

  \end{sagecommandline}
  \end{sagecolorbox}

Noticing that the second endomorphism of the outputted basis is $T - \pi$,
we conclude that $\End(\phi)$ is generated by the Frobenius, that is
$$\End(\phi) \simeq A[X]/(\chi_\phi(X)).$$
This is a more precise statement than~\cref{theo:endordinary}
which only gives this isomorphism after scalar extension to~$K$.
\medskip

In our example, one can check in addition that the discriminant
of $\chi_\phi(X)$ is squarefree, implying that $\End(\phi)$ is the
maximal order of $\End^0(\phi)$.

  \begin{sagecolorbox}
  \begin{sagecommandline}

    sage: chi = phi.frobenius_charpoly()
    sage: chi.discriminant()

  \end{sagecommandline}
  \end{sagecolorbox}

\end{sageexplain}

On the contrary, when $\phi$ is supersingular, 
$\End^0(\phi)$ continues to contain the quadratic field generated by the Frobenius endomorphism, but it might be strictly larger.
It actually has degree $2$ or $4$
In the latter case, it is a quaternion algebra over $K$, \ie a central
simple $K$-algebra of dimension~$4$, and $\End(\phi)$ is an order
in $\End^0(\phi)$.
SageMath Example~\ref{sage:supersingular-end} illustrates this situation.

In any case, while $\End(\phi)$ might be commutative and $\phi$ be supersingular, the following holds as a consequence of~\cref{prop_ordin_frobenius_trace} and~\cref{theo:endordinary}.

\begin{proposition}
  A rank-$2$ Drinfeld module $\phi$ is supersingular if, and only if,
  $\End_{\Fs}(\phi)$ is noncommutative.
\end{proposition}

This study is very theoretical. For explicit computations, we have to turn to a
completely different point of view, by computing bases with linear algebra
methods (see \cref{computing_isogenies}).

\subsubsection{The action of the Picard group} 
\label{computing-group-action}

We build on the classification established in
\cref{subsection:endomorphism-ring} to present a situation in which the
endomorphism and its class group are described in terms of hyperelliptic
curves.
Ordinary elliptic curves over finite fields and (some) ordinary rank-$2$ Drinfeld
modules over finite fields share a common property: the class group of the
endomorphism ring acts freely and transitively on a set of isomorphism classes. Such
actions allow us to realise isogenies of Drinfeld modules as ideal classes, whose
description can be made very explicit. We now explain how to compute a
particular instance of this group action, using imaginary hyperelliptic curves.

As previously, let us still assume that $F$ is a finite field, and that $\phi$ is a rank-$2$ Drinfeld module over
$(F, \gamma)$. Let $\chi_\phi \in A[X]$ be the characteristic
 polynomial of the Frobenius endomorphism. We have $\chi_{\phi}(X) = X^2 + aX + b$
where, thanks to \cref{th:bounds-chipi}, we know that the polynomials 
$a, \: b \in A=\Fq[T]$ verify $\deg(a) \leqslant d/2$ and $\deg(b) = d$.
As $a$ and $b$ both lie in $\Fq[T]$, the polynomial $\chi_{\phi}$ can be seen as a bivariate
polynomial in $T$ and $X$. As such, it defines an algebraic variety. We make the
following assumptions.

\begin{itemize}
  \item We assume that $\phi$ is ordinary. This implies that $a$ is nonzero
    (see \cref{prop_ordin_frobenius_trace}), so that $\chi_{\phi}$
    defines a hyperelliptic curve that we denote by $\H$
  \item We assume $\H$ to be smooth, and $d$ to be odd. In that case, $\H$ is
    \emph{imaginary}, with genus $g = \frac{d-1}{2}$.
\end{itemize}

It turns out that these hypotheses fully determine the $F$-endomorphism ring of
$\phi$, and that of all Drinfeld modules $F$-isogenous to $\phi$.
Indeed, letting $A_\H = A[X]/(\chi_\phi(X))$ denote the coordinate ring of $\H$,
the map
\begin{align*}
  A[X] &\to \End(\phi) \\
  P(T, X) &\mapsto P(\phi_T, \pi)
\end{align*}
factors through $A_\H$, giving rise to an injective morphism $A_\H \to \End(\phi)$. 
Since $A_\H$ is a Dedekind domain, and as such, it is the maximal order 
of the quadratic function field it lives in, we conclude that
$\End(\phi) \simeq A_\H$. 
Crucially, this explicit ring isomorphism yields an
explicit group isomorphism
\[
  \Pic^0(\H) \simeq \Cl(\End(\phi)),
\]
where $\Pic^0(\H)$ is the degree--$0$ Picard group of $\H$.

\paragraph{Abstract definition of the action.}

For an ideal $\mathfrak a$ of the endomorphism ring $\End(\phi)$, we
define $u_{\mathfrak a} = \mathrm{rgcd}(\{a \in \mathfrak a\})$,
the right-gcd of the Ore polynomials that live in
$\mathfrak a$. 
Let $\psi$ be a Drinfeld module $F$-isogenous to $\phi$.
By \cref{proposition:rgcd-isogeny}, $u_{\mathfrak a}$ defines an isogeny
from $\psi$ to a Drinfeld module that we denote by $\psi^{\mathfrak a}$, and we write
\[
  \mathfrak a \ast \psi = \psi^{\mathfrak a}.
\]
If $\mathfrak a$ is a principal ideal, meaning that it is generated by an
endomorphism of $\psi$, then $u_{\mathfrak a}$ is one of the generators of
$\mathfrak a$, and $\psi$ equals $\mathfrak a \ast \psi$.

\begin{sageexplain}
\label{sage:action}
We illustrate the above constructions with the Drinfeld module
$\phi$ of SageMath Example~\ref{sage:endordinary}.
Unfortunately, manipulations of ideals in $\End(\phi)$ are not yet implemented in SageMath.
We overcome this issue as follows: we reconstruct the ring $\End(\phi)$ 
as the quotient $E = A[X] / (\chi_\phi(X))$ and define a function
\texttt{endomorphism} to convert elements of $E$ into actual endomorphisms.

  \begin{sagecolorbox}
  \begin{sagecommandline}

    sage: E.<X> = A.extension(chi)
    sage: def endomorphism(elt, psi):  # $\psi$ is a Drinfeld module isogenous to $\phi$
    ....:     pi = psi.frobenius_endomorphism()
    ....:     return E(elt).lift()(pi)

  \end{sagecommandline}
  \end{sagecolorbox}

Now, we can write a function \texttt{action} by simply following the definition.

  \begin{sagecolorbox}
  \begin{sagecommandline}

    sage: def action(a, psi):
    ....:     u = psi.hom(0)
    ....:     for elt in a.gens():
    ....:         v = endomorphism(elt, psi)
    ....:         u = u.right_gcd(v)
    ....:     return u.codomain()

  \end{sagecommandline}
  \end{sagecolorbox}

We consider the ideal $\a = \left<T + 4, X + 3\right>$ and compute
$\psi = \a \ast \phi$.

  \begin{sagecolorbox}
  \begin{sagecommandline}

    sage: a = E.ideal([T+4, X+3])
    sage: psi = action(a, phi)
    sage: psi

  \end{sagecommandline}
  \end{sagecolorbox}

We notice that $\a^2$ is the principal ideal generated by $T + 4$.
Therefore it acts trivially on $\phi$, meaning that $\a \ast \psi
= \a^2 \ast \phi = \phi$. We check this below.

  \begin{sagecolorbox}
  \begin{sagecommandline}

    sage: action(a, psi)
    sage: action(a, psi) == phi

  \end{sagecommandline}
  \end{sagecolorbox}

\end{sageexplain}

It follows from what precedes that the map $\ast$ defines a group action 
of $\Cl(\End(\phi))$ on the set of isomorphism classes of Drinfeld modules
that are isogenous to $\phi$. In a slight abuse of notation, the group
action of $\Cl(\End(\phi))$ will still be denoted by $\ast$.
The latter is free and transitive. This can be proven directly by
means of tedious computations. A more conceptual approach, which is out of the
scope of this document, consists in ``reducing modulo $\mathfrak p$'' a similar
group action (known to be free and transitive) that exists for Drinfeld modules
over $\Cinf$ and in the context of the class field theory of function fields (see
\cite[Theorem~9.3]{hayes_brief_2011} or \cite[\S~2.3]{leudiere_computing_2024}). 

\paragraph{Practical computation of the group action.}

We saw in SageMath Example~\ref{sage:action} that computing the action
of ideals of $\End(\phi)$ can be easily implemented.
To go further and compute the action of $\Cl(\End(\phi))$, it 
only remains to find an efficient representation for the elements of
$\Cl(\End(\phi))$. Given that $\End(\phi)$ is isomorphic to $A_\H$, and
that $A_\H$ is the coordinate ring of an imaginary hyperelliptic curve, we have an explicit group isomorphism
\[
  \Pic^0(\H) \simeq \Cl(\End(\phi)),
\]
where $\Pic^0(\H)$ is the group of divisors of $\H$ of degree $0$ up to
rational equivalence. Besides, the elements of $\Pic^0(\H)$ can be represented
by \emph{Mumford coordinates} \cite[Theorem~14.5]{cohen_handbook_2012},
\ie pairs of polynomials $(u, v) \in A^2$ that verify
\[
  \deg(u) < \deg(v) \leqslant g,
\]
where $g=\frac{d-1}{2}$ is the genus of the curve $\mathcal{H}$.
The element represented by $(u, v)$ is the ideal class of $\langle u,
X-v\rangle$. In terms of $\End(\phi)$, this means that $(u, v)$ represents
the ideal class generated by the endomorphisms $\phi_u$ and $\pi - \phi_v$.

Going back to the definition of the group action $\ast$, Mumford coordinates
allow us to represent elements of the acting group by two explicit generators.
If $(u, v)$ are the Mumford coordinates representing the class of an ideal
$\a$, computing $\a \ast \psi$ simply
amounts to computing the right-gcd of the Ore polynomials $\psi_u$ and $\pi -
\psi_v$. This can be done, for example, with a variant of the Euclidean algorithm,
as discussed in \S\ref{ssec:drinfeld:def}. This means that computing the group
action merely amounts to computing a right-gcd of two Ore polynomials and a
right-Euclidean division of Ore polynomials. Explicit complexity statements, as
well as an implementation, are presented in \cite{leudiere_computing_2024}.

\paragraph{Comparison with elliptic curves.}
The ease with which the group action $\ast$ is computed suggests a profound
fracture with what happens for elliptic curves. In the case of elliptic curves,
the group action is defined in terms of kernels of isogenies. If $E$ is an ordinary elliptic
curve over the finite field $\Fq$, the isogenies for the group action are
defined as follows. Let $I$ be an ideal in $\End(E)$, and  consider
\[
  V_I = \bigcap_{f \in I}\ker(f).
\]
The kernels are not restricted to rational points. Therefore, $V_I$ is a
subgroup of the group of points of $E$ (but not necessarily $E(\Fq)$). Consequently, there
exists an elliptic curve $E^I$, as well as an isogeny $\iota_I$ from $E$ to $E^I$,
whose kernel is exactly $V_I$.

In the end, computing the group action is, as of the time of writing this
survey, rather complex~\cite{de_feo_towards_2018}. This is because the computation requires manipulating torsion points that live in possibly
large extensions of the base field. With this in mind, we suggest two reasons to explain why the computation of the group action is easy in the case
of Drinfeld modules, but not in the case of elliptic curves.

\begin{enumerate}

  \item The first reason is the manipulation of Ore polynomials (objects
    defined with information contained in $F$), rather than their kernels
    (objects that may live in large extensions of $F$). The existence of Ore
    polynomials as a latent space to the theory of Drinfeld modules allows us to take advantage of efficient Ore polynomial arithmetics in virtually all
    computational aspects of Drinfeld modules. In the case of manipulating
    kernels of isogenies of Drinfeld modules, this is compatible with the
    arithmetic of Ore polynomials, in the sense that the intersection of
    kernels of a family of Ore polynomials is the kernel of the right-gcd of
    the family.

  \item The second reason is related to the arithmetic of function fields. As
    we have seen, the characteristic polynomial of the Frobenius endomorphism,
    in the case of Drinfeld modules, defines an algebraic variety. Under our
    assumptions, this variety was an imaginary hyperelliptic curve, which
    allowed us to manipulate the class group of the endomorphism ring of a
    Drinfeld module as the degree-$0$ Picard group of the curve. This is a
    direct example of the benefits of using function field arithmetics.

\end{enumerate}

\begin{remark}

  We now explain how the group action presented here is an instance of a 
  more general construction. In the introduction, we stated that Drinfeld
  modules can be defined over a given ring $A'$ of functions on a curve over a
  finite field. We also add hypotheses on the curve to make $A'$ a Dedekind
  ring. Our case $A = \Fq[T]$ corresponds to picking the curve $\mathbb
  P^1_{\Fq}$. But more generally, one can study Drinfeld modules over $A'$,
  which are called \emph{Drinfeld $A'$-modules}. In that framework, it is a
  classical result that the class group of $A'$ acts freely and transitively on
  the set of isomorphism classes of rank-$1$ Drinfeld $A'$-modules defined over
  $\Cinf$. As shorthand, we are going to call that group action
  $\mathrm{GA}(A')$.

  In our case, the class group of $\Fq[T]$ is trivial, and all Drinfeld
  $A$-modules with rank one are isomorphic. One can prove that there exists an
  equivalence of categories between rank-$1$ Drinfeld $A_\H$-modules over
  $(F,\gamma)$, and rank-$2$ Drinfeld $A$-modules on $(F,\gamma)$ whose endomorphism ring is $A_\H$. Using this correspondence, we recover the group
  action computed in this section (the class group of $A_\H$ acts on the
  isomorphism classes of rank-$2$ Drinfeld modules whose endomorphism ring is
  $A_\H$) by ``reducing'' $\mathrm{GA}(A_\H)$ modulo a prime ideal.

  The general group action is very explicit: for an ideal $\mathfrak a$ of $A'$
  and a Drinfeld $A$'-module $\phi$, the corresponding isogeny is
  $\rgcd(\{\phi_a: a \in \mathfrak a\})$.

\end{remark}

\subsection{Over global function fields}
\label{ssec:over-function-fields}

We now consider the case where the base field is $F = \Fq(z)$, where $z$
is a formal variable, and $\gamma : A \to F$ is the ring homomorphism
taking $T$ to $z$.
Although $F$ is of course isomorphic to $\Fq(T)$ itself, it will be 
quite important in what follows to use a different variable name on 
$F$ to avoid confusion in notation, especially when we will consider
Anderson motives.

This case is undoubtedly closely related to that of finite fields,
since any Drinfeld module over $K$ can be reduced modulo almost all places $\p$ of 
$K$, giving rise to a Drinfeld module over a finite field. On the 
other hand, it is also related to the analytic theory (see \cref{ssec:drinfeld:uniformization})
since, given that 
$F$ embeds into $\Cinf$, a Drinfeld module over $F$ can be viewed as a 
Drinfeld module over $\Cinf$.

\subsubsection{\texorpdfstring{$L$}{L}-series}
\label{sssec:arith:Lseries}

We start with a Drinfeld module $\phi : A \to F\{\tau\}$
and write
$$\phi_T = z + g_1 \tau + g_2 \tau^2 + \cdots + g_r \tau^r$$
with $g_i \in F$ for all $i$. For almost all places $\p$ of $F$, it
makes sense to reduce all the $g_i$, and thus $\phi$, modulo $\p$.
We let $\phi \bmod \p : A \to \F_\p\{\tau\}$ be this reduction. Here 
$\F_\p$ denotes the residue field at $\p$, which is a finite extension of $\Fq$, hence a finite field.
One can therefore consider the characteristic polynomial of its
Frobenius endomorphism: we denote it by $\chi_{\phi \bmod \p}(X)$.

The $L$-series of $\phi$ is built by putting together all these 
characteristic polynomials. For simplicity, we assume that all the 
$g_i$ lie in $A$. Under this extra 
assumption, the reduction $\phi \bmod \p$ is well defined for 
all places $\p$ and so $\chi_{\phi \bmod \p}(X)$ is.
The $L$-series is then given by the following infinite product:
$$L(\phi; X) = \prod_{\p} 
  \frac{\chi_{\phi \bmod \p}(0)}{\chi_{\phi \bmod \p}(X^{\deg \p})}.$$
The numerator in the fraction above is a normalisation factor: it
ensures that all the factors have constant coefficient $1$. The
convergence of the series holds because there is 
only a finite number of places $\p$ of degree bounded by any given
constant. Indeed, the two previous observations imply that 
$L(\phi; X) \bmod X^n$ is given by a \emph{finite} product for
all $n$; hence the convergence in $K[[X]]$.

One proves that $L(\phi; X)$ is not merely a formal series, but that it moreover converges on $\Cinf$; in other words, it defines an analytic entire function. It turns out that, similarly to
the number field case, the value $L(\phi; 1)$ and, more generally
the Taylor expansion of $L(\phi; X)$ at $X = 1$, encode a lot of
arithmetic information.
The example below is a first striking illustration of this yoga.

\begin{example}
\label{ex:Lseries:carlitz}
We compute the $L$-series of the Carlitz module $c$ defined by
$c_T = z + \tau$.
By Example~\ref{ex:chiFrob:carlitz}, we know that
$\chi_{c \bmod \p}(X) = X - \p$. Plugging this into the definition
of the $L$-series, we find
$$L(c; X) = \prod_{\p} \frac{\p}{X^{\deg \p} - \p}
  = \prod_{\p} \frac{1}{1 - X^{\deg \p} \p^{-1}}$$
where the product runs over all places of $F$, that are all
irreducible monic polynomials over $F$.
We now observe that the above formula looks very similar to the Euler product of Equation~\eqref{eq:zetaCEuler} defining the Carlitz zeta function, and indeed, we have the relation
$L(c; 1) = \zeta_C(1)$.
\end{example}

The previous example suggests incorporating the variable $s$ in the general definition of the $L$-series as follows:
\begin{equation}
\label{eq:arith:defLseries}
L(\phi; X, s) = \prod_{\p} 
  \frac{\chi_{\phi \bmod \p}(0)}{\chi_{\phi \bmod \p}(X^{\deg \p}\p^{-s})}.
\end{equation}
Indeed, with this definition, the relation $L(c; 1, s) = \zeta_C(s+1)$
now holds for all $s \in \Z$.

\paragraph{Taelman's class formula.}

Beyond the example of the Carlitz module, Taelman showed 
in~\cite{Taelman12} that, for any Drinfeld module $\phi : A \to 
F\{\tau\}$, the value $L(\phi; 1)$ has a wonderful arithmetic 
interpretation in the spirit of the classical class number formula.
We recall briefly that this formula relates the special value of the Dedekind zeta 
function of a number field to several of its invariants of arithmetical nature:
the discriminant, the group of roots of unity, the class number and the regulator.
We refer to \cite[Chapter~VII, Corollary~5.11]{Neukirch99} for a precise statement and a proof of this formula.

Taelman's formula involves function field analogues of the class number
and the regulator.
Both of them are defined by means of the exponential function $e_\phi$
we have introduced in \cref{ssec:drinfeld:uniformization}. 
Let us recall briefly that $e_\phi$ is an Ore series in 
$F\{\!\{\tau\}\!\}$. Moreover, letting $F_\infty \coloneqq \Fq(\!(1/z)\!)$, the restriction of $e_\phi$ to $K_\infty$ still defines
an analytic function $e_\phi : F_\infty \to F_\infty$ satisfying the 
functional equation
$$\forall a \in A, \quad \forall x \in F_\infty, \quad e_\phi(\gamma(a)x) = 
\phi_a(e_\phi(x)).$$ 
In other words, it defines an $A$-linear map $F_\infty \to 
\phiact{F_\infty}$ where the structure of $A$-module on the 
domain is induced by $\gamma$ and we recall that $\phiact{F_\infty}$ 
is $F_\infty$ endowed with its structure of $A$-module coming from
$\phi$.
The \emph{class module} of~$\phi$ is defined as the quotient
$A$-module
$$H_{\phi} := \frac{\phiact{F_\infty}}{\phiact{R} + e_\phi(F_\infty)}$$
where we have set $R := \Fq[z] \subset F$. Taelman~\cite[Proposition~5]{Taelman12} proves that $H_{\phi}$ is a finite 
$A$-module, and thus we can consider its Euler--Poincaré characteristic
$|H_{\phi}|$ (see \cref{subsection:setting-and-notations}).

To define the analogue of the regulator, we introduce the
module of \emph{Taelman units}:
$$U_{\phi} := e_\phi^{-1}(R) = \left\{ x \in F_\infty \mid e_\phi(x) \in R \right\}.$$
Taelman~\cite[Theorem~1]{Taelman12} proves that $U_{\phi}$ is an $R$-line in $F_\infty$
and we denote by $u_{\phi} \in K_\infty$ an element such that $\gamma(u_{\phi})$ generates
$U_{\phi}$. Taelman's theorem now reads as follows.

\begin{theorem}[Taelman's class formula]
\label{theo:taelman}
For all Drinfeld modules $\phi : A \to F\{\tau\}$, we have
$$L(\phi; 1) A = u_{\phi} \cdot |H_{\phi}|.$$
\end{theorem}

\begin{example}
When $\phi$ is the Carlitz module, 
the class module $H_{\phi}$ is trivial, so that we have $|H_{\phi}| = A$.
We deduce from \cref{theo:taelman} that the Taelman lattice 
$U_{\phi}$ is generated by $L(\phi; 1) = \zeta_C(1)$. 
Coming back to the definition, we conclude that the image of
$\zeta_C(1)$ under the exponential map $e_\phi$ is a polynomial;
it is actually the constant polynomial~$1$, \ie we have
$e_\phi\big(\zeta_C(1)\big) = 1$.
Taking care of domains of convergence, we deduce that
$\zeta_C(1) = \ell_\phi(1)$ where $\ell_\phi$ is the logarithm of
the Carlitz module.

\begin{sageexplain} 
We can check what precedes in SageMath as follows.
\begin{sagecolorbox}
\begin{sagecommandline}

  sage: phi = CarlitzModule(A)
  sage: phi.class_polynomial()

\end{sagecommandline}
\end{sagecolorbox}

\begin{sagecolorbox}
\begin{sagecommandline}

  sage: lC = phi.logarithm(prec=100)
  sage: lC
  sage: lC = lC.polynomial()  # we remove the $O(\cdot)$ to be able to evaluate
  sage: lC(1) + O(1/T^40)

\end{sagecommandline}
\end{sagecolorbox}

According to Taelman's theorem,
we observe that the previous value agrees with the value at $s = 1$
of the Carlitz zeta function we computed in the SageMath Example~\ref{sage:carlitzZeta}.
\end{sageexplain}

\end{example}

\paragraph{Anderson's formula.}

Taelman's formula and many achievements on $L$-series of Drinfeld
modules are corollaries of a beautiful formula due to 
Anderson~\cite{A20} which, roughly speaking, expresses $L(\phi;X)$ 
as the cocharacteristic polynomial of a \emph{single}  operator 
$\sigma$ as follows:
\begin{equation}
\label{eq:andersonformula}
  L(\phi; X) = \det{}_{K[X]} \big(\text{id} - X \sigma\big).
\end{equation}
The drawback is that the space on which $\sigma$ acts is not finite-dimensional; hence, defining its cocharacteristic polynomial
requires some caution and produces a series and not a polynomial.
There actually exist several variations of Anderson's formula.
In what follows, we briefly present the motivic version where the 
operator $\sigma$ is cooked up from the Anderson motive $\MM(\phi)$ 
attached to $\phi$. 

We recall that $\MM(\phi)$ is a module over $F[T]$. 
We consider the dual motive $M := \MM(\phi)^\vee$ (see
\cref{sssec:AndersonTate}).
We further introduce the space $\Omega^1_{F/\Fq}$ of \emph{Kähler 
differentials} and the \emph{$q$-Cartier operator} $S$ acting on it.
Since $F = \Fq(z)$, we simply have 
$\Omega^1_{F/\Fq} = F{\cdot} dz$. The $q$-Cartier operator also has a very explicit description: on polynomials, it is given by
\[
  S\Big(\sum_i a_i z^i dz \Big) = \sum_i a_{qi+q-1} z^i dz
  \qquad (a_i \in \Fq)
\]
and it is extended to $\Omega^1_{F/\Fq}$ thanks to the formula 
$S(a b^{-1} dz) = S(ab^{q-1}) b^{-1} dz$ with $a, b \in 
\Fq[z]$.
We also consider the space $F[T]{\cdot}dz$ and extend $S$ to it
by $T$-linearity\footnote{For a general $A$, we should instead consider
$\Omega^1_{A \otimes F / A} = A \otimes \Omega^1_{F/\Fq}$, to which
$S$ can be extended by $A$-linearity in a similar fashion.}.

We are now ready to define the operator $\sigma$ appearing in
Equation~\eqref{eq:andersonformula}. It acts on the space
\[
  M^\star = \Hom_{F[T]} \big(M, \, F[T]{\cdot}dz\big)
\]
by
\[ 
  \sigma : f \mapsto S \circ f \circ \tau_M \quad (f \in M^\star).
\]
Like $S$, the operator $\sigma$ is linear with respect to the variable 
$T$, but semi-linear with respect to $z$ in the sense that it 
satisfies $\sigma(z^q f) = z{\cdot} \sigma(f)$ for any 
$f \in M^\star$.
The latter property is the key fact which enables us to define the 
cocharacteristic polynomial of $\sigma$.
Roughly speaking, the fact that $\sigma$ decreases the powers on the
variable $z$ implies that the contribution of large positive and negative powers of $z$ will be negligible. Hence, one can approximate the cocharacteristic polynomial of $\sigma$ by 
looking at its action on larger and larger finite-dimensional spaces
of the form $\left<z^i\right>_{|i| \leq N}$; passing to the limit on $N$, we finally get the desired result.

In~\cite{A20}, Anderson managed to give solid foundations to this
rough idea and paved the road towards a complete proof of the 
Equality~\eqref{eq:andersonformula}.

It is worth noticing that Anderson's formula is also quite explicit and 
very well-suited for implementation on computers: in \cite{CG24}, the 
authors turn it into an actual fast algorithm for computing the 
$L$-series $L(\phi; X)$--and more generally $L(\phi; X, s)$--attached 
to a Drinfeld module $\phi$.

\begin{sageexplain}
This algorithm is implemented in SageMath and performs very well in
practice, being able to compute in a couple of seconds thousands of
terms of any $L$-series.

\begin{sagecolorbox}
\begin{sagecommandline}

    sage: phi = DrinfeldModule(A, [T, T^2, 1])
    sage: phi.Lseries(prec=30)
    sage: phi.Lseries(1, prec=40)  # value at $1$

\end{sagecommandline}
\end{sagecolorbox}

\end{sageexplain}

\subsubsection{Triptych: rank-$2$ Drinfeld modules, elliptic curves over number fields, elliptic curves over function fields}
\label{sssec:triptych}

In this paragraph, we gather and compare upper bounds on the minimal degree of isogenies in three classical settings: elliptic curves over number fields, Drinfeld modules of rank 2, and elliptic curves over function fields. We hope the reader will get a feeling for why, at least in this particular case, the best analogue in the function field setting for elliptic curves over number fields, are in fact Drinfeld modules of rank~2, and not elliptic curves over function fields.

The estimates we will study are of two types: upper bounds on the minimal degree of an isogeny between two objects, and an estimate on the height of the relevant $j$-invariant in any isogeny class.

Let $L$ be a number field. We will use $h$ to denote the \emph{logarithmic Weil height}, which for an element $\alpha\in L$ is defined as
\[
h(\alpha)=\frac{1}{[L:\Q]}\sum_{v\in M_L} [L_v:\Q_v]\log\max\{ 1, \Vert\alpha\Vert_v\},
\]
where $M_L$ is the set of places of $L$, and for $v\in M_L$ we write $\Vert \cdot \Vert_v$ for the associated absolute value and $L_v$ for the completion of $L$ at $v$.

\begin{theorem}[\cite{GR14}]\label{minimal isog}
    Let $E $ and $E'$ be two $\overline{L}$-isogenous elliptic curves. Then, there exists an  $\overline{L}$-isogeny $f : E \rightarrow E'$ such that
    \[\deg(f)\leq c_0\cdot [L:\Q]\cdot h(j)^2,\]
    where $j$ is the j-invariant of the elliptic curve $E$.
\end{theorem}

We underline that
there are many isogenies between $E$ and $E'$ which do not satisfy the degree bound from \cref{minimal isog}. In contrast, the following theorem, which addresses a different but related question, is valid for any isogeny, with a bound that is sharper for minimal degree isogenies.

\begin{theorem}[\protect{\cite[Theorem 1.1]{P19}}]\label{ECNF}
    Let $f : E \rightarrow E'$ be an $\overline{L}$-isogeny between two elliptic curves $E$ and $E'$ of $j$-invariant $j$ and $j'$, respectively. Then
    \[|h(j)-h(j')|\leq 10+12\log(\deg f).\]
\end{theorem}

The Northcott property for the Weil height says that a set of algebraic numbers with bounded height and bounded degree is finite. Thus, combined with the two previous theorems, we get that within an isogeny class, there are only finitely many $\overline{L}$-isomorphism classes of elliptic curves over a fixed number field.

The goal of the rest of this subsection is to present similar results for Drinfeld modules and for elliptic curves defined over function fields.  The results presented in this subsection are summed up in \cref{fig:tryptich}. 

Recall that $A=\Fq[T]$ and $K=\Fq(T)$. Let $M_K$ denote the set of places of $K$.
For a place $v\in M_K$ we denote by $\Vert \cdot \Vert_v$ the absolute value normalised as follows: if $\mathfrak{p}\in A$ is finite corresponding to $v$, then $\Vert x \Vert_v=q^{-\deg(p)\cdot \mathrm{val}_\mathfrak{p}(x)}$, if $\mathfrak{p}\in A$ is infinite corresponding to $v$, then $\Vert x \Vert_v=q^{deg(x)}$.
Let $F/K$ be a finite extension of $K$. For a place $v\in M_F$ extending a place $w\in{M_K}$, we denote by $\Vert \cdot \Vert_v$ the associated absolute value such that for any $x\in K$ we have $\Vert x \Vert_v=\Vert x \Vert_w$. For a $n$-tuple $x=(x_1,\dots,x_n)$ of elements in $F$ we define its logarithmic Weil height $h$ as
\begin{equation}\label{Weil height}
h(x)=\frac{1}{[F:K]}\sum_{v\in M_K} [F_v:K_v]\log\max\{\Vert x_1 \Vert_v,\dots,\Vert x_n \Vert_v\}.
\end{equation}

\begin{remark}
    The logarithmic Weil height of Equation~\eqref{Weil height} is a completely different concept from the Frobenius height used in \cref{Frobenius height}. These are both called ``height'' in the literature, and are both denoted $h(\cdot)$. The context is generally clear.
\end{remark}

\begin{theorem}[\protect{\cite[Theorem 1.3]{DD99}}]
    Let $ \phi $ and $ \phi'$ be two $\Fs$-isogenous Drinfeld modules of rank $r$ defined over $F$. Then, there exists an $\Fs$-isogeny $f : \phi \rightarrow \phi'$ such that
    \[\deg(f)\leq c_0 [F:K] \cdot h(\phi')^{10\cdot 3^7},\]
    where $h(\phi')$ is the height of the tuple of Potemine invariants of the Drinfeld module $\phi'$ (see \cref{def:jinvariants}).
\end{theorem}

\begin{theorem}[\protect{\cite[Theorem 3.1]{BPR21}}]\label{DM}
    Let $f: \phi \rightarrow \phi'$ be an $\Fs$-isogeny between two Drinfeld modules of rank $r$. Suppose $\ker f\subseteq \phi[N]$. Then
\[|h(\phi')-h(\phi)|\leq \deg N+ \frac{q}{q-1}-\frac{q^r}{q^r-1},\]
where $h(\phi)$ and $h(\phi')$ are the height of the tuple of Potemine invariants of the Drinfeld modules $\phi$ and $\phi'$, respectively. 

    Moreover, if $r=2$, then we have
    \[|h(j')-h(j)|\leq q+ \frac{q^2-1}{2}\left( \log \deg f + \log\left(1+\frac{1}{q}h(j')\right)\right),\]
    where $j$ and $j'$ are the $j$-invariant of the Drinfeld module $\phi$ and $\phi'$, respectively.
\end{theorem}

Again, combining the two previous theorems with the Northcott property, we get that isogeny classes of Drinfeld modules of rank $2$ contain finitely many isomorphism classes.

\begin{remark}
In the case of isogenies in parallelogram configurations, there are finer results, see the recent article \cite{BGP_2025}. 
\end{remark}

Finally, let $F$ be the function field of a smooth projective and geometrically irreducible curve of genus $g$ defined over a perfect field. For elliptic curves defined over $F$, the following hold.

\begin{theorem}[\protect{\cite[Theorem B]{GP22}}]
    Let $ E $ and $ E'$ be two $\Fs$-isogenous elliptic curves with j-invariant $j$ and $j'$, respectively. Then, there exists an  $\Fs$-isogeny $f : E \rightarrow E'$ such that
    \[\deg(f)\leq 49 \cdot\max\{1,g\}\cdot\max\left\{\frac{\deg_{ins} j}{\deg_{ins}j'},\frac{\deg_{ins} j'}{\deg_{ins}j}\right\},\]
    where $\deg_{ins}$ denotes the inseparability degree.
\end{theorem}
\begin{theorem}[\protect{\cite[Theorem A]{GP22}}]\label{ECFF}
    Let $f : E \rightarrow E'$ be an $\Fs$-isogeny between two elliptic curves $E$ and $E'$ with j-invariant $j$ and $j'$, respectively. Then, 
    \[h(j')= \frac{\deg_{ins}f}{\deg_{ins}\hat{f}} \cdot h(j),\]
    where $\deg_{ins}$ denotes the inseparability degree.
\end{theorem}
\smallskip

The situation of isogeny classes of elliptic curves over function fields is very different from the two previous settings. These isogeny classes contain infinitely many isomorphism classes! We refer the reader to \cite{GP22} for more details. Taking a look at \cref{ECFF}, \cref{DM}, and \cref{ECNF}, it appears that isogeny classes of elliptic curves over number fields behave more like isogeny classes of Drinfeld modules of rank $2$, than like isogeny classes of elliptic curves over function fields!  

\renewcommand{\arraystretch}{1.4}
\begin{figure}
    \centering  
\begin{tabular}{@{}|C{0.3\textwidth}|C{0.24\textwidth}|C{0.36\textwidth}|@{}}
	\hline
	\textbf{Drinfeld modules of rank $2$} & \textbf{Elliptic curves over a number field} & \textbf{Elliptic curves over a function field} \\\hline
Let $f : \phi \rightarrow \phi'$ be an isogeny between two \textit{Drinfeld modules}  $\phi$ and $\phi'$ \emph{of rank $2$} on $A=\Fq[t]$. & Let $f : E \rightarrow E'$ be a $\overline{K}$--isogeny between two elliptic curves $E$ and $E'$ over the \emph{number field} $K$.& Let $f : E \rightarrow E'$ be a $\overline{K}$--isogeny between two elliptic curves $E$ and $E'$ over a \emph{function field} $K$ of a smooth projective and geometrically irreducible curve of  genus $g$ over a perfect field. \\\hline
\multicolumn{3}{|c|}{Then we can assume that the degree of the isogeny $f$ is at most...} \\\hline
$c_0(q)(h(j) [K : F])^{10 \cdot 3^7}$ &   $c_0 \cdot [K : \Q]\cdot h(j).$ 	& \(\displaystyle 49 \cdot\max\{1,g\}\cdot\max\left\{\frac{\deg_{ins} j}{\deg_{ins}j'},\frac{\deg_{ins} j'}{\deg_{ins}j}\right\}.\)	\\
\cite[Theorem 1.3]{DD99}  & \cite{GR14} &  \cite[Theorem B]{GP22}  \\\hline
\multicolumn{2}{|c|}{Then $|h(j)-h(j')|$ is bounded from above by} &  \\\cline{1-2}
$\begin{gathered}q+ \frac{q^2-1}{2}\left( \log \deg f \right.+ \\\ \left.\log(1+\frac{1}{q}h(j'))\right)\end{gathered}$& $10+12\log(\deg f)$ & \multicolumn{-1}{C{0.36\textwidth}|}{$\displaystyle h(j')= \frac{\deg_{ins}f}{\deg_{ins}\hat{f}} \cdot h(j)$.} \\
 \cite{BPR21} &  \cite[Theorem 1.1]{P19}  &\cite[Theorem A]{GP22} \\\hline
\end{tabular}
\caption{Isogenies of small degree and height of invariants for elliptic curves and Drinfeld modules of rank $2$.}
    \label{fig:tryptich}
\end{figure}

\section{Some applications of Drinfeld modules}

\label{sec:applications}

This section aims to give a glimpse of various applications of Drinfeld modules. We first discuss their use in polynomial factorisation in \cref{sub:facto}.  Some attempts of cryptosystems based on computational problems involving Drinfeld modules, which were believed to be hard enough to ensure security, are presented in \cref{sub:crypto}. Both of these applications leverage the analogies of Drinfeld modules with elliptic curves. Finally, in \cref{sub:codes}, we present the links between Drinfeld modules and coding theory.
    
\subsection{Drinfeld modules meet computer algebra}
\label{sub:facto}

Elliptic curves play a pivotal role in computer algebra. They are used for
primality testing (\emph{Elliptic Curve Primality Proving} method, developed by
Goldwasser--Killian, refined by Altkin and Morain, see
\cite{goldwasser_almost_1986, atkin_elliptic_1993}) or integer factorization
(\emph{Elliptic Curve Method}, developed by H. Lenstra
\cite{lenstra_factoring_1987}). Integer factorisation is a notoriously hard
problem, for which the best algorithms (like the \emph{General Number Field
Sieve} \cite{lenstra_factoring_1987}) only attain sub-exponential complexity.
On the other hand, factorizing polynomials in $A=\Fq[T]$ can be done very efficiently
using probabilistic methods. Some of them rely on Drinfeld modules.

In Subsections~\ref{ssec:facto:supersingular}--\ref{ssec:facto:implementation} below, we present an algorithm proposed by Doliskani, Narayanan, and Schost~\cite{doliskani_drinfeld_2021}.
It utilises many of the tools we have presented in the previous sections (supersingularity,
structure of the endomorphism ring), and matches the asymptotic complexity of the best methods
(as of 2025).
Finally, in Subsection~\ref{facto:compar}, we give an overview of other factorisations
algorithms, and compare their performances.

\subsubsection{Supersingular reductions}
\label{ssec:facto:supersingular}

Recall that $K$ is $\mathrm{Frac}(A) = \Fq(T)$, and let $\gamma: A \to K$ be
the canonical injection. We let $\phi : A \to K\{\tau\}$ be a Drinfeld module 
of rank $2$ defined by
$$\phi_T = T + g\tau + \Delta \tau^2, \quad g \in K, \, \Delta \in K^\times.$$
As $\phi$ is not defined over a finite field, it does not have a
Frobenius endomorphism. It thus cannot be supersingular. However, one can
wonder if it becomes so upon reducing it at a prime polynomial $f_i \in A$ (see
\cref{sssec:arith:Lseries} for reductions of Drinfeld modules). In that case,
we say that $\phi$ has \emph{supersingular reduction} at $f_i$. By
\cref{prop_ordin_frobenius_trace} and \cref{rem:hasseinvariant}, this is
easy to check: one computes the Hasse invariant of $\phi \bmod{f_i}$. 
In order to do so, we define the sequence $(r_{\phi, k})_{k \geq 0}$ by
$r_{\phi, 0} = 1$, $r_{\phi, 1} = g$ and the recurrence relation
\[
  r_{\phi, k+2} =   g^{q^{k+1}} r_{\phi, k+1}
                  - \big(T^{q^{k+1}} - T\big) \Delta^{q^k} r_{\phi, k}.
\]
Then, the Hasse invariant of $\phi \bmod{f_i}$ is the image of 
$r_{\phi, \deg(f_i)}$ in $A/f_i A$.

\begin{sageexplain}
The following function computes the Hasse sequence $r_{\phi, k}$ 
modulo the given modulus up to $k = n$ (recall that $q = 7$).

  \begin{sagecolorbox}
  \begin{sagecommandline}

    sage: def hasse_sequence(g, Delta, n, modulus):
    ....:     r = [1, g] + (n-1) * [None]
    ....:     for k in range(n-1):
    ....:         r[k+2] = (g^(7^(k+1)) * r[k+1] 
    ....:                 - (T^(7^(k+1)) - T) * Delta^(7^k) * r[k]) % modulus
    ....:     return r

  \end{sagecommandline}
  \end{sagecolorbox}

We compute the $r_{\phi, 2} \bmod {T^{7^2} - T}$ for
$\phi_T = T + T \tau + (T{+}1) \tau^2$.

  \begin{sagecolorbox}
  \begin{sagecommandline}

    sage: r = hasse_sequence(T, T+1, 2, T^49 - T)
    sage: h = r[2]
    sage: h
    sage: h.factor()

  \end{sagecommandline}
  \end{sagecolorbox}

The irreducible factors of $h$ of degree $2$ are exactly the
irreducible polynomials of degree $2$ modulo which $\phi$ has 
supersingular reduction.
We check it below in two cases.

  \begin{sagecolorbox}
  \begin{sagecommandline}

    sage: F1.<U1> = F7.extension(T^2 + 2*T + 2) # $U_1$ is the image of $T$ modulo $T^2 + 2T + 2$
    sage: phi1 = DrinfeldModule(A, [U1, U1, U1 + 1])
    sage: phi1.is_supersingular()

    sage: F2.<U2> = F7.extension(T^2 + 1)       # $U_2$ is the image of $T$ modulo $T^2 + 1$
    sage: phi2 = DrinfeldModule(A, [U2, U2, U2 + 1])
    sage: phi2.is_supersingular()

  \end{sagecommandline}
  \end{sagecolorbox}

In the second case, we check moreover that the image of $h$ in $F_2$ is the Hasse
invariant of $\phi_2$, that is the coefficient in $\tau^2$ of $\phi_2(T^2 + 1)$.

  \begin{sagecolorbox}
  \begin{sagecommandline}

    sage: F2(h)
    sage: phi2(T^2 + 1)

  \end{sagecommandline}
  \end{sagecolorbox}

\end{sageexplain}

In our use case, the $f_i$ will be the irreducible factors of a given 
polynomial $f$, and so their degrees are \emph{a priori} not known, which
prevents computing the Hasse invariant. However, building on the fact that the sequence $(r_{\phi, k})_{k \geq 0}$ satisfies a recurrence of order~$2$, one can prove the following alternative characterisation, which only involves
the degree of~$f$.

\begin{proposition}[\protect{\cite[Lemma~6]{doliskani_drinfeld_2021}}]
\label{prop:lemme-6}
  For $f \in A$, $f \neq 0$, we define
\[
  \overline{h}_{\phi, f}
  = \gcd\big(r_{\phi, \deg(f)} \bmod{f},\,
              r_{\phi, \deg(f)+1} \bmod{f}\big).
\]
  Let $f_i$ be an irreducible divisor of $f$ where $\phi$ has good reduction.
  Then,
  $\phi$ has supersingular reduction at $f_i$ if, and only if, $f_i$ divides
  $\overline{h}_{\phi, f}$.
\end{proposition}

\cref{prop:lemme-6} readily suggests a method for factoring $f$: we compute 
the gcd of $f$ and $\overline{h}_{\phi, f}$ and hope that it yields a
nontrivial divisor of $f$. This will actually occur as soon as $f$ has two
irreducible divisors $f_1$ and $f_2$ such that $\phi$ has ordinary reduction
modulo $f_1$ and supersingular reduction modulo $f_2$.

\subsubsection{Drinfeld modules with complex multiplications}
\label{ssec:facto:complexmultiplication}

To ensure that the previous event will occur with good probability, one must carefully choose $\phi$.
Indeed, picking it randomly would not work. However, this changes a lot when considering Drinfeld modules with complex multiplications, defined as
follows.

\begin{definition}
A Drinfeld module $\phi : A \to K\{\tau\}$ of rank $2$ 
has \emph{complex multiplications}
if its endomorphism algebra $\End^0_{\Ks}(\phi) = \End_{\Ks}(\phi) \otimes_A K$ 
(see~\cref{subsection:endomorphism-ring}) is isomorphic, as an $A$-algebra, to an imaginary quadratic function field.
\end{definition}

Drinfeld modules with complex multiplications
are scarcer than other Drinfeld modules (those whose endomorphism algebra is
isomorphic to $K$), but they enjoy extra criteria for supersingular reduction.
Namely, if $f_i$ is unramified in $\End^0_{\Ks}(\phi)$, then $\phi$ has good
supersingular reduction at $f_i$ if, and only if, $f_i$ is inert in
$\End^0_{\Ks}(\phi)$. Together with \cref{prop:lemme-6}, this gives the following.

\begin{proposition}
\label{prop:lemme-6-inert}
  We assume that $\phi$ has complex multiplications, and that the irreducible
  divisors of $f$ are all unramified in $\End^0_{\Ks}(\phi)$. 
  Then $\gcd(\overline{h}_{\phi, f}, f)$ is the product of all
  monic prime factors $f_i$ of $f$ that are inert in
  $\End^0_{\Ks}(\phi)$.
\end{proposition}

Our task now is thus to find an explicit Drinfeld module $\phi : A \to K\{\tau\}$ with complex multiplications.
We explain how to do so when $q$ is odd; when $q$ is even, we refer to \cite[Remark~7]{doliskani_drinfeld_2021}.
Picking an element $a \in \Fq$, we consider the field $K(\omega)$, where $\omega$ is a square root of $T{-}a$,
together with the Drinfeld module $\psi : A \to K(\omega)\{\tau\}$ defined by 
$$\psi_T = (\tau + \omega)^2 + a = T + (\omega + \omega^q) \tau + \tau^2.$$
One readily checks that the Ore polynomial $\tau + \omega$ defines an endomorphism $u$ of $\psi$ whose square is the scalar multiplication by $T{-}a$.
We thus get a ring homomorphism $K(\omega) \to \End^0(\psi)$, $\omega \mapsto u$, proving that $\psi$ has complex multiplications.
Nonetheless, $\psi$ is not defined over $K$. We fix this issue by noticing that its $j$-invariant (see \cref{ssec:drinfeld:isomorphisms}) is
\begin{equation}
\label{equation:j-facto}
  j = (\omega + \omega^q)^{q+1}
    = \omega^{q+1}(1 + \omega^{q-1})^{q+1}
    = (T-a)^{\frac{q+1}{2}}\left(1 + (T-a)^{\frac{q-1}{2}}\right)^{q+1}
\end{equation}
and so it lies in $K$ given that $q$ is odd.
We can then consider the Drinfeld module 
$$\phi : A \to K\{\tau\}, \quad T \mapsto \phi_T = T + \tau + j^{-1} \tau^2.$$
It is isomorphic to $\psi$ over $\Ks$; hence it has complex multiplications as well.

\begin{sageexplain}
Doliskani, Narayanan, and Schost's algorithm can be easily implemented
as follows.

  \begin{sagecolorbox}
  \begin{sagecommandline}

    sage: def factor_DNS(f, a):
    ....:     n = f.degree()
    ....:     j = (T-a)^4 * (1 + (T-a)^3)^8   # here $q=7$
    ....:     _, Delta, _ = j.xgcd(f)   # $\Delta = j^{-1} \bmod{f}$
    ....:     r = hasse_sequence(1, Delta, n+1, f)
    ....:     return gcd([f, r[n], r[n+1]])

  \end{sagecommandline}
  \end{sagecolorbox}

In what precedes, we use Equation~\eqref{equation:j-facto} for the definition of
$j$, and $a$ is a parameter in the ground field $\mathbb F_7$.
Depending on the value of $a$, the function may or may not output a nontrivial
divisor. Below, we see that a nontrivial divisor of 
$$f = (T^2{+}1)\cdot (T^2{+}2) = T^4 + 3T^2 + 2$$
is found when $a = 1$.

  \begin{sagecolorbox}
  \begin{sagecommandline}

    sage: f = T^4 + 3*T^2 + 2
    sage: factor_DNS(f, a=0)
    sage: factor_DNS(f, a=1)
    sage: factor_DNS(f, a=2)

  \end{sagecommandline}
  \end{sagecolorbox}

Indeed, when $a = 1$, we check that the corresponding Drinfeld module has
ordinary reduction modulo the first factor of $f$, but has supersingular reduction modulo the second one.

  \begin{sagecolorbox}
  \begin{sagecommandline}

    sage: j = (T-1)^4 * (1 + (T-1)^3)^8

    sage: F1.<U1> = F7.extension(T^2 + 1)
    sage: phi1 = DrinfeldModule(A, [U1, 1, 1/j(U1)])
    sage: phi1.is_supersingular()

    sage: F2.<U2> = F7.extension(T^2 + 2)
    sage: phi2 = DrinfeldModule(A, [U2, 1, 1/j(U2)])
    sage: phi2.is_supersingular()

  \end{sagecommandline}
  \end{sagecolorbox}

  We use Equation~\eqref{equation:j-facto} again for the definition of $j$, and
  Remark~\ref{rem:j-inv-representative} for the Drinfeld modules.

\end{sageexplain}

\subsubsection{Implementation details and complexity}
\label{ssec:facto:implementation}

Turning the strategy presented above into an actual efficient algorithm requires some additional arguments and optimisations.

First of all, one needs to estimate the probability of success of the
algorithm. By \cref{prop:lemme-6-inert}, it can be estimated by studying
the proportions of polynomials of $A$ that remain inert in the imaginary 
quadratic function field $K\big(\sqrt T\big)$. This, in turn, can be done 
using standard results from the arithmetic of function fields, such as the 
Chebotarev density theorem or the Riemann--Hurwitz genus formula.
All in all, one finally finds that the probability of success in one round
of the algorithm is at least $\frac 1 4$ (see \cite[Lemma~8]{doliskani_drinfeld_2021} and the
discussion starting after Lemma~6 in \cite{doliskani_drinfeld_2021}).

Another important ingredient is a fast method for computing
$\overline{h}_{\phi, f}$. Indeed, a direct naive method is not performant
enough for the overall factorisation algorithm to match the complexity 
of the state-of-the-art. A crucial contribution of Doliskani, Narayanan
and Schost is actually a new procedure to compute $\overline{h}_{\phi, f}$. It
relies on modular composition and fast multiplication of two-by-two matrices
with entries in $A/fA$.

Regarding modular composition, the current algorithm with the best asymptotic complexity is that of Kedlaya and Umans \cite{kedlaya_fast_2011}.
One of the innovations of this algorithm was to perform bit operations, as
opposed to only algebraic operations. 
However, no efficient implementation of this algorithm yet exists.
For these reasons, the authors of \cite{doliskani_drinfeld_2021} elect to give two kinds of complexity analyses: one using only algebraic algorithms (even for modular composition) and counting
operations in $\Fq$, the second using the Kedlaya--Umans algorithm, and counting
bit operations. In the second case, one obtains a total bit complexity which grows
in $n^{\frac 3 2 + o(1)}$ (see \cite[Theorem~2]{doliskani_drinfeld_2021} for more general statements).

\subsubsection{Comparison with other algorithms}
\label{facto:compar}
For comparing algorithms, it is important to look at both the asymptotic
complexity and the practical efficiency. These usually largely differ. In the
case of factorisation of polynomials in $A$, the asymptotically best (as of
2025) algorithm is the combination of the Kaltofen--Schoup factorization algorithm \cite{kaltofen_subquadratic-time_1995} with the Kedlaya--Umans algorithm for modular composition \cite{kedlaya_fast_2011}. This leads to an asymptotic complexity, for factorising a polynomial in $A$ of degree $n$, of $n^{\frac 3 2 + o(1)}$ bit operations. As mentioned, this algorithm is highly impractical and would
only theoretically become gainful for inputs of significant size. Therefore,
the Cantor--Zassenhaus algorithm \cite{cantor_new_1981}, which leverages ideas developed by Berlekamp \cite{berlekamp_factoring_1967}, is often the
preferred method for implementation. It factors a polynomial with almost quadratic complexity (counted in operations in $\Fq$) with respect to $n$.

We also mention that all polynomial-time methods are probabilistic. The problem of finding a deterministic polynomial algorithm for the factorisation of polynomials on a finite field is still open.

\paragraph{Methods using Drinfeld modules.}

The known factorisation methods using Drinfeld modules are as follows.
\begin{itemize}
  
  \item The method of van der Heiden \cite{van_der_heiden_factoring_2004,
    van_der_heiden_addendum_2004} is an analogue of Lenstra's \emph{Elliptic Curve Method} method for factorising integers \cite{lenstra_factoring_1987}.
    Running the algorithm on a polynomial of degree $n$ costs $O(n^2 \log q +
    dn^3)$ arithmetic operations in $\Fq$
    \cite[Proposition~4.3]{van_der_heiden_factoring_2004}. The method requires
    $f$ to be given as a product of $k$ distinct prime polynomials with a
    certain degree $d$. Finding $d$ and normalising accordingly can be
    achieved by computing gcd's and derivatives following the first steps of
    the Berlekamp \cite{berlekamp_factoring_1967} and Cantor--Zassenhaus
    \cite{cantor_new_1981} algorithms. The method of van der Heiden is based on the identification of certain prime factors in the characteristic polynomial
    of the endomorphism $x \phi_T(x)$ of $K$, where $K$ is a given finite field
    and $\phi$ a random Drinfeld module with rank $d$. We can therefore use the
    independent works of Reiner \cite{reiner_number_1961} and Gerstenhaber
    \cite{gerstenhaber_number_1961} to compute the probability of finding a
    factor of $f$ in one step \cite[Remark~4.5]{van_der_heiden_factoring_2004}. Knowing bounds on the proportion of
    $d$-by-$d$ matrices with coefficients in $\Fq$ that have an irreducible
    characteristic polynomial, van der Heiden shows that by picking $\phi$
    randomly, the success probability is about $1 - \frac{(d-1)^k + 1}{d^k}$,
    provided that $q$ is sufficiently large (see
    \cite[Proposition~4.3]{van_der_heiden_factoring_2004} for a general
    statement). We also mention the survey of Randrianarisoa, which specifically targets the work of van der Heiden
    \cite{randrianarisoa_number_2014}.

  \item While van der Heiden proposed the first known complexity and success analysis of a factorisation method based on Drinfeld modules, the ideas behind his algorithm can be traced back to the work of Potemine and
    Panchishkin \cite[Section~4.3]{potemine_arithmetique_1997},
    \cite{panchishkin_algorithmes_1993}. This was acknowledged by van der
    Heiden \cite{van_der_heiden_addendum_2004}. In
    \cite{panchishkin_algorithmes_1993}, Panchiskin also mentions that one
    could adapt classical methods (\emph{Elliptic Curve Primality Proving}) for
    primality testing in $A$ using Drinfeld modules. Panchishkin also discusses
    a possible analogue of the Schoof point counting algorithm
    \cite{schoof_elliptic_1985}. One is now known to be described in
    \cite[Section~6]{musleh_computing_2019} (see \cref{sec:review-frobenius-charpoly}).

  \item Shortly before the work of Doliskani, Narayanan and Schost, Narayanan
    proposed two factorisation methods using Drinfeld modules
    \cite{narayanan_polynomial_2018}.
    \begin{itemize}

      \item The first one is not a factorisation algorithm \emph{per se}, but an algorithm to compute the degree of the smallest prime factor of $f$.
        Knowing this information allows for significant speed-ups in the Kedlaya--Umans version of the Kaltofen--Schoup: the bit asymptotic
        complexity falls to $n^{1 + o(1)} (\log q)^{2 + o(1)}$.
        Narayanan reads
        this smallest degree on the ``number of points''
        (see \cref{th:number-points}) of a random rank-$2$ Drinfeld module
        over a finite field.
      \item The second method is a tweaked version of the Berlekamp algorithm, in which the action of a Frobenius endomorphism is replaced by the
        action of a Drinfeld module. No complexity analysis is provided by the author.
    \end{itemize}
\end{itemize}

Among these methods, only that of Doliskani, Narayanan and Schost
\cite{doliskani_drinfeld_2021} matches the asymptotic state of the art. It is also the only method among those that does not use general random Drinfeld
modules, but random Drinfeld modules with complex multiplications. The authors
have developed an implementation using the NTL C++ library \cite{NTL}, which uses more straightforward methods for modular composition. It is slower
than the default NTL function by a constant factor, due to the manipulation of
two-by-two polynomial matrices, according to the authors.

\subsection{Drinfeld modules meet cryptography}
\label{sub:crypto}

Cryptography relies on computationally hard problems, be them used in encryption, key exchange, signature or authentication protocols. Classical computationally hard problems include factoring integers, computing discrete
logarithms on finite fields and elliptic curves over finite fields, or
computing isogenies between elliptic curves over finite fields. These are used
in the RSA \cite{rivest_method_1978}, Diffie-Hellman \cite{diffie_new_1976},
Elliptic Curve Diffie-Hellman \cite{koblitz_elliptic_1987, miller_use_1986},
and SQIsign cryptosystems \cite{de_feo_sqisign_2020}, respectively.

In the hope of gaining efficiency with function field arithmetics, there have
been many attempts at introducing cryptosystems based on Drinfeld modules, inspired by the aforementioned classical constructions. Unfortunately, all these attempts have failed so far.

\begin{itemize}

  \item In \cite{scanlon_public_2001}, Scanlon defined the \emph{Drinfeld
    module discrete logarithm problem}, and the \emph{Drinfeld module inversion problem}, in the hope of building Drinfeld module analogues of the RSA and
    Diffie-Hellman cryptosystems. He proved the computational ease of these
    problems in the same paper.

  \item In \cite{joux_drinfeld_2019}, Joux and Narayanan derived analogues of
    the SIDH \cite{jao_towards_2011} and CSIDH \cite{castryck_csidh_2018}
    cryptosystems, and claimed their weakness.
  
  \item In \cite{leudiere_hard_2022}, Leudière and Spaenlehauer proposed an
    analogue of the CRS cryptosystem. The work of Wesolowski on the computation
    of isogenies of Drinfeld modules (\cref{computing_isogenies}) proved that this new construction, as well as that of Joux and Narayanan, cannot be
    used safely.
\end{itemize}

\begin{remark}
  We mention that the SIDH cryptosystem was broken in 2022, following a
  series of attacks started by Castryck and Decru \cite{castryck_efficient_2023},
  Maino and Martindale \cite{maino_attack_2022}, and that culminated in the
  unconditional cryptanalysis of Robert \cite{robert_breaking_2023}. As of 2025, CSIDH has not proven insecure. However, the ideas behind SIDH and
  CSIDH were adapted in the construction of SQISign, a signature protocol in the NIST competition for standardising post-quantum cryptosystems.
\end{remark}
Let $\phi$ be a Drinfeld module defined over a finite $A$-field $(F, \gamma)$.
Given $x,y \in F$, the \emph{Drinfeld module discrete logarithm problem}
consists in finding, if it exists, an element $a \in A$ such that $\phi_a(x) =
y$. As far as it is concerned, the \emph{Drinfeld module inversion problem}
asks, given $a \in A$ such that $\phi_a$ is a bijective $\Fq$-linear
endomorphism of $K$, to find $b \in A$ such that $\phi_b$ (as an $\Fq$-linear
endomorphism of $K$) is the compositional inverse of $\iota(\phi_a)$. Scanlon
shows that both problems can be solved using linear algebra techniques: the
maps $a \mapsto \phi_a(x)$ and $x \mapsto \phi_a(x)$ are both $\Fq$-linear. In
\cite{gillard_utilisation_2003}, Gillard, Leprévost, Panchishkin and Roblot
proposed a fix for the \emph{Drinfeld module inversion problem}. Their proposal was proven insecure later, by Blackburn, Cid and Galbraith, in
\cite{blackburn_cryptanalysis_2006}.

As far as the constructions of Joux--Narayanan and Leudière--Spaenlehauer are concerned, these are all insecure because one can efficiently compute isogenies
between Drinfeld modules over a finite field, thanks to the work of Wesolowski
(see \cref{computing_isogenies}).

The common point of these attacks is a finite-dimensional vector space over
$\Fq$. No such thing exists in the classical case, because no field lies
beneath the ring of integers. We refer to
\cite[Appendix~A]{leudiere_morphisms_2024} for details on these attempts.

\begin{remark}

  It may not be obvious that the Drinfeld module inversion problem is related
  to the RSA cryptosystem. Yet, there are important symmetries between the two
  constructions. Both involve a module (over $\Z$ or over $A$). In each case,
  we have an encryption cryptosystem, and we now explain (see
  \cref{table:RSA}) that the encryption
  function can be seen as a linear endomorphism on a torsion space of this
  module. For simplicity, let us assume that we work with rank-$1$ Drinfeld
  modules, and that we only consider elements $\phi_a$ when $a$ is away from
  the $A$-characteristic. 

  \begin{table}[h]
    \centering  
    \begin{tabular}{|C{2cm}||c|c|}
      \hline
      & \textbf{Classical RSA} & \textbf{Drinfeld module RSA} \\\hline
      \hline

      \textbf{Module}
      & \makecell{$\Z$-module defined by: \\
          $\begin{array}{rcl}
            \Z\times \Q^\ast & \to & \Q^\ast \\
            (n, x) &\mapsto & n \ast x = x^n
          \end{array}$
        }
      & \makecell{$A$-module defined by: \\
          $\begin{array}{rcl}
            A \times F & \to & F \\
            (a, x) &\mapsto & a \ast x = \phi_a(x)
          \end{array}$
        }
      \\ \hline

      \textbf{Torsion}
      & $n$-torsion (in $\overline{\Q}$): $\Z/n\Z$
      & $a$-torsion (in $\Fs$): $A/aA$
      \\ \hline

      \textbf{Encryption function}
      & \makecell{Linear automorphism \\
        $\begin{array}{rcl}
          \Z/n\Z & \to & \Z/n\Z \\
          x &\mapsto & e \ast x
        \end{array}$ \\
        with $e$ carefully chosen}
      & \makecell{Linear automorphism \\
        $\begin{array}{rcl}
          F & \to & F \\
          x &\mapsto & a \ast x
        \end{array}$ \\
        with $a$ carefully chosen}
      \\ \hline
    \end{tabular}
  \caption{Comparison between the classical RSA setting and its Drinfeld analogue.}\label{table:RSA}
  \end{table}

  One might wonder why the Drinfeld module RSA encryption function would not be
  defined on $A/aA$. In the classical case, the isomorphism between the
  $n$-torsion and $\Z/n\Z$ is explicit, and given by primitive $n$-th roots of
  unity. This allows us to view the multiplicative law $\ast$ directly on
  $\Z/n\Z$. No such explicit isomorphism is given in the construction of the
  Drinfeld module case (see \cref{prop:torsion}). That being said, given that
  $F$ is finite, all elements of $\phiact F$ are torsion elements, namely, they
  are all of $\chi_\phi(1)$-torsion---where $\chi_\phi$ is the characteristic polynomial
  of the Frobenius endomorphism---by \cref{th:number-points},
  and $\phiact F$ can be seen as an
  $A$-submodule of a torsion space.

  These similarities between the classical and Drinfeld module RSA construction
  come from the constructions of cyclotomic number fields and function fields.
  As we saw in \cref{applications_cyclotomic_theory}, cyclotomic number
  fields are obtained by adding roots of unity to $\Q$, while cyclotomic
  function fields are obtained by adding to $\Fq(T)$ the $a$-torsion of a rank-$1$ Drinfeld module defined over $\Cinf$.
\end{remark}

\subsection{Drinfeld modules meet linear codes}\label{sub:codes}

\emph{Error-correcting codes} are mathematical tools designed to detect and correct errors in transmitted or stored data. A data item is encoded by adding redundancy before being transmitted or stored. Then the original data can be recovered through a decoding operation, even if it has been corrupted during transmission in space and time. 
The most widely studied model of error-correcting codes is that of \emph{linear codes}. 

In this section, we present some known connections between Drinfeld modules and coding theory. Drinfeld modules played a significant role in the construction of asymptotically good codes, which historically gave algebraic geometry its credentials in coding theory. We briefly review this piece of history in \cref{subsub:tower}. Very recently, Drinfeld modules have reemerged in coding theory to design locally recoverable codes in the rank metric \cite{BDM24preprint}. We present this construction in \cref{subsub:LRC}, which originally relies on Carlitz modules, and we discuss its generalisation to higher-rank Drinfeld modules.

\subsubsection{Linear codes in the Hamming and the rank metrics}

A \emph{linear code} is a $\Fq$-vector subspace of some finite-dimensional $\Fq$-linear ambient space. The efficiency of a linear code is captured by three main parameters, which are its \emph{length}, defined as the dimension of the ambient space in which the code sits, its \emph{dimension} as a vector space, and its \emph{minimum distance}, which depends on the metric put on the ambient space. Notably, the dimension of a code represents the quantity of data that one can encode with it, the length is the size of the transmitted message, which includes the redundancy added to the raw data, while the minimum distance is related to the error detection and correction capacity. Indeed, it is well known that a code with minimum distance $d$ can correct up to $\frac{d-1}{2}$ errors or $d-1$ erasures.

 Traditionally, linear codes are defined as $\Fq$-vector subspaces of $\Fq^n$ endowed with the \emph{Hamming distance}, which measures the number of positions in which two codewords, seen as vectors in $\Fq^n$, differ.
 \begin{definition}[Codes in the Hamming metric]   
The \emph{Hamming distance} between $x,y\in\Fq^{n}$ is defined as 
\[d(x,y)=\#\{i\in[n]\;:\; x_i\neq y_i\}.\]
A \emph{code in the Hamming metric} $\mathcal{C}$ is an $\Fq$-linear subspace of $\Fq^{n}$ endowed with the Hamming distance. Its \emph{dimension} $k$ is $\dim_{\Fq} \mathcal{C}$, its \emph{minimum distance} is defined as 
\begin{align*}
d(\mathcal{C})\coloneqq
&\min\{d(x,y)\;:\; x,y\in\mathcal{C},x\neq y\}\\
=&\min\{d(x,0)\;:\; x\in\mathcal{C},x\neq 0\}.
\end{align*}
We call a code $\mathcal{C}\subseteq \Fq^{n}$ of dimension $k$ and minimum distance $d(\mathcal{C})=d$ a $[n,k,d]_q$-code. 
\end{definition}

A large family of codes in the Hamming metric is composed of so-called \emph{evaluation codes}, which are built, as the name suggests, by evaluating a space of functions at a set of elements. This is the case of \emph{Reed--Solomon} \cite{RS60} and \emph{Reed--Muller codes} \cite{reed1954class, muller1954application}, where one evaluates polynomials in one or $m$ variables at distinct elements of $\Fq$ or $\Fq^m$, respectively, and of \emph{Algebraic Geometry codes} \cite{Goppa81}, obtained by evaluating rational functions at points over an algebraic curve. A $[n,k,d]_q$-code in the Hamming metric respects the so-called \emph{Singleton bound} $k+d\leq n+1$. Codes attaining this bound are called \emph{Maximum Distance Separable} (MDS). The aforementioned Reed--Solomon code, described below, provides an example of an MDS code.

\begin{definition}
Let $k \leq n \leq q$ be integers. Let $\mathbf{x}=(x_1,\dots,x_n)$ be a tuple of pairwise distinct elements of the finite field $\Fq$. The \emph{Reed--Solomon code} of support $\mathbf{x}$ and dimension $k$ is defined as
    \[\textsf{RS}_k(\mathbf{x})=\{(f(x_1),\dots,f(x_n)) : f \in \Fq[T], \: \deg f < k \}.\]
    It is an $[n,k,n-k+1]_q$-code.
\end{definition}

More recently, metrics other than that of Hamming have been considered for error-correcting codes, better suited to address specific information theory problems. Among them is the \emph{rank metric} \cite{Delsarte}. 

\begin{definition}[Codes in the rank metric] 
Let $\Fq^{n\times m}$ denote the space of $n\times m$ matrices with coefficients in $\Fq$. The \emph{rank distance} between $M,N\in\Fq^{n\times m}$ is defined as 
\[d_\rk(M,N)=\rk(M-N).\]
A \emph{rank metric code} $\mathcal{C}$ is an $\Fq$-linear subspace of $\Fq^{n\times m}$ endowed with the rank distance. Its dimension $k$ is  $\dim_{\Fq} \mathcal{C}$, its minimum rank distance is defined as 
\begin{align*}
d_\rk(\mathcal{C})\coloneqq
&\min\{d_\rk(M,N)\;:\; M,N\in\mathcal{C},M\neq N\}\\
=&\min\{\rk(M)\;:\; M\in\mathcal{C},M\neq \boldsymbol{0}\}.
\end{align*}
We call a code $\mathcal{C}\subseteq \Fq^{n\times m}$ of dimension $k$ and minimum distance $d \coloneqq d_\rk(\mathcal{C})$ an \emph{$[nm,k,d]_q$ rank-metric code}. 
\end{definition}

\begin{remark}
Here, we present rank-metric codes in the formalism of matrices to comply with the notation of \cite{BDM24preprint}, a paper we will present in \cref{subsub:LRC}. However, rank-metric codes can also be seen as spaces of linear morphisms. Let $V$ be a $\Fq$-subspace of $\Fqm$ of dimension $n$. The \emph{rank distance} between ${f,g\in\Hom(V,\Fqm)}$ can naturally  be defined as 
$d_\rk(f,g)=\rk(f-g)$.
A \emph{rank metric code} $\mathcal{C}$ is then an $\Fq$-linear subspace of $\Hom(V,\Fqm)$ endowed with the rank distance. One can recover the matrix point of view by fixing $\Fq$-bases for $V$ and $\Fqm$.
\end{remark}

A particular case of rank-metric codes that we will focus on is the so-called \emph{vector rank metric codes}. After choosing an $\Fq$-basis of $\Fqm$, one can associate to any vector $x\in\Fqm^n$ a matrix $M_x\in\Fq^{n\times m}$. Then any vector subspace of $\Fqm^n$ can be endowed with the rank metric and be regarded as a matrix code. Compared to matrix rank-metric codes, they naturally come with the extra property of $\Fqm$-linearity.
\begin{definition}
    A \emph{(vector) rank metric code} $\mathcal{C}$ is an $\Fqm$-linear subspace of $\Fqm^{n}$ endowed with the rank distance. Its dimension $k$ is  $\dim_{\Fqm} \mathcal{C}$, its minimum rank distance is defined as 
\[d_\rk(\mathcal{C})\coloneqq
\min\{\rk(M_x)\;:\; x\in\mathcal{C},x\neq \boldsymbol{0}\}.\]
We call a code $\mathcal{C}\subseteq \Fqm^n$ of dimension $k$ and minimum distance $d_\rk(\mathcal{C})=d$ an \emph{$[n,k,d]_{q^m}$ rank-metric code}. 
\end{definition}

Rank-metric codes are particularly effective in contexts where errors may affect entire rows or columns rather than isolated symbols, such as in network coding, space-time coding, and cryptography. The interested reader can consult \cite{bartz2022rank} for a survey on the applications of rank metric codes.
\smallskip

The theory of Ore polynomials allowed the construction of evaluation codes in the rank metric, leading to the analogues of all the aforementioned codes \cite{Delsarte,ALG18,ACLN21,BC25,BC24}. In particular, the counterparts of Reed--Solomon codes in the rank metric world are Gabidulin codes, introduced by Delsarte \cite{Delsarte} and Gabidulin \cite{Gabidulin}.

\begin{definition}
For a positive integer $\kappa \leq m$ and a $\Fq$-linear subspace $W$ of $\Fqm$
of dimension $n$, the associated \emph{Gabidulin code} is defined as the image of the map 
\Function
      {\enc}
      {\Fqm\{\tau\}_{\leq \kappa-1}}
      {\Hom_{\Fq}(W,\Fqm)\simeq\Fq^{n\times m}}
      {f}
      {f|_W}
where $\Fqm\{\tau\}_{\leq \kappa-1}$ is the vector space of Ore polynomials of degree bounded by $\kappa-1$, and we recall from Subsection~\ref{subsubsection:ore-pols-kernels} that, in a slight abuse of notation, we continue to write $f$ for the induced linear map on $\Fqm$.

The Gabidulin code is an $[mn,m\kappa,n{+}1{-}\kappa]_q$ rank-metric code. One can look at a Gabidulin code as an $\Fqm$-linear subspace of $\Fqm^n$, thus obtaining a $[n,\kappa,n{+}1{-}\kappa]_{q^m}$ (vector) rank-metric code.
\end{definition}
The parameters of rank-metric codes satisfy the rank-Singleton bound \cite[Theorem 3.5]{Gor21}:
$$k\leq m (n-d_\rk(\mathcal{C})+1).$$
Codes whose parameters reach this bound are called \emph{Maximum Rank Distance} (MRD) codes. Gabidulin codes provide examples of MRD codes.

While Ore polynomials have been successfully applied in coding theory in the last forty years, Drinfeld modules have not been used in this context until very recently, when Carlitz modules were exploited to construct codes in the rank metric with optimal properties for distributed storage \cite{BDM24preprint}, giving so-called locally recoverable codes. This construction will be described in \cref{subsub:LRC}. Before that, in the next section, we briefly discuss the importance of Drinfeld modular curves for codes in the Hamming metric.

\subsubsection{Modular curves for asymptotically good codes in the Hamming metric}\label{subsub:tower}

Given a sequence of codes $(C_s)_{s\in \N}$ of parameters $[n_s,k_s,d_s]_q$, we define its \emph{rate} $R \coloneqq \lim_{s \rightarrow \infty} k_s/n_s$ and its relative distance $\delta \coloneqq \lim_{s \rightarrow \infty} d_s/n_s$. A sequence of codes $(C_s)_{s\in \N}$ is said to form a family of \emph{asymptotically good codes} if $R > 0$ and $\delta>0$.

The \emph{asymptotic Gilbert--Varshamov bound} \cite[Proposition~8.4.4]{Sti09} states that families of increasingly long random codes achieve $R > 1 -h_q(\delta) + o(1)$ where $h_q$ is the $q$-ary entropy function defined by \[\forall x \in [0,1], \quad h_q(x)=x\log_q(q-1)-x\log_q(x)-(1-x)\log_q(1-x).\] For a long time, it was believed that the Gilbert--Varshamov bound was indeed a bound, that is, no family of structured codes could achieve a better rate and relative distance than random codes, until Tsfasman, Vlăduţ, and Zink constructed Algebraic Geometry codes that beat the Gilbert--Varshamov bound \cite{TVZ82}. 
The parameters of an Algebraic Geometry code constructed from a curve $\calX$  depend on the geometry of $\calX$. For instance, the length is bounded from above by the number of $\Fq$-points $\#\calX(\Fq)$ of $\calX$. Tsfasman, Vlăduţ, and Zink's construction relies on an asymptotically good tower of curves.

A \emph{tower of curves} is an infinite sequence of curves and surjective maps
\[\dots \twoheadrightarrow \calX_3 \twoheadrightarrow \calX_2 \twoheadrightarrow \calX_1 \twoheadrightarrow \calX_0,\]
which are all defined over a finite field $\Fq$ and whose genera $g(\calX_s)$ satisfy 
\begin{equation}\label{eq:tower}
	\lim_{s \rightarrow \infty} g(\calX_s) = +\infty.
\end{equation}
By the Hasse--Weil theorem, the number of $\Fq$-points of each curve $\calX_s$ is bounded by
\[\#\calX_s(\Fq) \leq q+1+2\sqrt{q}g(\calX_s).\]

The tower is said to be \emph{asymptotically good} if the number of $\Fq$-points on the curves $\calX_s$ grows significantly faster than their genera, \ie
\[\limsup_{n \rightarrow \infty} \frac{\#\calX_n(\Fq)}{g(\calX_n)} > 0.
\] 
From the Hasse-Weil bound, we know that $\frac{\#\calX_n(\Fq)}{g(\calX_n)}\leq 2\sqrt{q}$. Ihara \cite{ihara1981} noticed that this bound is far from being met for large-genus curves and introduced the quantity
\[A(q)=\limsup_{g \rightarrow \infty} \frac{\max\{\#\calX(\Fq) : \calX \text{ curve of genus }g \}}{g}, \]
now called \emph{Ihara's constant}. He proved that $A(q)  \geq \sqrt{q}-1$ when $q$ is a square. Later, Vlăduţ and Drinfeld proved that $A(q) \leq \sqrt{q}-1$
\cite{vluaduct1983}, meaning that $A(q)=\sqrt{q}-1$ when $q$ is a square. A tower $(\calX_n)$ is said to be \emph{optimal} if
\[\limsup_{n \rightarrow \infty} \frac{\#\calX_n(\Fq)}{g(\calX_n)} = \sqrt{q}-1.
\] 

Modular curves provide asymptotically good towers of curves, which are essential for constructing Algebraic Geometry codes with strong asymptotic performance. The classical modular curves $X_0(N)$, parametrising elliptic curves with cyclic $N$-isogenies, form asymptotically good towers over $\F_{p^2}$, for prime $p$. This was leveraged by Tsfasman, Vlăduţ, and Zink to construct Algebraic Geometry codes beating the Gilbert--Varshamov bound \cite{TVZ82}. The curious reader is invited to read \cite{C22survey} for a detailed introduction to this topic.

Similarly, Drinfeld modular curves, which parametrise rank-$2$ Drinfeld modules with level structures, also yield asymptotically good towers over any quadratic field $\F_{q^2}$ \cite{bassa2015good}, and more generally any non-prime field. These curves are particularly useful over finite fields $\Fq$ where $q$ is a square, as seen in the optimal towers constructed by Garcia and Stichtenoth \cite{GS95}. Elkies gave a modular interpretation for certain recursive towers previously studied in coding theory, such as the Garcia-Stichtenoth towers, providing a deeper structural explanation for their optimality \cite{E01}. More precisely, he showed that the tame case corresponds to classical modular curves \cite{E97}, while the wild case corresponds to Drinfeld modular curves \cite{E01}. Subsequently, many explicitly known recursively defined towers have been given a modular interpretation. Bassa, Beelen and Nguyen \cite{bassa2014good,bassa2015good} showed that the defining equations for these (classical or Drinfeld) modular towers can be read off directly from the modular polynomial. The theory of modular curves is therefore considered an efficient machinery to produce explicitly defined families of curves, particularly good and optimal towers of curves. We refer the reader to Beelen's survey \cite{beelen2022surveyrecursivetowersiharas} for more details.

\subsubsection{Locally recoverable codes in the rank metric}\label{subsub:LRC}
One main application of coding theory is data recovery in distributed storage systems. Originally, data replication was used to ensure reliability against node (\eg individual machine) failures. In the last decades, distributed storage systems have been transitioning to the use of \emph{erasure codes} as they offer higher reliability at significantly lower storage costs. One can easily check that an $[n,k,d]$ linear code can recover up to $d-1$ erasures. For instance, in 2014, Facebook deployed a version of the Hadoop Distributed File System (HDFS) that relies on a $[14, 10, 5]$ Reed--Solomon code and can thus resist as many as $4$ node failures. However, this requires accessing the $10$ other nodes and performing high-degree polynomial interpolation.

\emph{Locally recoverable codes} are a class of error-correcting codes designed to efficiently address this issue by allowing each symbol in a codeword to be recovered by accessing only a small subset of other symbols.  Locally recoverable codes were originally introduced in the Hamming metric \cite{GHSY12,lrc}. We recall their definition below.
\begin{definition}{\cite[Definition 2]{BTV17}}\label{def:hamminglocality}
A code $\mathcal{C}\subseteq \Fq^{n}$ is \emph{locally recoverable} with \emph{locality} $(t,\delta)$ if for every coordinate index $i\in\{1,\dots, n\}$ there exists a subset $\Gamma[i]\subseteq \{1,\dots, n\}$ containing $i$ such that $|\Gamma(i)| \leq t+\delta-1$ and the code $\mathcal C|_{\Gamma(i)}$ obtained by restricting $\mathcal{C}$ to the coordinates in $\Gamma(i)$ has minimumdistance $d(\mathcal C|_{\Gamma(i)})\geq \delta$.
\end{definition}
The definition was generalised to the rank metric in \cite{duursma} in the following way.

\begin{definition}\label{def:ranklocality}
A rank metric code $\mathcal{C}\subseteq \Fq^{n\times m}$ is \emph{locally recoverable} with \emph{rank-locality} $(t,\delta)$ if, for every column index $i\in\{1,\dots, m\}$, there exists a set of columns $\Gamma[i]\subseteq \{1,\dots, m\}$ containing $i$ such that $|\Gamma(i)| \leq t+\delta-1$ and the code $\mathcal C|_{\Gamma(i)}$ obtained by restricting $\mathcal{C}$ to the coordinates in $\Gamma(i)$ has minimum rank distance $d_\rk(\mathcal C|_{\Gamma(i)})\geq \delta$.
\end{definition}

\begin{remark}
The above definition is a natural extension of the one originally introduced in the Hamming metric. In particular, when the rank-metric code $\mathcal{C}$ is $\Fqm$-linear, it coincides exactly with the classical definition of locally recoverable codes in the Hamming setting. 
The notion of locality relates to code puncturings. A puncturing of a rank-metric matrix code is a projection onto a coordinate subspace that, in opposition to the Hamming setting, can be precomposed with a rank-preserving isomorphism $A \in GL_n(\Fq)$ \cite[Definition~2.15]{neri19}. \cref{def:ranklocality} only involves puncturings with trivial isomorphisms $A=I_n$. Therefore, the current notion of locality in the rank metric does not fully leverage the richness of the rank metric and may be considered somewhat unsatisfactory. However, since this is the definition used in \cite{BDM24preprint}, we will stick to it.
\end{remark}

A locally recoverable rank-metric code $\mathcal{C}\subseteq\Fq^{n\times m}$ 
of dimension $m\cdot k$ and with rank-locality $(t,\delta)$ satisfies the following Singleton-type bound \cite[Theorem 1]{KED19}
\begin{equation}\label{eq:lr-sing-bound}
    d_\rk(\mathcal{C})\leq n-k+1-\left(\left\lceil \frac{k}{t}\right\rceil -1\right)(\delta-1).
\end{equation}

In what follows, we present the construction of locally recoverable codes in the rank metric from Carlitz modules as outlined in \cite{BDM24preprint}, while generalising it to Drinfeld modules of any rank. 

Let $\phi: \Fq[T]\to \F_{q^m}\{\tau\}$ be a Drinfeld Module of rank $r$ with $\phi_T\coloneqq z+g_1 \tau + \dots + g_r\tau^r$. 
We consider the message space to be
\[\mathcal{M}\coloneqq\left\{\sum_{k=0}^s f_k(\tau)\phi_T^k \;:\; f_k\in \F_{q^m}\{\tau\}_{\leq t-1}\right\}.\]

For some positive integer $\ell$ such that $\ell\geq s+1$ we choose $a_1,\dots,a_\ell\in\Fq^*$ and build the polynomial $h=\prod_{i=1}^\ell (T{-}a_i)$. Considering the $h$-torsion  $\phi[h]$ as defined in Section \ref{ss:torsion}, we have $\phi[h]\simeq \oplus_{i=1}^\ell \phi[T{-}a_i]$. We formulate the following hypothesis:

\begin{equation}\label{hyp}\tag{\textbf{H}}
\forall i\in\{1,\dots,\ell\}, \quad \phi[T{-}a_i]\subseteq \Fqm.
\end{equation}

By \cref{prop:torsion}, we know that $\dim_{\Fq} \phi[T{-}a_i]=r$. 

\begin{definition}\label{def:enc}
    Consider the encoding map
    \Function
      {\enc}
      {\mathcal M}
      {\Hom_{\Fq}\big(\phi[h],\Fqm\big)=\bigoplus_{i=1}^\ell \Hom_{\Fq}\big(\phi[T{-}a_i],\Fqm\big)\simeq \Fqm^{\ell r}}
      {f}
      {f|_{\phi[h]}}
      where the last isomorphism is given after the choice of bases of $\phi[T{-}a_i]$, for all $i\in\{1,\dots,\ell \}$.  
We define the code $\mathcal{C}(\phi,h)$ as the image $\enc(\mathcal{M})$. 
\end{definition}

The parameters of the locally recoverable codes in the rank metric obtained with Drinfeld modules of rank $1$ are given in \cite[Theorem~3.4]{BDM24preprint}. Here, we largely follow their proof to show, in the following theorem, that our generalised construction with Drinfeld modules of higher rank has the same parameters.
\begin{theorem}
    Let $s,\ell$ be two integers such that $s+1 \leq \ell$. For $t,\delta$ two integers such that $t\geq 1$ and $\delta\geq 2$, consider a Drinfeld module $\phi$ of rank $r=t+\delta-1>0$. Then $\mathcal{C}(\phi,h)$ is a $[ m \ell r,\, m(s{+}1) t,\, \ell r{-}rs{-}t{+}1]$ code with rank-locality $(t,\delta)$. Furthermore, the code is optimal, that is, it attains the Singleton-type bound.
\end{theorem}

\begin{proof}
    By construction, the elements in $\Fqm^{\ell r}$ are matrices of size $m \cdot \ell r$. As for the dimension, note that an Ore polynomial of degree $sr+t-1 < (s+1)r$ cannot be the zero map on the space $\Fqm^{\ell r}$ of $\Fqm$-dimension $\ell r \geq (s+1)r$, hence the encoding map is injective. Since $\mathcal M$  has $\Fqm$-dimension $(s+1)\cdot t$, the code has $\Fq$-dimension $m\cdot (s+1)\cdot t$. Now, observe that $\mathcal M$ is contained in the ring of Ore polynomials of degree at most $sr+t-1$. Therefore, $\mathcal{C}(\phi,h)$ is a subcode of a Gabidulin code of parameters $(\ell r,sr+t)$. Since Gabidulin codes are MRD, we obtain
    \begin{equation}\label{eq:lr-min-dist}
    d_\rk(\mathcal{C}(\phi,h))\geq \ell r-sr-t-1.
    \end{equation} 
    We now want to establish the rank-locality property according to \cref{def:ranklocality}.
    For each $i \in \{1, \ldots, r\}$, let $\Lambda(i)$ be the set of column indices corresponding to the summand $\Hom_{\Fq}\big(\phi[T{-}a_i],\Fqm\big) \simeq \Fqm^r$.
    We have $|\Lambda(i)| = \dim \phi[T{-}a_i] = r = t + \delta - 1$.
    Besides, for any $f\in\mathcal{M}$ and any $x\in \phi[T{-}a_i]$, we have
    \[f|_{\phi[T-a_i]}(x)=\sum_{k=0}^s f_k(a_i^k x)=\sum_{k=0}^s a_i^k f_k(x),\]
    where $f_k(x)\in\Fqm\{\tau\}_{\leq t-1}$. We conclude that $\mathcal{C}(\phi,h)|_{\Lambda(i)}$ is included in a $[r,t]_{q^m}$ Gabidulin code, hence, it has minimum distance at least $r+1-t=\delta$.
    Now, if $\gamma \in \{1, \ldots, \ell r\}$ is a column index, we set $\Gamma(\gamma) := \Lambda(i)$ for the unique index $i$ such that $\gamma \in \Lambda(i)$.
    By what precedes, $\mathcal{C}(\phi,h)|_{\Gamma(\gamma)}$ has minimum distance at least $\delta$, proving that $\mathcal{C}(\phi,h)$ has $(t,\delta)$ rank-locality.
    From the Singleton-type bound of Equation \eqref{eq:lr-sing-bound} we entail that
    \[d_\rk(\mathcal{C}(\phi,h))\leq \ell r - rs-t+1,\]
    which, together with Equation~\eqref{eq:lr-min-dist}, finally gives $d_\rk(\mathcal{C}(\phi,h))= \ell r - rs-t+1$.
\end{proof}

A key ingredient in the construction of locally recoverable codes in any metric is the choice of a \emph{good} polynomial respecting some conditions which allow the recovery of the lost pieces of data (see, for instance, \cite[Section A]{TamoBarg14}). The present construction of LR codes in the rank metric makes no exception. 
Here, the evaluation points of \cref{def:enc} are carefully chosen by selecting a polynomial $h\in\Fq[T]$ so that the $h$-torsion is defined over $\Fqm$, to satisfy \eqref{hyp}. For Drinfeld modules of rank $1$, the existence of such a polynomial $h$ is guaranteed using an effective version of Dirichlet's Theorem \cite[\S~5]{BDM24preprint}. It would be interesting to study under which conditions such a polynomial $h$ exists when generalising the construction to Drinfeld modules of any rank. This question could be approached by developing an explicit Chebotarev density theorem for Drinfeld modules, using the theory of $L$-series presented in~\cref{sssec:arith:Lseries} as a main tool.

    \section*{Acknowledgements}

    This presentation is inspired by the talks given in the second edition of the \textit{Coding theory, cryptogrAphy, arIthmetic geometry, and comPuter Algebra Itinerant} symposium\footnote{CAIPI:\, \url{https://caipi_symposium.pages.math.cnrs.fr/page-web/index-en.html}}.
    The authors are very grateful to the sponsors of this edition, namely the Mathematics Institute of Marseille (I2M – UMR7373) and the FRUMAM research federation (FR 2291 - CNRS).
    CA was supported by ANR {\it Jinvariant} (ANR-20-CE40-0003) and ANR \emph{PadLEfAn} (ANR-22-CE40-0013).
    EB was supported by the grant ANR-22-CPJ2-0047-01.
    XC and AL are grateful to the ANR \emph{PadLEfAn} (ANR-22-CE40-0013) for supporting several ``research in pairs'' devoted to code development.
    AL gratefully acknowledges that this research was supported in part by the Pacific Institute for the Mathematical Sciences.
    JN is supported by the French government \textit{Investissements d’Avenir} program ANR-11-LABX-0020-01.
    FP was supported by ANR {\it Jinvariant} (ANR-20-CE40-0003), and is grateful to the University of Bordeaux for the hospitality.
    This work was partially funded by the grant ANR-21-CE39-0009-BARRACUDA. 

	\bibliography{biblioMD}
	\bibliographystyle{alpha}
	
\end{document}